\documentclass[reqno,12pt,letterpaper]{amsart}
\usepackage{amsmath,amssymb,amsthm,graphicx,mathrsfs,url}
\usepackage[usenames,dvipsnames]{color}
\usepackage[colorlinks=true,linkcolor=Red,citecolor=Green]{hyperref}

\def\?[#1]{\textbf{[#1]}\marginpar{\Large{\textbf{??}}}}

\newcommand{\stopthm}{\hfill$\square$\medskip}

\newtheorem{theo}{Theorem}
\newtheorem{prop}{Proposition}[section]
\newtheorem{defi}[prop]{Definition}

\newtheorem{lemm}[prop]{Lemma}
\newtheorem{corr}[prop]{Corollary}
\newtheorem{rem}{Remark}
\numberwithin{equation}{section}

\newcommand{\mc}{\mathcal}
\newcommand{\rr}{\mathbb{R}}
\newcommand{\nn}{\mathbb{N}}
\newcommand{\cc}{\mathbb{C}}
\newcommand{\hh}{\mathbb{H}}

\newcommand{\la}{\lambda}
  
\newcommand{\eps}{\epsilon}
\newcommand{\de}{\delta}

\newcommand{\pl}{\partial}
\newcommand{\x}{\times}
\newcommand{\om}{\omega}
\newcommand{\til}{\widetilde}
\newcommand{\bbar}{\overline}

\newcommand{\cjd}{\rangle}
\newcommand{\cjg}{\langle}

\newcommand{\demi}{\tfrac{1}{2}}

\newcommand{\eb}{\mathbf{e}}

\DeclareMathOperator{\grad}{grad}

\DeclareMathOperator{\supp}{supp}

\title[X-Ray Transform and Boundary Rigidity]{X-Ray Transform and Boundary
  Rigidity for Asymptotically Hyperbolic Manifolds}  
\author[Graham, Guillarmou, Stefanov, and Uhlmann]{C. Robin Graham} 
\address{Department of Mathematics, University of Washington,
Box 354350\\
Seattle, WA 98195-4350, USA}
\email{robin@math.washington.edu}

\author[]{Colin Guillarmou}
\address{Laboratoire de Math\'ematiques d'Orsay, UMR 8628 du CNRS,
Universit\'e Paris-Sud, 91405 Orsay Cedex, France} 
\email{colin.guillarmou@math.u-psud.fr}

\author[]{Plamen Stefanov}
\address{Department of Mathematics, Purdue University, West Lafayette, IN 47907, USA}
\email{stefanov@math.purdue.edu}

\author[]{Gunther Uhlmann}
\address{Department of Mathematics, University of Washington, Seattle,
WA 98195-4350, USA, Department of Mathematics and Statistics, University of
Helsinki, Finland, and Institute for Advanced Study of the Hong Kong
University of Science and Technology}
\email{gunther@math.washington.edu}

\begin{document}

\begin{abstract}
We consider the boundary rigidity problem for asymptotically hyperbolic
manifolds. We show injectivity of the X-ray transform in several cases and 
consider the non-linear inverse problem which consists of recovering a 
metric from boundary measurements for the geodesic flow.
\end{abstract}

\maketitle

\thispagestyle{empty}

\section{Introduction}
In this work, we consider the problem of the geodesic X-ray transform on 
asymptotically hyperbolic manifolds, and some applications to the boundary
rigidity problem in that non-compact setting.
 
Let $\bbar{M}$ be a compact connected smooth manifold-with-boundary of
dimension $n+1$ with $n\geq 1$.  A smooth metric $g$ on the interior $M$ of 
$\bbar{M}$ is 
said to be \emph{asymptotically hyperbolic} if $\bbar{g}_0:=\rho_0^2g$  
extends to a smooth metric on $\bbar{M}$ with $|d\rho_0|_{\bbar{g}_0}=1$ at  
$\pl \bbar{M}$, where $\rho_0\in C^\infty(\bbar{M};\rr_{\geq 0})$ is a
smooth defining function for $\pl\bbar{M}$, i.e. $\{\rho_0=0\}=\pl\bbar{M}$
with $d\rho_0$ not vanishing at $\pl \bbar{M}$.   
The boundary $\pl \bbar{M}$ equipped with the conformal
class of $\bbar{g}_0|_{T\pl \bbar{M}}$ is called the   
\emph{conformal boundary}, or \emph{conformal infinity}, of $(M,g)$.  It
follows from \cite{GL} that for each metric $h$   
in the conformal infinity, there exists a smooth 
boundary defining function $\rho$ so that $|d\rho|_{\rho^2g}=1$ near $\pl
\bbar{M}$ and $\rho^2g|_{T\pl\bbar{M}}=h$; this function is uniquely 
determined near $\pl \bbar{M}$ by $h$.  Such a function $\rho$ is called
a \emph{geodesic boundary defining function} associated to the conformal
representative $h$.  The flow of the gradient of $\rho$ with respect to the 
metric $\bbar{g}:=\rho^2g$ induces a product decomposition $(0,\eps)_\rho\x
\pl\bbar{M}$ of a collar neighborhood $\mc{C}_\eps$ near $\pl \bbar{M}$
in which the metric has the form  
\[ 
g= \frac{d\rho^2+h_{\rho}}{\rho^2} \quad\textrm{ on } (0,\eps)_\rho\x
\pl\bbar{M},
\]
with $h_{\rho}$ a smooth 1-parameter family of metrics on $\pl\bbar{M}$
which extends smoothly to $\rho\in[0,\eps)$ and satisfies $h_0=h$. For  
convenience, we can extend freely $\rho$ as a smooth positive function to 
$M$ so that $\rho\geq \eps$ in 
$M\setminus \mc{C}_\eps$. The metric $g$ is a complete metric with
sectional curvatures tending to $-1$ at $\pl \bbar{M}$; it has infinite
volume and all convex co-compact hyperbolic manifolds are particular cases
of asymptotically hyperbolic manifolds. 

Geodesics of $g$ can be viewed as integral curves of the Hamiltonian vector
field $X$ of $|\xi|^2_g/2$ on the unit cotangent bundle $S^*M:=\{
(x,\xi)\in T^*M; |\xi|_g=1\}$ of $M$, projected to $M$ by $\pi:S^*M\to M$
the projection on the base.  Geodesics approach $\pl\bbar{M}$ normally and
are determined by their second order deviation from the normal.  In
order to encode this, we introduce an extension   
$\bbar{S^*M}$ of $S^*M$ to $\bbar{M}$.  Recall, from e.g. \cite{Me}, that   
the $b$-cotangent 
bundle ${}^bT^*\bbar{M}$ is a smooth vector bundle on $\bbar{M}$
isomorphic to $T^*M$ over $M$ and with local smooth sections 
$\{d\rho/\rho,dy_1,\dots,dy_n\}$  
near $\pl\bbar{M}$, if $(\rho,y_1,\dots,y_n)$ are local coordinates near $\pl\bbar{M}$. 
The dual metric to $g$, viewed as a
metric on ${}^bT^*\bbar{M}|_M$, extends smoothly to $\bbar{M}$ but
degenerates over $\pl\bbar{M}$.  The extension $\bbar{S^*M}$ is defined to
be the unit cosphere bundle in ${}^bT^*\bbar{M}|_M$ with respect to the 
quadratic form $g$.  It takes the form  
\[
\bbar{S^*M}= S^*M\sqcup \pl_-S^*M\sqcup \pl_+S^*M, 
\]
where each of $\pl_\pm S^*M$ is a canonical subset of 
${}^bT^*\bbar{M}|_{\pl M}$ independent of $g$ which can be identified with  
$T^*\pl\bbar{M}$ upon choosing a metric $h$ in the conformal infinity of
$g$.  Elements of $\pl_\pm S^*M$ correspond by duality using $g$ to 
second order tangential deviations from the normal at a boundary point. 
$\pl_-S^*M$ is regarded as  
the incoming boundary and $\pl_+S^*M$ as the outgoing boundary.   
Each point $z_\pm\in \pl_\pm S^*M$ is the limit of a unique integral curve
of $X$ as as $t\to \pm\infty$.  

The \emph{trapped set} $K$ of the flow
$\varphi_t:S^*M\to S^*M$ of $X$ is the set of points $z\in S^*M$ for which
the integral curve $\{\varphi_t(z); t\in \rr\}$ remains in a compact set; 
since the regions $\{\rho\geq \eps\}$ are strictly convex for small $\eps$,
this can alternatively be defined by (here $\rho$ is lifted to $S^*M$ by
$\pi$)  
\[ 
K:=\{ z\in S^*M;  \inf_{t\in\rr} \rho(\varphi_t(z))>0\}.
\]
This is a compact set that is globally invariant by $\varphi_t$. We say
that $g$ is \emph{non-trapping} if $K=\emptyset$. In this work, we will
consider either the non-trapping case or the case where $K$ is a
\emph{hyperbolic set} for the flow in the following sense: there is a
continuous, flow-invariant splitting of $T_K(S^*M):=T(S^*M)|_K$ of the
form   
\[ 
T_K(S^*M)=\rr X\oplus E_s\oplus E_u 
\] 
where $E_s,E_u$ are subbundles over $K$ satisfying that there is  $\nu>0$
and $C>0$ such that for all $z=(x,\xi)\in K$ and  
\begin{equation}\label{hypofK}
\begin{gathered}
\forall \zeta\in E_s(z), \quad ||d\varphi_t(x).\zeta||_{G}\leq Ce^{-\nu
  t}||\zeta||_{G}, \quad\forall\, t\geq 0,\\
\forall \zeta \in E_u(z), \quad ||d\varphi_t(x).\zeta||_{G}\leq Ce^{-\nu
  |t|}||\zeta||_{G}, \quad \forall\, t\leq 0,
\end{gathered}
\end{equation}
(Here $G$ denotes the Sasaki metric for $g$, see \eqref{sasaki}).     
The incoming and outgoing trapped sets are defined by   
\[\Gamma_\pm:=\{ z\in S^*M;  \inf_{t\in\rr^+} \rho(\varphi_{\mp t}(z))>0\},\]
they correspond to geodesics trapped in the past (+) or in the future
(-).  When $K$ is hyperbolic, then $\Gamma_\pm$ and $K$ have zero Liouville 
measure.   
Each untrapped geodesic $\gamma(t)$ of $g$ converges to a point $y_-\in\pl
\bbar{M}$ in the past and $y_+\in \pl \bbar{M}$ in the future,
and the corresponding integral curve on $S^*M$ converges to some $z_-\in
\pl_-S^*M$ in the past and $z_+\in \pl_+S^*M$ in the future. 
The set of untrapped geodesics is parametrized by 
$\pl_-S^*M\setminus \bbar{\Gamma_-}$,
corresponding to the backward limit of the integral curve.  
In the non-trapping case, $\Gamma_\pm$ are empty.  

Our first result concerns the X-ray transform on symmetric $m$-tensors, 
which can be defined as the operator  
\[\begin{gathered}
I_m: C_c^\infty(M; \otimes^m_ST^*M)\to C^\infty(\pl_-S^*M\setminus
\bbar{\Gamma_-}), \quad  
I_m(f)(z)=\int_{\rr} f(\gamma_{z}(t))(\otimes^m\dot{\gamma}_z(t))dt
\end{gathered}\] 
where $\gamma_z(t)$ is the geodesic with backward limit $z\in
\pl_-S^*M$ (here $\otimes_S^mT^*M$ denotes the bundle of  
symmetric tensors of rank $m$ on $M$). This operator extends 
continuously to the space
$\rho^{1-m}C^\infty(\bbar{M};\otimes_S^mT^*\bbar{M})$.  
\begin{theo}\label{injXray}
Let $(M,g)$ be an asymptotically hyperbolic manifold such that $g$ has no
conjugate points and the trapped set is either empty or a hyperbolic
set. 
Let $f\in \rho^{1-m} C^\infty(\bbar{M};\otimes^m_S T^*\bbar{M})$ satisfy
$I_mf=0$.
\begin{itemize}
\item[(i)] If $m=0$ then $f=0$.
\item[(ii)] If $m=1$, there exists $q\in \rho C^\infty(\bbar{M})$ such that
  $f=dq$.
\item[(iii)] If $m>1$ and if the curvature of $g$ is non-positive, then there
exists a symmetric tensor $q\in \rho^{2-m}C^\infty(\bbar{M},
\otimes^{m-1}_S T^*\bbar{M})$ such that $f=Dq$, where $D$ denotes the
symmetrized covariant derivative.   
\end{itemize}
\end{theo}

\noindent
\textbf{Corollary 1.} {\it
Let $(M,g)$ be an asymptotically hyperbolic manifold with negative   
curvature and $m\geq 0$.  If $f\in \rho^{1-m}
C^\infty(\bbar{M};\otimes^m_S T^*\bbar{M})$ satisfies 
$I_mf=0$, then there exists 
$q\in \rho^{2-m}C^\infty(\bbar{M}, \otimes^{m-1}_S T^*\bbar{M})$ 
such that $f=Dq$.  (In particular, if $m=0$, then $f=0$.)}

\noindent
Corollary 1 follows from Theorem~\ref{injXray} since 
manifolds with negative curvature cannot have conjugate points and the 
trapped set, if nonempty, is hyperbolic.  

The X-ray transform for functions was studied on the hyperbolic space
$\hh^{n+1}$ by Helgason and Berenstein-Casadio Tarabusi: injectivity is
proved in \cite{He2} for functions decaying like $e^{-d_g(\cdot,o)}$ for
$o\in \hh^{n+1}$ fixed  
(this corresponds exactly to the decay condition in Theorem \ref{injXray}), 
and an inversion formula is given in \cite{He1,BeCa}. For Cartan-Hadamard
manifolds, recent work by Lehtonen \cite{Le} shows injectivity of $I_0$ in
dimension $2$ and then Lehtonen-Railo-Salo \cite{LRS} extended the
result to higher dimensions and tensors.  In comparison to \cite{Le,LRS},
we allow hyperbolic 
trapping, we do not require $M$ to be simply connected, and for 
$m\in \{0,1\}$ we allow some positive curvature, but      
our assumption about the geometry at infinity is stronger.  We 
have not tried to obtain the sharpest regularity assumptions on $f$ and it
can easily be seen from the proof that the regularity assumptions can be
relaxed -- we refer to \cite{LRS} for sharper regularity conditions. 

The study of the geodesic X-ray transform on compact domains has a long 
history.  \emph{Simple metrics} are metrics on domains with strictly convex 
boundaries for which the exponential map is a diffeomorphism at each
point.  Injectivity of the X-ray transform  goes back to Mukhometov 
\cite{Mu} for functions, then to Anikonov-Romanov \cite{AnRo} for 1-forms, while  
Pestov-Sharafutdinov \cite{PeSh} proved injectivity for all tensors in
negative curvature (see also Paternain-Salo-Uhlmann \cite{PSU2} 
for more general results on tensors). Similar results for tensors of rank 
$m\leq 2$ was shown for analytic simple metrics and for generic simple
metrics by Stefanov-Uhlmann \cite{StUh1}.  
For simple metrics in dimension $2$, injectivity for $2$-tensors was first
shown by Sharafutdinov \cite{Sh2} and has been proved recently by
Paternain-Salo-Uhlmann \cite{PSU1} for  tensor fields of all ranks. For 
manifolds with strictly convex foliations, injectivity is 
shown in Uhlmann-Vasy \cite{UhVa} for functions and in
Stefanov-Uhlmann-Vasy \cite{SUV2}  for $2$-tensors.  Injectivity for all
tensors for all metrics 
with negative curvature and strictly convex boundary is proved in
Guillarmou \cite{Gu}, without simplicity assumptions.  Microlocal analysis 
of the X-ray transform for some cases with conjugate points was done in
\cite{StUh3, StUh4, MSU, HoUh} with generic uniqueness and stability
results for a certain class of non-simple metrics in \cite{StUh3}.  

To prove Theorem \ref{injXray}, we need to do a careful analysis of the
geodesic flow near infinity.  We show that the X-ray transform determines
the function (or tensor modulo $Dq$ terms) up to $\mc{O}(\rho^\infty)$ at
the boundary by using the  
``short geodesics'', i.e. those geodesics staying in regions   
$\{\rho\leq \eps\}$ for small $\eps>0$. We then conclude by using Pestov
identities on large regions $\{\rho\geq \eps\}$, with $\eps\to 0$.  We also
use the results of \cite{Gu} to deal with the trapped case. 
We observe that our assumptions in (i) and (ii) of Theorem~\ref{injXray}  
allow conjugate points at infinity, in the sense that there could be Jacobi 
fields vanishing at the endpoints $y_-,y_+$ at infinity along a non-trapped
geodesic.

The \emph{boundary rigidity} problem for simple metrics on compact domains
asks if one can recover a simple metric from its boundary distance function
(the set of distances between boundary points).  Many results are known on
the boundary and lens rigidity problems in the compact setting, we refer to
the surveys \cite{Sh1,Cr,Iv,StUh2} and to the introduction of \cite{SUV3}
for references.   Here, we consider an analogue of the boundary rigidity
problem for asymptotically hyperbolic metrics.  First, for each $z_-\in
\pl_-S^*M$, there is a unique geodesic $\gamma_{z_-}$ with backward limit
$z_-$. If $\gamma_{z_-}$ is not trapped in the future, we denote its
forward limit by $z_+\in \pl_+S^*M$.  Thus we can define a map   
\[ 
S_g: \pl_-S^*M\setminus \bbar{\Gamma_-}\to \pl_+S^*M\setminus 
\bbar{\Gamma_+}, \quad S_g(z_-)=z_+\] called the \emph{scattering map} for
the geodesic flow. It is a symplectic map  
with respect to the canonical symplectic structures on $\pl_\pm S^*M$
induced by their identifications with $T^*\pl\bbar{M}$.  
For such a geodesic $\gamma_{z_-}$, given a defining function $\rho$, we
define the \emph{renormalized length} relative to $\rho$ by  
\[
 L_g(z_-):= \lim_{\eps\to 0}\big(\ell_g( \gamma_{z_-}\cap\{\rho\geq
 \eps\})+2\log \eps\big)
\]
where $\ell_g$ denotes the length for the metric $g$.
We show that $L_g$ is a well-defined function on $\pl_-S^*M\setminus 
\bbar{\Gamma_-}$ which  
depends on the choice of $\rho$ in a simple explicit fashion (see 
\eqref{dependLg}).  We may also view $L_g$ as determined by a choice of
representative metric $h$ in the conformal infinity by taking $\rho$ to be
the corresponding geodesic defining function.  The functions $L_g$ and  
$S_g$ are closely related to the sojourn time and scattering relation
appearing in Sa Barreto-Wang \cite{SaWa2}.   
Renormalized volumes, areas
and lengths already appeared quite naturally when analyzing the geometry of    
asymptotically hyperbolic Einstein manifolds and in the AdS/CFT 
correspondence (see for example \cite{Gra,AlMa}). 
Boundary rigidity and integral geometry appear in the physics literature
concerning the AdS/CFT duality and holography as well, 
see \cite{PoRa, CLMS}.  

We first show that the renormalized length data determine the metric to
infinite order at the boundary.
\begin{theo}\label{determination}
Let $\bbar{M}$ be a compact connected manifold-with-boundary and let $g,g'$
be two asymptotically hyperbolic metrics on $M$. Suppose for some choices
$h$ and $h'$ of conformal representatives in the conformal infinities of
$g$ and $g'$, the renormalized lengths agree for the two metrics,
i.e. $L_g=L_{g'}$.  Then there exists a diffeomorphism    
$\psi:\bbar{M}\to \bbar{M}$ which is the identity on $\pl\bbar{M}$ and such 
that $\psi^*g'-g=\mc{O}(\rho^\infty)$ at $\pl\bbar{M}$. 
\end{theo}

As a consequence of Theorem~\ref{determination}, we deduce boundary
rigidity for real-analytic metrics under a topological hypothesis.  If
$\bbar{M}$ is a real-analytic 
manifold-with-boundary, we say that a metric $g$ on $M$ is a real-analytic 
asymptotically hyperbolic metric if $g$ is real-analytic, asymptotically
hyperbolic, and $\bbar{g}=\rho^2 g$ is real-analytic up to $\pl\bbar{M}$,
where $\rho$ is a real-analytic defining function for $\pl\bbar{M}$.   

\begin{theo}\label{realanalytic}
Let $\bbar{M}$ be a compact connected real-analytic
manifold-with-boundary such that 
$\pi_1(\bbar{M},\pl\bbar{M})=0$.  Let $g,g'$ be two real-analytic 
asymptotically hyperbolic metrics on $M$. If $L_g=L_{g'}$ for some
real-analytic metrics $h$ and $h'$ in the
conformal infinities of $g$ and $g'$, then there exists a real-analytic 
diffeomorphism $\psi:\bbar{M}\to \bbar{M}$ which is the identity on
$\pl\bbar{M}$ and such that $\psi^*g'=g$.  
\end{theo}

In Theorems~\ref{determination} and \ref{realanalytic}, we only require
that $L_g=L_{g'}$ on a neighborhood of infinity in $\pl_-S^*M$, 
corresponding to short geodesics.  
We will show that $\pl_-S^*M\cap\bbar{\Gamma_-}$ is a compact set, so the
domain of $L_g$ always contains a full neighborhood of infinity.     

In the case of compact simple metrics, 
the determination of the metric and the curvature at $\pl \bbar{M}$ was
proved by Michel \cite{Mi}, and the result corresponding to Theorem 
\ref{determination} was shown by 
Lassas-Sharafutdinov-Uhlmann \cite{LSU} (see also Stefanov-Uhlmann
\cite{StUh3} for  non-simple metrics).  
A rigidity result for real-analytic metrics on compact
manifolds-with-boundary is proved in \cite{Va}.     

Finally, we prove a deformation rigidity result for the boundary rigidity 
problem.  We define a non-trapping asymptotically hyperbolic manifold to be 
\emph{simple} if it has no conjugate points
at infinity in the sense stated above.  This holds in particular 
for non-trapping metrics when the sectional curvature is non-positive.
This definition of simple in the asymptotically hyperbolic case is a weaker
starting point than the requirement 
that the exponential map be a diffeomorphism in the compact case.  We show
that if $(M,g)$ is simple, then its geodesic flow is hyperbolic.
This together with a recent result of Knieper \cite{Kn} giving sufficient    
conditions for no conjugate points on a complete non-compact Riemannian
manifold with hyperbolic geodesic flow imply that a simple asymptotically
hyperbolic manifold has no 
conjugate points.  Other consequences of the hyperbolicity of the geodesic
flow are that for each pair of points $y_-\not=y_+\in \pl\bbar{M}$,   
there is a unique geodesic with endpoints $y_\pm$, and that the   
exponential map extends smoothly to the boundary as a diffeomorphism in an
appropriate sense (Propositions~\ref{uniquegeodesics} and \ref{Ediff}).   
The fact that there is a unique geodesic joining any two boundary points
enables us to define the renormalized boundary distance relative  
to a defining function $\rho$ by
\[ 
d_g^R: \pl\bbar{M}\x\pl\bbar{M}\setminus{\rm diag}\to \rr, \quad 
d_g^R(y_-,y_+):=L_g(z_-)
\]
where $z_-\in \pl_-S^*_{y_-}M$ is defined by the equation
$S_g(z_-)=z_+$ for some $z_+\in \pl_+S^*_{y_+}M$.  
\begin{theo}\label{deformation}
Let $\bbar{M}$ be a compact connected manifold-with-boundary and suppose
that for  
$s\in[0,1]$, $g(s)$ is a smooth family of non-trapping asymptotically  
hyperbolic metrics with non-positive sectional curvature.  Assume that for
some smooth family $h(s)$ of representatives of the conformal
infinities of $g(s)$, one of the following two conditions holds: 
\begin{itemize}
\item[(i)] The renormalized length functions $L_{g(s)}$ and scattering maps
$S_{g(s)}$ are constant in $s$.  (Here $S_{g(s)}$ is a viewed as a map
  $:T^*\pl\bbar{M}\to T^*\pl\bbar{M}$ via the identifications induced by
  $h(s)$).  

\item[(ii)]
The renormalized boundary distance functions $d_{g(s)}^R$ are constant in
$s$.  
\end{itemize}

\noindent
Then there is a
smooth family of diffeomorphisms $\psi(s):\bbar{M}\to \bbar{M}$ for $s\in
[0,1]$ which satisfies $\psi(s)^*g(s)=g(0)$ and 
$\psi(s)|_{\pl \bbar{M}}={\rm Id}$.    
\end{theo}

To prove Theorem~\ref{deformation} under hypothesis (i), we first use 
Theorem~\ref{determination} to arrange that the metrics agree to infinite
order at the boundary.  Then we use Theorem \ref{injXray} after proving
that the linearization of the pair $(L_{g(s)},S_{g(s)})$ reduces to the 
description of the kernel of the X-ray transform on symmetric $2$-tensors.   
We reduce (ii) to (i) by showing in
Proposition~\ref{equivalencedglens} that if two simple metrics 
have the same renormalized boundary distance functions, then they have the
same scattering maps and renormalized length functions.  As one of the
steps in doing this, we 
show in Proposition~\ref{dtsmooth} that if $(M,g)$ is simple, then
$d_g(p,q)+\log \rho(p)+\log\rho(q)$ extends smoothly to 
$\bbar{M}\times\bbar{M}\setminus {\rm diag}$, where $d_g(p,q)$ denotes the
distance function in the metric $g$.  

\noindent
\textbf{Acknowledgements.} Research of C. R. Graham is partially supported
by NSF grant \# DMS 1308266.  C. Guillarmou is supported by 
ANR-13-BS01-0007-01 and ANR-13-JS01-0006, and ERC consolidator grant
IPFLOW.  P. Stefanov is partially supported by NSF grant \# DMS-1600327. 
G. Uhlmann was partly supported by NSF grant \# DMS-1265958, a
Si-Yuan Professorship at the Institute for Advanced Study of the Hong Kong
University of Science and Technology, and a FiDiPro at the University of
Helsinki, Finland.  We finally thank Gabriel Paternain, Yiran Wang and
Nikolas Eptaminitakis for useful discussions, and Gerhard Knieper for
sharing with us his result \cite{Kn}.  

\section{Asymptotically hyperbolic manifolds and their geodesic flow}

For an asymptotically hyperbolic manifold $(M,g)$, the sectional curvatures
of $g$  tend to $-1$ uniformly at the boundary, and more precisely the
curvature tensor $\mc{R}_g$ of $g$ is of the form 
\begin{equation}\label{curvature} 
\mc{R}_g= -g\circ g +\rho^{-3}Q
\end{equation}
where $\circ$ denotes the Kulkarni-Nomizu product and $Q\in
C^\infty(\bbar{M}; \otimes^4T^*\bbar{M})$.    

By abuse of notation, we will sometimes write also $g$ and $h_\rho$ for the
metrics induced by $g$ and $h_\rho$ on the cotangent bundles $T^*M$ and  
$T^*\pl\bbar{M}$. 

\subsection{The geodesic flow on the unit cotangent bundle}
Let us describe the geodesic flow near the boundary $\pl\bbar{M}$. We work
on the unit cotangent bundle 
\[
S^*M=\{ (x,\xi)\in T^*M; |\xi|^2_g=1\}
\]
and denote by $\pi:S^*M\to M$ the projection to the base.  The 
Liouville $1$-form is denoted $\alpha$ and the Hamilton vector field of
$\demi|\xi|^2_g$ is the generator $X$ of the geodesic flow $\varphi_t$ on
$S^*M$.  If $(\rho,y^1,\dots,y^n)$ are local coordinates near $\pl
\bbar{M}$ with $\rho$ a defining function for $\pl \bbar{M}$, we use 
dual coordinates $(\xi_0,\eta_1,\dots,\eta_n)$ on $T^*\bbar{M}$ so that 
$\xi=\xi_0 d\rho+ \eta.dy$.  We begin by 
considering an extension of the $g$-cosphere bundle $S^*M$ to $\bbar{M}$.     

Recall that ${}^bT^*\bbar{M}$ denotes the $b$-cotangent bundle of
$\bbar{M}$, a smooth vector bundle on $\bbar{M}$.  
A basis for its fibers near $\pl\bbar{M}$ consists of 
$\{\rho^{-1}d\rho,dy^1,\ldots,dy^n\}$, and we use dual coordinates 
$(\bbar{\xi}_0, \eta)$ so that 
$\xi = \bbar{\xi}_0\rho^{-1}d\rho + \sum_i\eta_idy^i$.  It is easily
verified that the function 
$\xi\mapsto \bbar{\xi_0}$ is an invariant on
${}^bT^*\bbar{M}|_{\pl\bbar{M}}$, i.e. it is independent of the choice of
coordinates $(\rho,y)$.  In particular, the subsets $\{\bbar{\xi}_0=\pm
1\}$ of ${}^bT^*\bbar{M}|_{\pl\bbar{M}}$ are invariantly defined
independently of any choices.    

Observe that $g$ defines a smooth
quadratic form on ${}^bT^*\bbar{M}$ all the way up to $\pl\bbar{M}$,
which however degenerates on $\pl\bbar{M}$.  We denote by 
$\bbar{S^*M}=\{(x,\xi)\in {}^bT^*\bbar{M}:|\xi|_g=1\}$ the  
unit cosphere bundle in ${}^bT^*\bbar{M}$ with respect to $g$.  
Choose a representative metric $h$ for the conformal structure at 
infinity and use the induced product decomposition near
$\pl\bbar{M}$.  If $y=(y^1,\dots,y^n)$ are local coordinates on
$\pl\bbar{M}$, we obtain coordinates $(\rho,y)$ on $\bbar{M}$.
We have for $x\in \bbar{M}$ near $\pl\bbar{M}$:
\[
\bbar{S^*_xM}=\{(x,\xi): \bbar{\xi}_0{}^2 +\rho^2|\eta|^2_{h_\rho}=1\}.   
\]
It follows that $\bbar{S^*M}$ is a smooth non-compact
submanifold-with-boundary of ${}^bT^*\bbar{M}$ which is naturally
identified with $S^*M$ over $M$.  We define 
$\pl_{\pm} S^*M=\{(x,\xi); x\in \pl\bbar{M}, \bbar{\xi_0}=\mp 1$\}; as
noted above, these subsets of ${}^bT^*\bbar{M}|_{\pl\bbar{M}}$ are
independent of the choice of $g$ and of the local coordinates.  
Thus we have    
\[
\bbar{S^*M} = S^*M \sqcup \pl_-S^*M\sqcup \pl_+S^*M.
\]
Given $g$ and the choice of
conformal representative metric $h$, we can identify each of 
$\pl_{\pm} S_x^*M$ with $T^*_x\pl\bbar{M}$ via the identifications  
\begin{equation}\label{ident}
\mp \rho^{-1} d\rho + \sum_i\eta_idy^i\mapsto \sum_i\eta_idy^i. 
\end{equation}  
We view   
$\pl_- S^*M$ as the incoming boundary and  
$\pl_+ S^*M$ as the outgoing boundary.  
We will denote by $\iota_\pl: \pl_-S^*M\cup \pl_+S^*M\to \bbar{S^*M}$ the
smooth inclusion map.  The projection $\pi:S^*M\to M$ extends as a smooth map 
\[ \pi: \bbar{S^*M}\to \bbar{M}.\]  We define the vertical bundle
$\mathcal{V}=\ker d\pi$, a smooth subbundle of $T\bbar{S^*M}$ of rank 
$n$.   

Since $g$ degenerates at $\pl\bbar{M}$ as a metric on ${}^bT^*\bbar{M}$, it
does not induce an isomorphism between ${}^bT^*\bbar{M}$ and 
${}^bT\bbar{M}$ over $\pl\bbar{M}$.  Instead, it induces an isomorphism
globally with another natural bundle extending the tangent bundle.  Suppose
$\bbar{M}$ is a manifold-with-boundary equipped with a line 
subbundle $\mc{L}\subset T\bbar{M}|_{\pl{\bbar{M}}}$ which is transverse to  
$T\pl\bbar{M}$.  In our setting, $\mc{L}$ is the orthogonal complement 
to $T\pl\bbar{M}$ with respect to $\bbar{g}=\rho^2g$.  Consider the space
of smooth vector fields ${}^{\mc{L}}\mc{V}$ on $\bbar{M}$ defined by 
$$
{}^{\mc{L}}\mc{V}=\{V\in C^\infty(\bbar{M};T\bbar{M}):V|_{\pl\bbar{M}}=0 \text{ and }
(\rho^{-1}V)(x)\in \mc{L}_x,\,\, x\in \pl\bbar{M}\}.
$$
In the usual way, ${}^{\mc{L}}\mc{V}$ can be regarded as the space of
smooth sections of a smooth vector bundle ${}^{\mc{L}}T\bbar{M}$ on
$\bbar{M}$.  If $(\rho, y=y^1,\dots,y^n)$ are any local coordinates near a point of 
$\pl\bbar{M}$ so that $\mc{L}=\operatorname{span}\{\pl_\rho\}$, then  
$\{\rho\pl_\rho,\rho^2\pl_{y^1},\ldots,\rho^2\pl_{y^n}\}$ is a basis for 
${}^{\mc{L}}T_x\bbar{M}$ for any $x\in\bbar{M}$ near $\pl\bbar{M}$.  For an
asymptotically hyperbolic metric $g$ in normal form, the induced isomorphism $T^*M\rightarrow TM$
maps 
$\bbar{\xi}_0\rho^{-1}d\rho+\sum_{i}\eta_idy^i\mapsto \bbar{\xi}_0\rho\pl_\rho + 
\sum_{ij} h_\rho^{ij}\eta_i\rho^2\pl_{y^j} $, where $(h^{ij}_\rho)$ 
denotes the matrix of the metric  induced by $h_\rho$ on $T^*\pl \bbar{M}$
in the coordinates $\eta_i$. 
Clearly this isomorphism extends to the boundary as a smooth isomorphism of
vector bundles ${}^bT^*\bbar{M}\rightarrow {}^{\mc{L}}T\bbar{M}$ which
pulls back the degenerate metric induced by $g$ on ${}^{\mc{L}}T\bbar{M}$
to that on ${}^bT^*\bbar{M}$.  The bundle
${}^bT^*\bbar{M}\cong {}^{\mc{L}}T\bbar{M}$ is a natural extension of the 
(co)tangent bundle 
for the study of geodesics of an AH metric.  For instance, the tangent
vector field of a geodesic $\gamma$ is a smooth nonvanishing section of
${}^{\mc{L}}T\bbar{M}|_{\gamma}$ all the way up to the boundary, and, as we 
will see, the geodesics emanating from or ending on a boundary
point $x$ are parametrized by the fibers of $\pl_\mp S^*_xM$.  
As a comparison, recall that the $0$-cotangent bundle is the smooth bundle 
${^0}T^*\bbar{M}$ over $\bbar{M}$ whose fibers near the boundary have basis
$\{\frac{d\rho}{\rho},\frac{dy^i}{\rho}\}$. The $0$-unit cotangent bundle
is ${^0S}^*\bbar{M}:=\{(x,\xi)\in {^0}T^*\bbar{M}; |\xi|_{g}=1\}$; this is a 
compact manifold-with-boundary.  The bundle ${^0}T^*\bbar{M}$ is the
natural bundle for analysis of differential operators defined in terms of
an asymptotically hyperbolic metric (see \cite{MaMe}); we will use it only
mildly in \S\ref{Xraysection}.

\begin{lemm}\label{Xform}
The Hamiltonian vector field $X$ on $S^*M$ has the form $X=\rho\bbar{X}$, 
where $\bbar{X}$ is a smooth vector field on $\bbar{S^*M}$ which is
transverse to the boundary $\pl \bbar{S^*M}=\pl_-S^*M\sqcup \pl_+S^*M$. 
\end{lemm}
\begin{proof}
As a vector field on $T^*M$, we know that $X$ is tangent to $S^*M$, so it
suffices to analyze $X$ in coordinates on ${}^bT^*\bbar{M}$.  Since
$H=\frac12 \rho^2(\xi_0^2+|\eta|_{h_\rho}^2)$, we have in coordinates
$(\rho,y,\xi=\xi_0 d\rho+\eta.dy)$ 
\begin{equation}\label{X}
X=\rho^2\xi_0\pl_\rho + \rho^2\sum_{i,j}h_\rho^{ij}\eta_i\pl_{y^j}
-\big[\rho(\xi_0^2+|\eta|_{h_\rho}^2) +\tfrac12 \rho^2\pl_\rho
  |\eta|_{h_\rho}^2\big]\pl_{\xi_0}
-\tfrac12 \rho^2\sum_{k}\pl_{y^k}|\eta|_{h_\rho}^2\pl_{\eta_k}.
\end{equation}
Smooth coordinates $(\bbar{\rho},\bbar{y},\bbar{\xi}_0,\bbar{\eta})$ on 
${}^bT^*\bbar{M}$ are given by
\begin{equation}\label{newcoords}
\bbar{\rho}=\rho,\quad \bbar{y}=y,\quad
\bbar{\xi}_0=\rho\xi_0,\quad \bbar{\eta}=\eta.
\end{equation}
So
\[
\pl_\rho =\pl_{\bbar{\rho}} +\xi_0\pl_{\bbar{\xi}_0},\quad
\pl_{y}=\pl_{\bbar{y}},\quad
\pl_{\xi_0}=\rho\pl_{\bbar{\xi_0}},\quad
\pl_{\eta}=\pl_{\bbar{\eta}}.
\]
Substituting into \eqref{X}, one finds $X=\rho \bbar{X}$, with
\begin{equation}\label{formofbarX}
\bbar{X}=\bbar{\xi}_0\pl_{\bbar{\rho}} +
\bbar{\rho}\sum_{i,j}h^{ij}_\rho\bbar{\eta}_i\pl_{\bbar{y^j}} 
-\big[\bbar{\rho} |\bbar{\eta}|_{h_\rho}^2
+\tfrac12 \bbar{\rho}^2\pl_\rho
|\bbar{\eta}|_{h_\rho}^2\big]\pl_{\bbar{\xi_0}} 
-\tfrac12
\bbar{\rho}\sum_{k}\pl_{y^k}|\bbar{\eta}|_{h_\rho}^2\pl_{\bbar{\eta_k}}.  
\end{equation}
The result is now clear, since $\pl \bbar{S^*M}$
is given by $\bbar{\rho}=0$, and $\bbar{\xi_0}=\pm 1$ on 
$\pl \bbar{S^*M}$.
\end{proof}
We notice that a similar observation was made in \cite[Lemma 2.6]{MSV}.
For simplicity, in what follows we will use the notation
$(\rho,y,\bbar{\xi}_0,\eta)$ for the coordinates on ${}^bT^*\bbar{M}$,
instead of $(\bbar{\rho},\bbar{y},\bbar{\xi}_0,\bbar{\eta})$.

Recall that we identify each of $\pl_\mp S^*M$ with $T^*\pl \bbar{M}$ via
\eqref{ident}.  This identification depends on
the product decomposition induced by the choice of conformal 
representative $h$.  If $\widehat{h}=e^{2u}h$ is another choice, 
with $u\in C^\infty(\pl\bbar{M})$, 
and $\widehat{\rho}$, $\widehat{y}^i$ denote the
corresponding coordinates, then 
$\widehat{\rho}=e^u \rho +\mc{O}(\rho^2)$, $\widehat{y}^i=y^i+\mc{O}(\rho)$.  An 
easy calculation shows that
$$
\pm\widehat{\rho}^{-1}d\widehat{\rho}+\sum_i\widehat{\eta_i}d\widehat{y}^i 
=\pm \rho^{-1} d\rho + \sum_i\widehat{\eta_i}dy^i \pm du
$$ 
as elements of $\pl_\mp S^*M$.  So the identification \eqref{ident} is
determined up to the map $(y,\eta)\mapsto (y,\eta \mp du(y))$ of
$T^*\pl\bbar{M}$.  This is a symplectomorphism  
of $T^*\pl\bbar{M}$ for each $u$, so it follows that each of $\pl_\mp
S^*M$ has a canonical structure as a symplectic manifold, with symplectic
form $\sum_id\eta_i\wedge dy^i$.  

The Liouville $1$-form $\alpha$ on $T^*M$ is given by
$\alpha=\xi_0 d\rho+\eta.dy$ near the boundary and the symplectic form on
$T^*M$ is $d\alpha=d\xi_0\wedge d\rho+d\eta\wedge dy$.  The form $\alpha$ 
restricts to $S^*M$ as a contact form, satisfying $\alpha(X)=1$ and 
$\iota_X d\alpha=0$.  The associated volume form is $\mu=\alpha\wedge
(d\alpha)^n$.  We call 
Liouville symplectic form on $\pl_\pm S^*M$ the symplectic form 
$\sum_id\eta_i\wedge dy^i$ described in the previous paragraph.  
We call Liouville volume form on $\pl_\pm S^*M$ the volume
form $\mu_\pl=(\sum_id\eta_i\wedge dy^i)^n$ 
obtained from the Liouville symplectic form. The volume
forms $\mu$ and $\mu_\pl$ induce densities $|\mu|$ and $|\mu_\pl|$ on
$S^*M$ and $\pl_\pm S^*M$ called Liouville measures. 
The flow $\varphi_t:S^*M\to S^*M$ of $X$ preserves the Liouville measure.   

\begin{lemm}\label{mesure}
The Liouville $1$-form $\alpha$ on $S^*M$ is such that $\rho \alpha$ and  
$d\alpha$ extend smoothly to $\bbar{S^*M}$ and 
$\iota_{\pl}^*(d\alpha)$ is the symplectic form on
$\partial \bbar{S^*M}$.  The volume form $\mu=\alpha\wedge (d\alpha)^n$   
on $S^*M$ is such that $\rho\mu$ and $\iota _X\mu$ extend 
smoothly to $\bbar{S^*M}$, and $\iota_{\pl}^*\iota _X\mu$ is equal to the
Liouville volume form $\mu_\pl$ on $\pl \bbar{S^*M}$. 
\end{lemm}
\begin{proof}
We work on ${}^bT^*\bbar{M}$ in the coordinates
$(\rho,y,\bbar{\xi}_0,\eta)$.  We have  
\begin{equation}\label{alpha}
\alpha = \xi_0d\rho +\sum_i\eta_idy^i =\rho^{-1}\bbar{\xi}_0d\rho +\sum_i\eta_idy^i.
\end{equation}
Clearly $\rho\alpha$ extends smoothly to all of ${}^bT^*\bbar{M}$.  Now
$d\alpha = \rho^{-1}d\bbar{\xi}_0\wedge d\rho +\sum_id\eta_i\wedge dy^i.$
But differentiating 
$\bbar{\xi}_0{}^2 +\rho^2|\eta|^2_{h_\rho}=1$ shows that 
$\bbar{\xi}_0 d\bbar{\xi}_0 =- \rho |\eta|^2_{h}d\rho + \mc{O}(\rho^2)$ on
$T\bbar{S^*M}$.  Hence
\begin{equation}\label{dalpha}
d\alpha = \sum_id\eta_i\wedge dy^i +\mc{O}(\rho) \quad \text{on  } \,T\bbar{S^*M}. 
\end{equation}
In particular, $d\alpha$ extends smoothly to
$\bbar{S^*M}$ and $\iota_{\pl}^*(d\alpha)=\sum_id\eta_i\wedge dy^i$ as claimed.  
It follows also that $\rho\mu=(\rho\alpha)\wedge (d\alpha)^n$ and  
$\iota _X\mu =\alpha(X)(d\alpha)^n+0=(d\alpha)^n$ extend smoothly to
$\bbar{S^*M}$, and 
$\iota_{\pl}^*\iota _X\mu = \iota_{\pl}^*((d\alpha)^n) = 
(\iota_{\pl}^*(d\alpha))^n = (\sum_id\eta_i\wedge dy^i)^n$.
\end{proof}

\noindent
Observe from \eqref{alpha}, \eqref{dalpha} that 
$$
\rho \mu = \bbar{\xi}_0 d\rho \wedge (\sum_id\eta_i\wedge dy^i)^n +\mc{O}(\rho).
$$
Since $\bbar{\xi}_0=\pm 1$ on $\pl_\mp S^*M$, it follows that the
orientations induced by $\rho \mu$ and $(\sum_id\eta_i\wedge dy^i)^n$ agree on 
$\pl_+S^*M$, but are opposite on $\pl_-S^*M$.    

The boundary behavior of the geodesics of a conformally compact metric was
analyzed in \cite{Ma}, where in particular it was proved that the flow
$\varphi_t$ is complete.  The following lemma 
describing the trajectories of the flow lines of $X$ near the boundary
is essentially contained in \cite{Ma}.  We formulate the result in terms of
$\bbar{SM}$, and for completeness and for use in our intended applications,
we give a proof.  Note from    
\eqref{X} that Hamilton's equations for the integral curves of $X$ on the
level set $S^*M$ are given near $\pl\bbar{M}$ by    
\begin{equation}\label{floweq} 
\begin{gathered}
\dot{\rho}=\rho^2 \xi_0, \quad \dot{y}^j=\rho^2\sum_{i}h^{ij}_{\rho}\eta_i,\quad
 \dot{\xi}_0=-\frac{1}{\rho}-\frac{\rho^2}{2}\pl_{\rho}|\eta|^2_{h_\rho},
\quad \dot{\eta}_j=-\frac{\rho^2}{2}\pl_{y^j}|\eta|^2_{h_\rho}.
\end{gathered}
\end{equation} 

\begin{lemm}\label{convtobdry}
There is $\eps>0$ small enough so that for each $(x,\xi)\in S^*M$ with $\rho(x)<\eps$, and 
$\xi=\xi_0d\rho+\eta.dy$ with $\xi_0\leq 0$, the flow trajectory $\varphi_t(x,\xi)$ converges 
to a point $z_+\in \pl_+S^*M$ with rate $\mc{O}(e^{-t})$ as $t\to +\infty$
and $\rho(\varphi_t(x,\xi))\leq \rho(x,\xi)$ for all $t\geq 0$.   
In addition, if $A\subset S^*M\cap \{\rho\in(0,\eps), \xi_0\leq 0\}$ is a
compact set, then 
the set $\{\varphi_t(x,\xi);(x,\xi)\in A, t\geq 0\}$ is contained in a
compact set of $\bbar{S^*M}\cap  \{\rho\in [0,\eps), \xi_0\leq 0\}$. The
same results hold with $\xi_0\geq 0$ and backward time, with limit $z_-\in 
\partial_-S^*M$.   
\end{lemm}
\begin{proof}
First note that for any $z=(x,\xi)=(\rho,y,\xi_0,\eta)\in S^*M$,  
the trajectory $\varphi_t(x,\xi)=(\rho(t),y(t),\xi_0(t),\eta(t))$ satisfies  
$\rho(t)^2 \xi_0(t)^2+\rho(t)^2|\eta(t)|^2_{h_\rho(t)}=1$.  
In particular, 
$\rho(t)|\eta(t)|_{h_\rho(t)}$ is bounded.  
{From} \eqref{floweq}, we see that if $\eps>0$ is small enough and  
$\xi_0\leq 0$, $\rho\leq \eps$, then 
\[ 
\forall t\geq 0, \quad \dot{\xi}_0(t)<-\frac{1}{4\rho(t)} 
\textrm{ and } \dot{\rho}(t)\leq 0.
\]
Thus $u(t):=\rho(t)^{-1}$ satisfies 
\[
\ddot{u}=-\dot{\xi_0}\geq \tfrac14 u,\quad u(0)\geq \epsilon^{-1}, 
\quad \dot{u}(0)\geq 0.
\]
It follows that $u(t)\geq \epsilon^{-1}\cosh(t/2)$, so 
\begin{equation}\label{firstdecay}
\rho(t)\leq \frac{\eps}{\cosh(t/2)}.
\end{equation}
This preliminary decay estimate will be improved below.   

Now, differentiating $\bbar{\xi}_0(t):=\rho(t)\xi_0(t)$ by using
\eqref{floweq}, we get 
\[
\dot{\bbar{\xi}}_0(t)=\bbar{\xi}_0(t)^2-1+\mc{O}(\rho(t)^3|\eta(t)|^2_{h_{\rho(t)}})
=(\bbar{\xi}_0(t)^2-1)(1+\mc{O}(\rho(t))),
\] 
where we used $\rho(t)^2|\eta(t)|^2_{h_{\rho(t)}}=1-\bbar{\xi}_0(t)^2$ and
the remainder is uniform in $z$.     
Thus there exists $C>0$ uniform in $(x,\xi)$ such that 
\[ 
 \pl_t (F(\bbar{\xi}_0(t)))\leq -1+C\rho(t)  
\]  
where $F(v)=\demi \log \frac{1+v}{1-v}$.   Now \eqref{firstdecay} shows
that $\int_0^\infty \rho(t)\,dt<\infty$, so $F(\bbar{\xi}_0(t))\leq -t+C'$.   
Since $v+1=2e^{2F(v)}/\big(e^{2F(v)}+1\big)$, it follows that 
there is $C>0$  uniform such that for all $t\geq 0$
\begin{equation}\label{xiestimate}
0\leq \bbar{\xi}_0(t)+1\leq Ce^{-2t}.
\end{equation}
This implies that $\rho(t)|\eta(t)|_{h_{\rho(t)}}=\mc{O}(e^{-t})$ as $t\to
+\infty$ uniformly in $(x,\xi)$,  
thus from \eqref{floweq} we have $y(t)$ and $\eta(t)$  
converging exponentially fast to limits for each $(x,\xi)$  and moreover  
$(y(t),\eta(t))$ stays in a compact set 
if $(x,\xi)$ is in a fixed compact set of $S^*M\cap \{\rho 
\in(0,\eps),\xi_0\leq 0\}$.  Now we deduce
from this and from \eqref{xiestimate}, \eqref{floweq} that 
\begin{equation}\label{uniformrho}
0\leq \dot{\rho}/\rho+1\leq Ce^{-2t}, \quad  \rho(0)e^{-t}\leq \rho(t)\leq C\rho(0)e^{-t}
\end{equation}
where $C>0$ is uniform with respect to the initial condition $(x,\xi)$.  
Since all the ${}^bT^*\bbar{M}$-coordinates
$(\rho(t),y(t),\bbar{\xi}_0(t),\eta(t))$ of $\varphi_t(x,\xi)$ converge
exponentially with $\rho(t)\rightarrow 0$ and $\bbar{\xi}_0(t)\rightarrow
-1$, it follows that $\varphi_t(x,\xi)$ converges to some point 
$z_+\in \partial_+S^*M$ as $t\rightarrow \infty$.   
The same argument works in backward time with initial conditions such that
$\xi_0\geq 0$. 
\end{proof}

We remark that one can give an alternate proof of Lemma~\ref{convtobdry} by
analyzing the flow of the vector field $\bbar{X}$ defined in
Lemma~\ref{Xform}, which is smooth up to $\partial \bbar{S^*M}$.  We will
use such an approach in further analysis of the geodesic flow below.   

Lemma~\ref{convtobdry} implies via the duality isomorphism  
${}^bT^*\bbar{M}\cong {}^{\mc{L}}T\bbar{M}$ that the tangent vector to the
geodesic $\gamma(t)=\pi(\varphi_t(x,\xi))$ has the form
$\bbar{\xi}_0(t)\rho(t)\pl_\rho + \rho(t)^2\sum_{ij}
h_{\rho(t)}^{ij}\eta_i(t)\pl_{y^j}$  
with $\bbar{\xi}_0(t)\rightarrow -1$ and $\eta_i(t)$ convergent as
$t\rightarrow \infty$.  
Also, as a consequence of Lemma~\ref{convtobdry}, we see that the regions 
$\{\rho\geq \eps\}$ are strictly convex with respect to the flow for
$\eps>0$ small enough.   

We define the \emph{incoming (-) and outgoing (+) tails} of the flow by 
\[
\Gamma_\mp=\{ (x,\xi)\in S^*M;  \rho(\varphi_t(x,\xi))\not\to 0 \textrm{ as
}t\to \pm \infty\}.
\]
These are closed flow-invariant sets in $S^*M$.  
By Lemma \ref{convtobdry}, there is $\eps>0$ such that
\[
\Gamma_-\cap \{\rho<\eps, \xi_0\leq 0\}=\emptyset, \quad \Gamma_+\cap
\{\rho<\eps, \xi_0\geq 0\}=\emptyset.  
\]
We define the \emph{trapped set of the flow} to be the compact
flow-invariant set \[ K :=\Gamma_-\cap \Gamma_+.\] 
Notice that $K\cap \{\rho\leq \eps\}=\emptyset $ for some $\eps>0$ small 
enough, by Lemma \ref{convtobdry}. We say that $(M,g)$ is
\emph{non-trapping} if $K=\emptyset$.
\begin{lemm}
$(M,g)$ is non-trapping if and only if $\Gamma_+=\emptyset$ if and only if
  $\Gamma_-=\emptyset$.  
\end{lemm}
\begin{proof}
We show that $\Gamma_-\neq \emptyset$ implies $K\neq \emptyset$; 
the argument for $\Gamma_+$ is the same with the direction
of time reversed.  If $z\in \Gamma_-$, we can choose $t_n\rightarrow
\infty$ so 
that $z_n:=\varphi_{t_n}(z)\rightarrow y$ for some $y\in S^*M$.  Then $y\in
\Gamma_-$ since $\Gamma_-$ is closed.  But we also have $y\in \Gamma_+$,  
since otherwise by Lemma~\ref{convtobdry} there would be a small ball $B$
containing $y$ and 
$\eps>0$, $T>0$ so that $\varphi_{-t}(B)\subset \{\rho< \eps\}$ for all
$t>T$.  But $z_n\in B$ for large $n$, and $\varphi_{-t_n}(z_n)=z\notin
\{\rho< \eps\}$ if $\eps$ is small enough.  Since $t_n\rightarrow
\infty$ as $n\rightarrow \infty$, this is a contradiction.  
\end{proof}

Observe that if $(M,g)$ is non-trapping, then $M$ is necessarily simply
connected, as otherwise there would be a closed geodesic, and $g$ would
have a non-empty trapped set.  

Later, we will deal with the two cases where either $K=\emptyset$, or the  
trapped set $K$ is a hyperbolic set in the sense defined in the
introduction.   
It is shown in Proposition 2.4 of \cite{Gu} that if $K$ is a hyperbolic set, then 
for all $\eps>0$ small, ${\rm Vol}_{\mu}(\Gamma_\pm \cap\{\rho\geq
\eps\})=0$ (here we can use the results of \cite{Gu} since $\{\rho\geq
\eps\}$ is a strictly convex set in $S^*M$). In particular this implies
that ${\rm Vol}_{\mu}(K)=0$ and ${\rm Vol}_\mu(\Gamma_\pm)=0$ in $S^*M$. We
can also define the dual decomposition  
\[
T^*_K(S^*M)= \rr \alpha\oplus E_s^*\oplus E_u^*
\]
where $E_u^*(E_u\oplus \rr X)=0$, $E_s^*(E_s\oplus \rr X)=0$, and $\alpha$ is the contact 
form. As explained in Section 2.3 of \cite{Gu} (see also \cite[Lemma
  2.10]{DyGu}), the bundle $E_{s}$ has a continuous extension to
$\Gamma_-$, denoted $E_-$,  
and $E_u$ has a continuous extension to $\Gamma_+$, denoted $E_+$, 
in a way that the hyperbolicity estimates \eqref{hypofK} still hold. 
The dual bundles also have extensions $E_-^*$ over $\Gamma_-$ and $E_+^*$ 
over $\Gamma_+$, and $E_\pm^*$ are globally invariant by the symplectic
flow $\Phi_t$ on $T^*(S^*M)$.  Here $\Phi_t$ is the symplectic lift of the
flow $\varphi_t$ to $T^*(S^*M)$ given by 
\begin{equation}\label{symplift}
\Phi_t(z,\zeta)=(\varphi_t(z),(d\varphi_t(z)^{-1})^T.\zeta), \quad \zeta \in T^*_z(S^*M).
\end{equation}

As a consequence of Lemma~\ref{convtobdry}, we have the 
\begin{corr}\label{boundmap} 
The following maps are well-defined and smooth
\[B_\pm: S^*M\setminus \Gamma_\mp \to \pl_\pm S^*M , \quad
B_\pm(x,\xi):=\lim_{t\to \pm \infty}\varphi_t(x,\xi).\]  
Moreover, they extend smoothly to 
$\bbar{S^*M}\setminus \bbar{\Gamma_\mp}$, where 
$\bbar{\Gamma_\mp}$ denotes the closure of $\Gamma_\mp$ in $\bbar{S^*M}$, and 
$B_\pm(z)=z$ for each $z\in \pl_\pm S^*M$.   
\end{corr}
\begin{proof}
If $z_0:=(x_0,\xi_0)\notin \Gamma_-$, for $\eps_0>0$ as in Lemma
\ref{convtobdry},  there is $T>0$ large enough so that  
we have $\rho(t):=\rho(\varphi_t(z_0))<\eps_0$ for all $t>T$. There is
necessarily an open  
interval $A\subset [T,\infty)$ where $\rho(t)$ is decreasing, thus
  $\xi_0(t)<0$ on $A=(a,b)$ and by  
Lemma~\ref{convtobdry}, $\rho(t)$ is actually decreasing on $[a,+\infty)$ 
  and $\varphi_t(z_0)$ converges to a point in $\pl_+S^*M$ as $t\to
  +\infty$; this point is denoted by $B_+(z_0)$.  
Extend $\rho$ from a neighborhood of $\partial\bbar{M}$ 
to all of $\bbar{M}$ so that $\rho>0$ on $M$, and write $X=\rho\bbar{X}$    
as in Lemma~\ref{Xform}.   
The flow lines  of $\bbar{X}$ in $S^*M$ are the same as the flow lines of 
$X$, only the parametrization changes: if $\bbar{\varphi}_\tau(z)$ is the
flow of $\bbar{X}$, then 
\begin{equation}\label{reparam}
\forall z\in S^*M,\quad  \bbar{\varphi}_\tau(z)=\varphi_{t(\tau,z)}(z)
\textrm{ with }
t(\tau,z):=\int_0^\tau\frac{1}{\rho(\bbar{\varphi}_s(z))}ds. 
\end{equation}
Since $\bbar{X}$ is smooth on $\bbar{S^*M}$, does not vanish and is
transverse to $\pl\bbar{S^*M}$,   
the implicit function theorem gives that there is a finite time $\tau_+(z)$ 
smooth in $z$ near $z_0$ such that $\rho(\bbar{\varphi}_{\tau_+(z)}(z))=0$
for all $z$ near $z_0$. The map $B_+(z)$ is simply  
$\bbar{\varphi}_{\tau_+(z)}(z)$ and thus is smooth and extends smoothly to
$\pl\bbar{S^*M}\setminus \bbar{\Gamma_-}$ 
with $B_+(z)=z$ when $z\in \pl_+S^*M$. 
The same argument works with $B_-$.
\end{proof}

As in the proof of Corollary~\ref{boundmap}, we will always denote by
$\tau_\pm(z)\geq 0$ the time so that 
\[
\bbar{\varphi}_{\pm\tau_\pm(z)}(z)=B_{\pm}(z) , \quad z\notin 
\bbar{\Gamma_\mp}.  
\]
We also note that the closures can be described by 
\begin{equation}\label{Gammaclos}
\bbar{\Gamma_\pm}=\Gamma_\pm \cup \{z\in \pl_\pm S^*M; 
\bbar{\varphi}_\tau(z)\in \Gamma_\pm, \forall \tau, \mp
\tau>0\}.
\end{equation} 

We now define the {\it scattering map} 
$S_g: \pl_-S^*M \setminus \bbar{\Gamma_-}\to 
\pl_+S^*M \setminus\bbar{\Gamma_+}$
for the flow by
\begin{equation}\label{scattering}
S_g(z)=B_+(z)=\bbar{\varphi}_{\tau_+(z)}(z).  
\end{equation} 
Corollary~\ref{boundmap} shows that $S_g$ is well-defined and smooth.  
\begin{prop}
The scattering map $S_g: \pl_-S^*M \setminus \bbar{\Gamma_-}\to 
\pl_+S^*M \setminus\bbar{\Gamma_+}$  is a symplectic map.  
\end{prop} 
\begin{proof}
Recall from Lemma~\ref{mesure} that the symplectic form on
$\partial\bbar{S^*M}$ is $\iota_\partial (d\alpha)$.  
Observe that for each $\tau$,   
$\bbar{\varphi}_\tau^*d\alpha=d\alpha$, since 
$\mc{L}_{\bbar{X}}(d\alpha)=d(\iota_{\bbar{X}}d\alpha)=0$ by the fact that 
$\iota_{\bbar{X}}d\alpha=\rho^{-1}\iota_{X}d\alpha=0$. Now we can write
$S_g(z)=\bbar{\varphi}_{\tau_+(z)}(z)$  
for each $z\in \pl_-S^*M\setminus
\bbar{\Gamma_-}$. We thus get for each $v\in T_z(\pl_-S^*M)$  
\[ dS_g(z).v= d\bbar{\varphi}_{\tau_+(z)}(z).v + \bbar{X}(S_g(z))d\tau_+(z).v .\]
Therefore we get for each $v,w\in T_{z}(\pl_-S^*M)$
\[ \begin{split}
S_g^*(d\alpha)_z(v,w)= &
d\alpha_{S_g(z)}(d\bbar{\varphi}_{\tau_+(z)}(z).v,d\bbar{\varphi}_{\tau_+(z)}(z).w)+ 
(d\tau_+(z).v) d\alpha_{S_g(z)}(\bbar{X}(S_g(z)),w)\\
&+ (d\tau_+(z).w) d\alpha_{S_g(z)}(v,\bbar{X}(S_g(z)))\\
=&  d\alpha_{z}(v,w)
\end{split}\]
since $\iota_{\bbar{X}}d\alpha=0$ and $d\alpha_{\bbar{\varphi}_\tau(z)}
(d\bbar{\varphi}_{\tau}(z).v,d\bbar{\varphi}_{\tau}(z).w)=d\alpha_z(v,w)$
for each $\tau\leq \tau_+(z)$. 
\end{proof}

Recall that a choice of representative metric $h$ in the conformal infinity
of $g$ induces the identifications \eqref{ident} between $\pl_{\pm}S^*M$
and $T^*\pl\bbar{M}$.  When comparing $S_g$ for different metrics, 
we will view $S_g$ as mapping $T^*\pl\bbar{M}$ to itself via such  
identifications.  The main reason for this is that then $S_g=S_{\psi^*g}$
if $\psi:\bbar{M}\to \bbar{M}$ is a diffeomorphism restricting to the
identity on the boundary, so long as the identifications between   
$\pl_{\pm}S^*M$ and $T^*\pl\bbar{M}$ for $g$ and $\psi^*g$ 
are both with respect to the same metric
$h$.  Since $d\psi$ is generally nontrivial on $\pl_\pm S^*M$,
it is not typically the case that $S_g=S_{\psi^*g}$ when $S_g$ and  
$S_{\psi^*g}$ are viewed as maps from $\pl_-S^*M$ to $\pl_+S^*M$.       

Next we describe the pull-back by the flow.
\begin{lemm}\label{smoothpb}
Let $f\in C^\infty(\bbar{S^*M})$, then the function $(t,z)\mapsto
f(\varphi_t(z))$ is a smooth 
function on $\rr\x S^*M$ which can be written for $t\geq 0$ in the form
$f(\varphi_{\pm t}(z))=F_\pm(e^{-t},z)$ 
for some function $F_\pm \in C^\infty\big([0,1)\x (\bbar{S^*M}\setminus (\pl_\mp
S^*M\cup\Gamma_\mp))\big)$ satisfying  
\begin{equation}\label{Fpm} 
F_\pm(e^{-t},z)=f(B_\pm(z))+\mc{O}(\tau_\pm(z)e^{-t})
\end{equation} 
and the remainder is uniform for $z$ in compact sets of 
$\bbar{S^*M}\setminus (\pl_\mp S^*M\cup\Gamma_\mp)$.
\end{lemm}
\begin{proof}
The flow $\bbar{\varphi}_{\tau}$ of $\bbar{X}$ is a 
non-complete flow satisfying \eqref{reparam}. 
Since $\bbar{X}$ is smooth down to $\rho=0$ and since near each
$(y_0,\eta_0)\in \pl_+ S^*M$ we have  
$\bbar{X}= -\pl_\rho +\rho Y$ in the coordinates $(\rho,y,\eta)$ for some
smooth vector field $Y$ near $(y_0,\eta_0)$, we obtain that 
$\rho(\bbar{\varphi}_\tau(y,\eta))$ for small $\tau\leq 0$ 
is a smooth function of $(\tau,y,\eta)$ such that  
\begin{equation}\label{rhovphitau} 
\rho(\bbar{\varphi}_\tau(y,\eta))=-\tau+\mc{O}(\tau^2).
\end{equation}
with remainder uniform for $(y,\eta)$ in compact sets.
Now take $z_0\notin \Gamma_-\cup \pl_-S^*M$, we can write each point $z$ in
a small enough neighborhood of $z_0$ in $\bbar{S^*M}\setminus \pl_-S^*M$   
as $z=\bbar{\varphi}_{-\tau_+(z)}(B_+(z))$ with $\tau_+(z)$ smooth in $z$ and we get 
that for $\tau \in [0,\tau_+(z))$, the function $t(\tau,z)$ defined by
  \eqref{reparam} is given by 
\begin{equation}\label{taupm}
\begin{split}
t(\tau,z)=& \int_{-\tau_+(z)}^{\tau-\tau_+(z)}\frac{1}{\rho(\bbar{\varphi}_{s}(B_+(z)))}ds=
\int_{\tau_+(z)-\tau}^{\tau_+(z)}\frac{1}{s}ds +G(\tau,z)\\
=&-\log (1- \tfrac{\tau}{\tau_+(z)})+G(\tau,z) 
\end{split}
\end{equation}
where $G$ is a smooth function of $\tau,z$, for $\tau\in [0,\tau_+(z)]$ and
$z$ in a neighborhood  
of $z_0$ in $\bbar{S^*M}\setminus (\pl_-S^*M\cup \Gamma_-)$. 
This implies in particular that for $t\geq 0$, $\tau=
\tau_+(z)-e^{-t}\tau_+(z)H(e^{-t},z)$ for some smooth positive  
function $H$ on $[0,1)\x \bbar{S^*M}\setminus (\pl_-S^*M\cup \Gamma_-)$, thus for $t\geq 0$
\begin{equation}\label{renorflow}
\varphi_{t}(z)=\bbar{\varphi}_{-e^{-t}\tau_+(z)H(e^{-t},z)}(B_+(z)).
\end{equation} 
Thus if $f\in C^\infty(\bbar{S^*M})$, then $f(\varphi_{t}(z))$ 
is equal to $F_+(e^{-t},z)$ for some smooth function $F_+$ in $[0,1)\x \bbar{S^*M}\setminus (
\pl_-S^*M\cup \Gamma_-)$,  satisfying \eqref{Fpm}. The same argument works
with $f(\varphi_{-t}(z))$ for $t\geq 0$. 
\end{proof}

\subsection{Short geodesics}
In asymptotically hyperbolic manifolds, there are geodesics that are
arbitrarily small when viewed in the conformally compactified manifold. For
example, in hyperbolic space viewed as the unit ball, half-circles 
orthogonal to the unit sphere $\mathbb{S}^n=\pl\hh^{n+1}$ with endpoints
arbitrarily close to one another are such geodesics.  In order to prove the
existence and analyze these geodesics in general, we introduce two
types of local coordinates near the boundary $\pl\bbar{S^*M}$ and describe
the form of $\bbar{X}$ in each of them.  

Fix $0<\eps$ small.  We cover the region $\rho\in [0,\eps)$ of
$\bbar{S^*M}$ by the two types of neighborhoods   
\[
\begin{gathered}
U_1:=\{(\rho,y,\bbar{\xi}_0,\eta)\in \bbar{S^*M}; \rho<\eps,\,\,  \rho|\eta|_{h_\rho}<\tfrac12\},\\
U_2:= \{(\rho,y,\bbar{\xi}_0,\eta)\in \bbar{S^*M}; \rho<\eps,\,\,  |\eta|_{h_\rho}>\tfrac12\},
\end{gathered}
\]
and we use the coordinates $(\rho,y,\eta)$ on $U_1$ and 
$(\theta,y,\eta)$ on $U_2$, where $\theta\in[0,\pi]$ is defined by  
\[ \sin(\theta)= \rho|\eta|_{h_{\rho}}, \quad \cos(\theta)=\bbar{\xi}_0.\]
Notice that $U_1$ has two connected components $U_1^\pm$ corresponding 
to $\textrm{sign}(\bbar{\xi}_0)=\pm 1$. For $U_2$, if $\eps$ is small
enough, the function 
$\rho\to \rho|\eta|_{h_{\rho}}$ has positive derivative for $\rho<\eps$ so
is invertible; the limit $\rho\to 0$ when $\eta$ is in a compact set
corresponds to either $\theta\to 0$ or $\theta\to \pi$.  We can 
recover the coordinate $\xi_0$ by the expression   
$\xi_0=|\eta|_{h_\rho}\cot(\theta)$.  
A function $f:\bbar{S^*M}\to \rr$ is smooth if it is smooth on 
$S^*M$ and if $f|_{U_1}$ viewed in the coordinates $(\rho,y,\eta)$ 
extends as a smooth function to $\{\rho=0\}$, 
and $f|_{U_2}$ viewed in the coordinates $(\theta,y,\eta)$ extends as a smooth function 
down to $\{\theta=0\}$ and up to $\{\theta=\pi\}$. 
The vector field $\bbar{X}$ can be written in the coordinates $(\rho,y,\eta)$ in $U_1$ as
\begin{equation}\label{formofX} 
\bbar{X}= \textrm{sign}(\xi_0)\sqrt{1-\rho^2|\eta|^2_{h_\rho}}\pl_\rho+\rho
\sum_{i,j}h^{ij}_\rho\eta_i\pl_{y^j}- 
\demi \rho \sum_{k}\pl_{y^k}|\eta|^2_{h_{\rho}}\pl_{\eta_k}
\end{equation}
and in the coordinates $(\theta,y,\eta)$ in $U_2$ as 
\begin{equation}\label{forminXU2} 
Y:=|\eta|_{h_\rho}^{-1}\bbar{X}= (1+Q)\pl_\theta
+\sin(\theta)\sum_{i,j}\frac{h^{ij}_{\rho}\eta_i}{|\eta|_{h_\rho}^2}\pl_{y^j}
-\tfrac12 
\sin(\theta)\sum_{k}\frac{\pl_{y^k}|\eta|^2_{h_\rho}}{|\eta|^2_{h_\rho}}\pl_{\eta_k} 
\end{equation}
with $Q:=\frac{\sin \theta}{2|\eta|^3_{h_\rho}}\pl_{\rho}|\eta|^{2}_{h_\rho}$.
For instance, to evaluate the coefficient of $\pl_\theta$ in
\eqref{forminXU2}, from $\theta=\cos^{-1}\bbar{\xi_0}$ one has  
$\pl_{\bbar{\xi_0}}\theta= -(1-\bbar{\xi_0}^2)^{-1/2}=-(\rho|\eta|_{h_\rho})^{-1}$,
so from \eqref{formofbarX} there follows
\[
|\eta|_{h_\rho}^{-1}\bbar{X}\theta 
=(\rho|\eta|_{h_\rho}^2)^{-1}\big[\rho|\eta|_{h_\rho}^2
+\tfrac12 \rho^2\pl_\rho|\eta|_{h_\rho}^2\big]=1+Q.
\]

The existence and asymptotics of short geodesics is the content of the 
following lemma.   
\begin{lemm}\label{smallgeo} 
There exists $R_0>0$ so that if $z=(y_0,\eta_0)\in \pl_-S^*M$ with
$|\eta_0|_{h_0}>R_0$, then $z\notin \bbar{\Gamma}_-$.   
If we set $R=|\eta_0|_{h_0}>R_0$ and $\de=R^{-1}$, the integral 
curves $(\theta(s),y(s),\eta(s))$ of $Y$ have the property 
that $\theta$, $y$, and $\de \eta$ extend smoothly in $\de$ down to   
$\de =0$.  Moreover,
\[ 
\tau_+(z)=\de\pi+\mc{O}(\de^2), \quad
\rho(\bbar{\varphi}_{\tau}(z))=\de\sin(\alpha_z(\tau))+\mc{O}(\de^2) 
\] 
where $\alpha_z:[0,\tau_+(z)]\to [0,\pi]$ is a diffeomorphism depending
smoothly on $z$ and satisfying $\pl_\tau\alpha_z(\tau)=R+\mc{O}(1)$.    
\end{lemm}
\begin{proof} We write $X=\sin(\theta)Y$ in the region $U_2$ of $S^*M$, 
where $Y=|\eta|_{h_\rho}^{-1}\bbar{X}$ is given by  
\eqref{forminXU2} and we recall $\sin(\theta)=\rho|\eta|_{h_\rho}$. Denote
by $(\theta(s),y(s),\eta(s))$ the integral curve of $Y$ with initial
condition $(0,y_0,\eta_0)$ and set $R=|\eta_0|_{h_0}$ with $R> R_0$.     
Rescale the integral curve equations:  set $\de = R^{-1}$,
$u=(y-y_0)/\de$, $\om=\de \eta$.  Then $\om_0=\de\eta_0$ has 
$|\om|_{h_0}=1$ and the integral curve equations for $Y$ become 
\begin{equation}\label{intY}
\frac{d\theta}{ds}=1+\de\til{Q}\qquad
\frac{du^i}{ds}=\sin\theta\,\frac{h^{ij}_\rho\om_j}{|\om|^2_{h_\rho}}\qquad
\frac{d\om_i}{ds}=-\delta \sin\theta\,  
\frac{\pl_{y^i}h^{jk}_\rho\om_j\om_k}{2|\om|^2_{h_\rho}}
\end{equation}
with initial conditions $\theta(0)=0$, $u(0)=0$, $\om(0)=\om_0$, where 
$\til{Q}=Q/\de=\frac{\sin\theta}{2|\om|^3_{h_\rho}}\pl_\rho
h_\rho^{ij}\om_i\om_j$.  Everywhere the argument of $h_\rho$ and its derivatives
is $y_0+\de u$ and $\rho$ is determined implicitly as the solution of  
$\rho|\om|_{h_\rho}=\de\sin\theta$.  The right-hand sides of these
equations are smooth in all arguments $(\theta,u,\om,y_0,\de)$, including 
down to $\de =0$.  The solution for $\de=0$ is 
\[
\theta=s,\qquad \om=\om_0,\qquad u=(1-\cos s)\om_0^\sharp, 
\]
where $\om_0^\sharp$ is the dual vector to $\om_0$
using $h_0(y_0)$.  This corresponds to the geodesic on the hyperbolic space
defined  
by the metric $\rho^{-2}\big(d\rho^2+h_0(y_0)\big)$ with coefficients
frozen at $y_0$.  By a standard 
result \cite[Theorem 7.4]{CL}, there is a solution smooth in $\de\geq 0$
small and for all $s$ up to $\theta(s)=\pi$, whereupon $\rho= 0$ 
(one may continue slightly further by choosing some smooth extension of 
$h_\rho$ to $\rho<0$).  
Hence the geodesic reaches $\pl_+S^*M$, so $z\notin \bbar{\Gamma}_-$.  
The implicit function theorem implies that there is
a uniquely defined smooth function $s_0$ for $0\leq \de$ small for which  
$\theta(s_0)=\pi$ and $s_0=\pi$ for $\de = 0$.  We view $s_0(z)$ as a
function of $z=(y_0,\eta_0)$ for $R=|\eta_0|_{h_0}$ large.   So
$s_0(z)=\pi +\mc{O}(\de)$.  Since $|\om|_{h_\rho}=1$ 
for all $s$ when $\de=0$, it follows that $|\om|_{h_\rho}=1+\mc{O}(\de)$ 
for $s\in[0,s_0(z)]$, which becomes in terms of the original variables  
\begin{equation}\label{eta=1/R}
\de|\eta|_{h_\rho}=1+\mc{O}(\de).
\end{equation} 
We also have $\rho/(\de\sin\theta)=|\om|_{h_\rho}^{-1}=1+\mc{O}(\de)$, so 
$\rho=\de\sin\theta +\mc{O}(\de^2)$ uniformly for $s\in[0,s_0(z)]$.   

Since 
$Y=|\eta|_{h_\rho}^{-1}\bbar{X}$, we also can write the flow
$\bbar{\varphi}_{\tau}(z)$ as a reparametrization  
of the flow of $Y$ (just like in \eqref{reparam}) and viewing $s$ as a
function of $(\tau,z)$ we get 
$s(\tau,z)=R\tau+\mc{O}(\tau)$, $\tau_+(z)=\de\pi+\mc{O}(\de^2)$ and 
\[ 
\pl_\tau s(\tau,z)=|\eta(s(\tau,z))|_{h_\rho}=R+\mc{O}(1).
\]
Since $\dot{\theta}(s)=1+\mc{O}(\de)$, we obtain that
$\alpha_z(\tau):=\theta(s(\tau,z))$ satisfies   
$\pl_\tau\alpha_z(\tau)=R+\mc{O}(1)$ and this achieves the proof.
\end{proof}

\subsection{Splitting of $TS^*M$, Sasaki metric, conjugate
  points}\label{conjugate} 
As we discussed previously, for any Riemannian manifold $(M,g)$, the
tangent bundle of $S^*M$ is a contact manifold with contact splitting  
$TS^*M=\rr X\oplus \ker\alpha$.
Moreover, we have the further splitting
\begin{equation}\label{splitting} 
\ker\alpha=\mc{H}\oplus\mc{V},
\end{equation}
where $\mc{V}:=\ker d\pi$ is the vertical bundle and 
$\mc{H}=\ker \alpha\cap \ker \mc{K}$ is the horizontal bundle. 
Here $\mc{K}:TT^*M\rightarrow TM$ is the connection map, defined by 
$\mc{K}(\zeta)=D_t z(0)^\sharp$, where $\zeta\in T_{(x,\xi)}T^*M$, $z(t)$ 
is a curve in $T^*M$ with $z(0)=(x,\xi)$, $\dot{z}(0)=\zeta$, $D_t$ is 
the covariant derivative along the curve $\pi(z(t))$ in $M$, and 
${}^\sharp$ denotes the canonical isomorphism $T^*M\to TM$ induced by $g$  
(see \cite[Chap 1.3.1]{Pa} for details about $\mc{K}$).  If $z=(x,\xi)\in 
S^*M$, then any 
$\zeta \in \ker \mc{K}$, a priori only assumed to be in $T_z T^*M$,
is actually already in $T_z S^*M$.        
If $\mc{Z}\to S^*M$ is the bundle whose fibers are $\mc{Z}_{(x,\xi)}=\{ 
v\in T_xM: \xi(v)=0\}$, the maps  
$d\pi|_{\mc{H}}: \mc{H}\to \mc{Z}$ and $\mc{K}|_{\mc{V}}:\mc{V}\to
\mc{Z}$ are isomorphisms.  We denote by $\mc{L}$ the isomorphism  
\begin{equation}\label{L}
\mc{L}:\ker\alpha\rightarrow \mc{Z}\oplus \mc{Z},\qquad
\mc{L}(\zeta)=\big(d\pi(\zeta),\mc{K}(\zeta)\big). 
\end{equation}
The Sasaki metric $G$ on $S^*M$ is defined by  
\begin{equation}\label{sasaki}
G(\zeta,\zeta')=g(d\pi(\zeta),d\pi(\zeta'))+g(\mc{K}(\zeta),\mc{K}(\zeta')),  
\qquad \zeta,\zeta'\in T_z(S^*M).  
\end{equation}

If $z\in S^*M$, the space of normal Jacobi fields along the geodesic 
$\gamma_z(t):=\pi(\varphi_t(z))$ is isomorphic to
$\ker\alpha_z=\mc{H}_z\oplus\mc{V}_z$.  For $\zeta=h+v\in
\mc{H}_z\oplus\mc{V}_z$, the corresponding 
Jacobi field $Y(t)$ is determined by the initial conditions 
$\big(Y(0),D_tY(0)\big)=\mc{L}(\zeta)=\big(d\pi(h),\mc{K}(v)\big)$.
Two points $p,q\in M$ are said to be \emph{conjugate points} if  
there exist $z\in S^*_pM$ and $T>0$ so that $\varphi_T(z)\in S^*_qM$ and 
\begin{equation}\label{noconjugatepoints}  
d\varphi_{T}(z).\mc{V}(z)\cap \mc{V}(\varphi_T(z))\neq\{0\}.
\end{equation}
This is equivalent to 
the statement that there is a normal Jacobi field along $\gamma$ which
vanishes at both $0$ and $T$.  

Lemma, p. 201, of \cite{GKM} asserts that if $p$ is a point in a simply
connected Riemannian manifold $M$ such that $\exp_p$ is everywhere
defined and a local diffeomorphism, then the exponential map
$\exp_p:T_pM\to M$ is a diffeomorphism.  Since $\exp_p$ is everywhere
defined and a local diffeomorphism for each $p$ in a complete 
manifold with no conjugate points, it follows
that the exponential map $\exp_p:T_pM\to M$ is a diffeomorphism at each
point in a non-trapping asymptotically hyperbolic manifold with no
conjugate points.

\section{Boundary value problem and X-ray transform}\label{Xraysection}

We  first consider the non-trapping case, i.e. 
$\Gamma_-\cup\Gamma_+=\emptyset$. 
\subsection{Resolvent in the non-trapping case}
The first boundary value problem we consider is the following:
\begin{lemm}\label{resol1}
For each $\la\in\cc$, for each $f\in C_c^\infty(S^*M)$, there is a unique
$u_\pm(\la) \in C^\infty(S^*M)$ such that  
\[ (-X\pm\la)u^\la_\pm=f ,\,\textrm{ with } u_\pm(\la)=0 \textrm{ near }\pl_\pm S^*M.\]
and the operator $R_\pm(\la): C_c^\infty(S^*M)\to C^\infty(SM)$ defined by $R_\pm(\la)f=u_\pm(\la)$ is 
continuous and holomorphic in $\la$. 
\end{lemm}\label{resol2}
\begin{proof} The operator $R_\pm(\la)$ is simply given by
\begin{equation}\label{Rpm} 
R_+(\la)f(z)=\int_0^\infty e^{-\la t}f(\varphi_t(z))dt ,\quad
R_-(\la)f(z)=-\int_{-\infty}^0e^{\la t}f(\varphi_t(z))dt, 
\end{equation}
its continuity and uniqueness are clear. 
\end{proof}
We want to extend the action of these operators to $C^\infty(\bbar{S^*M})$.  
\begin{lemm}
The operators of Lemma~\ref{resol1} extend to 
holomorphic families of operators 
$R_\pm(\la):C^\infty(\bbar{S^*M})\to C^\infty(S^*M)$ for ${\rm Re}(\la)>0$
with meromorphic extensions to $\cc$ with first order poles at     
$-\nn_0$. The residue of $R_\pm(\la)$ at $\la=0$ is the operator $P_\pm$
defined by \[ P_\pm f=\pm f\circ B_\pm.\] 
\end{lemm} 
\begin{proof}
We just consider $R_+(\la)$ as $R_-(\la)$ is similar. First, we notice that
from Lemma \ref{convtobdry},  
for $z$ in any compact set $B\subset S^*M$, the curves $(\varphi_t(z))_{t\in\rr,z\in B}$ lie 
in a compact region of $\bbar{S^*M}$, thus
\[ R_+(\la)f(z)=\int_{0}^{\infty}e^{-\la t}f(\varphi_t(z))dt\]
converges uniformly on compact sets of $S^*M$ for ${\rm Re}(\la)>0$; it is
smooth  
and holomorphic in $\la$ there, and also bounded on $A\cap S^*M$ for each
compact set $A\subset \bbar{S^*M}$ (by a constant depending on ${\rm
  Re}(\la)$). Next, by Lemma \ref{smoothpb}, we can  
write for $z$ in each open set $A\subset S^*M$ with compact closure in $S^*M$
\[
R_+(\la)f(z)= \int_0^\infty e^{-\la t}F_+(e^{-t},z)dt
\]
for some $F_+$ smooth in $[0,1)\x \bbar{S^*M}\setminus \pl_-S^*M$ with $F_+(e^{-t},z)=f(B_+(z))+
\mc{O}(e^{-t}\tau_+(z))$, thus we have by Taylor expansion of $F_+(u,z)$ at
$u=0$ that for ${\rm Re}(\la)>0$ 
and each $N\in\nn$ 
\[\begin{split}
R_+(\la)f(z)= & \sum_{j=0}^N\int_0^\infty e^{-\la t}F_{+,j}(z)e^{-jt}dt+
\int_0^\infty e^{-\la t}e^{-t(N+1)}r_N(t,z)dt\\ 
= & \sum_{j=0}^N\frac{F_{+,j}(z)}{\la+j}+ \int_0^\infty e^{-\la t}e^{-t(N+1)}r_N(t,z)dt
\end{split}
\]
for some $r_N$ bounded in $[0,\infty)\x \bbar{A}$. The last integral is
  holomorphic in ${\rm Re}(\la)>-N-1$ and the first terms admit a 
  meromorphic extension with poles at $-\nn_0$. We notice that the residue
  at $0$ is given by the operator $P_+f:=f\circ B_+$. 
\end{proof}

We can then define the operators $R_\pm(0):=\lim_{\la\to
  0^+}(R_\pm(\la)-\la^{-1}P_\pm)$ acting on $C^\infty(\bbar{S^*M})$, which
by using \eqref{Fpm} can also be written (for $z\in S^*M$) as the
converging integrals 
\begin{equation}\label{R0}
\begin{gathered}
R_+(0)f(z)=\int_0^\infty (f(\varphi_t(z))-f(B_+(z)))dt, \\
R_-(0)f(z)=\int_0^\infty (f(B_-(z))-f(\varphi_{-t}(z)))dt.
\end{gathered}\end{equation}
We define $\rho C^\infty(\bbar{S^*M})$ to be the subspace of
$C^\infty(\bbar{S^*M})$ consisting of smooth functions on $\bbar{S^*M}$
which vanish at $\pl\bbar{S^*M}$ (such functions $f$  
can be factorized as $f=\rho \til{f}$ for some smooth $\til{f}$).
\begin{lemm}\label{propRpm}
The operators $R_\pm(0)$ defined by \eqref{R0} extend as continuous
operators 
\[ 
R_\pm(0):  C^\infty(\bbar{S^*M})\to  C^\infty(\bbar{S^*M}\setminus \pl_\mp
S^*M)
\]
satisfying $(R_\pm(0)f)|_{\pl_\pm S^*M}=0$ and 
\begin{equation}\label{identityR} 
-XR_\pm(0)={\rm Id}\mp P_\pm.
\end{equation}
If $f\in C^\infty(\bbar{S^*M})$ vanishes at $\pl_\pm S^*M$, then $u_\pm:=R_\pm(0)f$ 
is the unique smooth solution in $\bbar{S^*M}\setminus \pl_\mp S^*M$ of the boundary value problem
\[ -Xu_\pm=f ,\quad u_\pm|_{\pl_\pm S^*M}=0.\]
Finally $R_\pm(0)$ extend as continuous operators  
\begin{equation}\label{extR0}
R_\pm(0): \rho C^\infty(\bbar{S^*M})\to C^\infty(\bbar{S^*M}).
\end{equation}
\end{lemm}
\begin{proof} From Lemma \ref{smoothpb} and \eqref{Fpm}, we notice that $R_\pm(0)f$ extend smoothly 
to $\pl_\pm S^*M$ and we get near $\pl_\pm S^*M$ (uniformly on compact sets
of $\bbar{S^*M}\setminus\pl_-S^*M$)  
\begin{equation}\label{annulation}
R_\pm(0)f(z)=\mc{O}(\tau_\pm(z))
\end{equation}
thus vanishing at $\pl\bbar{S^*M}$. The second statement is clear since the
difference $u$ of  
two solutions would be constant along flow lines of $X$, and thus equal to
$u\circ B_\pm=0$. For the last part, we use \eqref{reparam} to get by a
change of variable $\tau\mapsto t(\tau,z)$ 
\begin{equation}\label{R_+viabarphi} 
R_+(0)f(z)=\int_0^{\tau_+(z)} f(\bar{\varphi}_{\tau}(z)) \frac{1}{\rho(\bar{\varphi}_{\tau}(z))}d\tau=
\int_0^{\tau_+(z)} \til{f}(\bar{\varphi}_{\tau}(z)) d\tau
\end{equation}
if $f=\rho \til{f}$ for some some smooth $\til{f}\in C^\infty(\bbar{S^*M})$. This
proves the last claim of the Lemma and the same argument works with 
$R_-(0)$. 
\end{proof}

\subsection{Extension operator and X-ray transform in the non-trapping case} 
The next boundary value problem for the flow we consider is the extension problem. We have the 
\begin{lemm}\label{ext}
For each $\omega\in C^\infty(\pl_-S^*M)$, there is a unique $w\in C^\infty(\bbar{S^*M})$ 
such that 
\[Xw=0, \quad w|_{\pl_-S^*M}=\omega\]
and it is given by $w(z)=\omega(B_-(z))$. Its value at $\pl_+S^*M$ is
$w|_{\pl_+S^*M}=\omega \circ S_g^{-1}$ 
where $S_g$ is the scattering map defined by \eqref{scattering}.
\end{lemm}
\begin{proof}
The solution $w$ has to satisfy $\bbar{X}w=0$ and $w$ is constant along
flow lines of $\bbar{X}$, thus  
$w(z)$ is given by $\omega(B_-(z))$. The other part is clear.
\end{proof}
We define the \emph{extension operator} using this Lemma by 
\begin{equation}\label{mcE}
\mc{E}: C^\infty(\pl_-S^*M)\to C^\infty(\bbar{S^*M}), \quad \mc{E}\omega (z)=\omega(B_-(z)).
\end{equation}
By Lemma \ref{convtobdry}, we also see that $\mc{E}: C_c^\infty(\pl_-S^*M)\to C_c^\infty(\bbar{S^*M})$.

We can now define the \emph{$X$-ray transform} operator $I$ by 
\begin{equation}\label{defXray}
I: \rho C^\infty(\bbar{S^*M})\to C^\infty(\pl_-S^*M), \quad If(z):=(R_+(0)f)|_{\pl_-S^*M}. 
\end{equation}
We can relate $\mc{E}$ to the operator $I$ by the
\begin{lemm}\label{adj}
The extension operator $\mc{E}$ defined by \eqref{mcE} is the adjoint of $I$ with respect to 
the scalar product induced by the Liouville measures $|\mu|$ on $S^*M$ and $|\mu_\pl|$  on 
$\pl_-S^*M$.
\end{lemm}
\begin{proof}
Let $f\in C_c^\infty(S^*M)$ and $\omega\in C^\infty(\pl_-S^*M)$, then by
using $\mc{L}_X\mu=0$, $X\mc{E}(\omega)=0$, $R_+(0)f|_{\pl_+S^*M}=0$ and
Lemma \ref{mesure}, we get 
\[ \begin{split}
\int_{S^*M}(f\cdot\mc{E}(\omega))|\mu|=&
\int_{S^*M}-X(R_+(0)f\cdot\mc{E}(\omega))\mu 
= -\int_{S^*M}\mc{L}_X(R_+(0)f\cdot\mc{E}(\omega)\mu)\\
=&-\int_{S^*M}d(R_+(0)f\cdot\mc{E}(\omega) \iota_X\mu)
= \int_{\pl_-S^*M}If\cdot\omega\,|\mu_\pl|  
\end{split}\]
which gives the desired property. The same argument works with $\omega$
smooth compactly supported and $f\in \rho C^\infty(\bbar{S^*M})$. 
\end{proof}
In view of this Lemma, we will instead write $I^*$ instead of $\mc{E}$ for
what follows when $\mc{E}$ acts on  
$C_c^\infty(\pl_-S^*M)$.
Using a similar argument, we also get a Santalo formula
\begin{lemm}\label{santalo}
Let $f\in C_c^\infty(S^*M)$, we have the identity
\[ \int_{S^*M} f|\mu| = \int_{\pl_-S^*M}If(z) |\mu_\pl(z)|.\] 
Consequently $I$ extends to a bounded operator $I:L^1(S^*M,|\mu|)\to L^1(\pl_-S^*M,|\mu_\pl|)$.
\end{lemm}
\begin{proof}
We just use Stokes formula like in the proof of Lemma \ref{adj} to get 
\[ 
\int_{S^*M}f\mu= -\int_{S^*M}X(R_+(0)f)\mu=\int_{\pl_- S^*M}If\,|\mu_\pl|.  
\] 
The boundedness of $I$ on $L^1$ just follows by density.
\end{proof}
We next relate the operator $I^*I$ to the resolvents $R_\pm(0)$. First, define the operator 
\begin{equation}\label{defPi}
\Pi : C^\infty(\bbar{S^*M})\to C^\infty(S^*M), \quad \Pi :=R_+(0)-R_-(0).
\end{equation} 
It satisfies for each $f\in C^\infty(\bbar{S^*M})$, $X\Pi f= P_+f+P_-f$ and thus 
\[ X\Pi f= 0 \,\, \textrm{ if }\,\, f|_{\pl \bbar{S^*M}}=0.\]
If $f\in \rho C^\infty(\bbar{S^*M})$, we can actually write $\Pi f$ as the converging integral
\begin{equation}\label{expPi}
\forall z\in S^*M, \quad  \Pi f(z)=\int_{-\infty}^\infty f(\varphi_t(z))dt
\end{equation}
(notice from \eqref{renorflow} that \eqref{expPi} also converges if $f$ is 
any continuous functions on $S^*M$ 
which is $\mc{O}(1/|\log(\rho)|^\alpha)$ for $\alpha>1$ near $\pl
\bbar{S^*M}$). Then we get 
\begin{lemm}\label{I*I}
We have that $\Pi=I^*I$ as operators mapping 
$C_c^\infty(S^*M)$ to $C^\infty(\bbar{S^*M})$, and this extends to the
identity $\Pi=\mc{E}I$ as operators mapping $\rho C^\infty(\bbar{S^*M})$ to 
$C^\infty(\bbar{S^*M})$. 
\end{lemm}
\begin{proof}
First by \eqref{R0}, we have for each $f,f'\in C_c^\infty(S^*M)$ real valued, 
\begin{equation}\label{adjoint}
\cjg R_+(0)f,f'\cjd=-\cjg f,R_-(0)f'\cjd,
\end{equation}
that is $R_+(0)^*=-R_-(0)$.
If $u:=R_+(0)f$, we have $u|_{\pl_-S^*M}=If$ and by Stokes formula 
\[ \cjg R_+(0)f,f\cjd =-\int_{S^*M}Xu.u \mu=-\demi
\int_{S^*M}X(u^2)\mu=\demi \int_{\pl_-S^*M}(If)^2|\mu_\pl| 
\]
which shows $\Pi f=I^*If$ using \eqref{adjoint}. By Lemma \ref{adj} and
since $If\in C_c^\infty(\pl_-S^*M)$ if $f\in C_c^\infty(S^*M)$, we get 
$\Pi=\mc{E}I$ on $C_c^\infty(S^*M)$ and thus $\Pi=\mc{E}I$ 
on $\rho C^\infty(\bbar{S^*M})$ by density and boundedness of $\mc{E}$ on $C^\infty(\bbar{S^*M})$.
\end{proof}
We can extend this identity to weighted $L^2$ spaces by using 
\begin{lemm}\label{weighted}
For $\beta>1/2$, the operator $I : |\log\rho|^{-\beta} L^2(S^*M,|\mu|)\to
L^2(\pl_-S^*M,|\mu_\pl|)$ is bounded  
and we have $\Pi=I^*I$ as a bounded operator $\Pi:  |\log\rho|^{-\beta} 
L^2(S^*M,|\mu|)\to  |\log\rho|^{\beta} L^2(S^*M,|\mu|)$. 
\end{lemm}
\begin{proof} Let $\beta>1$.  By Lemma \ref{smallgeo} and a change of
variable $\tau\mapsto \alpha_z(\tau)$ (with $\alpha_z(\tau)$ the
function of Lemma   \ref{smallgeo}), there is $R_0$ so that  
for each $R>R_0$ and all $z=(y,\eta)\in \pl_-S^*M$ with $|\eta|_{h}=R$,
there are $C,C'>0$ so that 
\[
\begin{split}
\int_{0}^{\tau_+(z)}
\rho^{-1}|\log \rho|^{-\beta}(\bbar{\varphi}_{\tau}(z))d\tau\leq & CR\int_{0}^{\tau_+(z)}
\sin(\alpha_z(\tau))^{-1}\Big|\log \frac{\sin(\alpha_z(\tau))}{R}\Big|^{-\beta}d\tau\\
\leq & C\int_0^\pi
\sin(\alpha)^{-1}\Big|\log\frac{\sin(\alpha)}{R_0}\Big|^{-\beta} 
(1+\mc{O}(1/R))d\alpha\\ 
\leq & C'.
\end{split}
\]
If $z=(y,\eta)$ is such that $|\eta|_{h}\leq R_0$, the
trajectory $\bbar{\varphi}_\tau(z)$ stays in a compact set of $\bbar{S^*M}$
and using \eqref{rhovphitau}, one has    
\[\int_{0}^{\tau_+(z)}\frac{\rho^{-1}}{|\log \rho|^{\beta}}(\bbar{\varphi}_{\tau}(z))d\tau\leq C\]
for some $C$ uniform in $z$ (depending on $R_0$). 
We can then write, by using Lemma \ref{santalo}, \eqref{R_+viabarphi} and
Cauchy-Schwartz, that for $f$ real-valued and $\beta>1$ 
\[ \begin{split}
||If||^2_{L^2}\leq &
\int_{\pl_-S^*M}\Big(\int_{0}^{\tau_+(z)}
\frac{1}{\rho|\log \rho|^{\beta}}(\bbar{\varphi}_{\tau}(z))d\tau\int_{0}^{\tau_+(z)}
\big(\frac{|\log \rho|^{\beta}}{\rho}f^2\big)(\bbar{\varphi}_{\tau}(z))d\tau\Big) |\mu_\pl(z)|\\
\leq & C\int_{\pl_-S^*M}\int_{0}^{\tau_+(z)}
(\rho^{-1}|\log \rho|^{\beta}f^2)(\bbar{\varphi}_{\tau}(z))d\tau |\mu_\pl(z)|\\
\leq & C|||\log \rho|^{\beta/2}f||_{L^2(S^*M,|\mu|)}
\end{split}\]
for some $C>0$ uniform. This proves the claim.
\end{proof}
Next we relate the X-ray transform to the scattering map. 
\begin{lemm}\label{Xrayscat}
Let $f\in C^\infty(\bbar{S^*M})$, then $IXf= S_g^*(f|_{\pl_+S^*M})-f|_{\pl_-S^*M}$.
\end{lemm} 
\begin{proof}
We have $Xf\in \rho C^\infty(\bbar{S^*M})$, thus $IXf$ makes sense as an
element in $C^\infty(\pl_-S^*M)$  
and by \eqref{R_+viabarphi}, we have for $z\in \pl_-S^*M$
\[ IXf(z)=\int_0^{\tau_+(z)} (\rho^{-1}Xf)(\bbar{\varphi}_\tau(z))d\tau=
\int_0^{\tau_+(z)} (\bbar{X}f)(\bbar{\varphi}_\tau(z))d\tau=f(\bbar{\varphi}_{\tau_+(z)}(z))-f(z)
\]
which completes the proof.
\end{proof}
Now we characterize the kernel of $I$ in the following Lemma.
\begin{lemm}\label{kerI}
Let $f\in \rho C^\infty(\bbar{S^*M})$, then $If=0$ if and only if there
exists $u\in \rho C^\infty(\bbar{S^*M})$ such that $Xu=f$. 
\end{lemm}
\begin{proof} 
If we set $u:=-R_+(0)f$, we have $Xu=f$ and $u|_{\pl_+S^*M}=0$ and 
$u\in C^\infty(\bbar{S^*M})$ by \eqref{extR0}. By definition of $If$, if
$If=0$ then $u|_{\pl_-S^*M}=0$. Notice that we also have $u=-R_-(0)f$.  
The converse follows from Lemma \ref{Xrayscat}. 
\end{proof}

\subsection{The case with hyperbolic trapping}
In this section, we assume that the trapped set $K$ is a hyperbolic set for
the geodesic flow. It has zero Liouville measure and $\Gamma_\pm$ also have
zero measure by \cite[Section 2.4]{Gu}.  
The incoming and outgoing resolvents $R_\pm(\la)$ can be defined like in
the non-trapping case by the expression \eqref{Rpm} for ${\rm Re}(\la)>0$.  
These integrals extend analytically to $\la\in \cc$ continuously as maps 
\[ R_\pm(\la): C_c^\infty(S^*M)\to C^\infty(S^*M\setminus \Gamma_\mp)\]
Since each compact set of $S^*M$ is included in some strictly convex manifold
$\{z\in S^*M; \rho(z)\geq \eps\}$ with boundary, we can use
\cite[Propositions 4.2 and 4.3]{Gu} which say that  
\begin{equation}\label{Rpm0trap} 
R_\pm(0): H^s_{\rm comp}(S^*M)\to H^{-s}_{\rm loc}(S^*M)\cap L^{p}_{\rm
  loc}(S^*M)
\end{equation}
for all $s>0$ and all $p\in[1,\infty)$, and the wave-front sets of
  $R_\pm(0)f$ satisfy 
\[
{\rm WF}(R_\pm(0)f)\subset E_\mp^*
\]
where $E_\mp^*\subset T_{\Gamma_\mp}^*(S^*M)$ are continuous 
subbundles over $\Gamma_\mp$ satisfying 
\[
E_-^*(E_s\oplus \rr X)=0, \quad E_+^*(E_u\oplus \rr X)=0
\]
on $K$.  Moreover $R_\pm(0)$ satisfy 
\[ 
-XR_\pm(0)={\rm Id}
\]
in the distribution sense in $S^*M$ and are given by the expressions 
\begin{equation}\label{expRtrap}
\begin{gathered}
R_+(0)f(z)=\int_0^{\tau_+(z)} f(\bbar{\varphi}_{\tau}(z))
\frac{1}{\rho(\bbar{\varphi}_{\tau}(z))}d\tau, \quad \forall z\notin
\Gamma_- , \\ 
R_-(0)f(z)=-\int_{-\tau_-(z)}^{0} f(\bbar{\varphi}_{\tau}(z))
\frac{1}{\rho(\bbar{\varphi}_{\tau}(z))}d\tau, \quad \forall z\notin
\Gamma_+.  
\end{gathered}\end{equation}
\begin{prop}\label{WFanalysis}
If the trapped set $K$ is a hyperbolic set, the resolvents $R_\pm(0)$
extend as bounded maps 
\[ 
\rho C^\infty(\bbar{S^*M})\to L^p_{\rm loc}(S^*M)
\]
for all $p<\infty$, $R_\pm(0)f$ extend as functions in  
$C^\infty(\bbar{S^*M}\setminus \bbar{\Gamma_\mp})$  where
$\bbar{\Gamma_\mp}$ are given by \eqref{Gammaclos}, and 
$(R_\pm(0)f)|_{\pl_\pm S^*M}=0$.  
Finally, as distributions on $S^*M$, we have  
\[
{\rm WF}(R_\pm(0)f)\subset E_\mp^*.
\]
\end{prop}
\begin{proof} First, the expression \eqref{expRtrap} for the resolvent and
  the arguments of Lemma~\ref{propRpm} show that for $f\in \rho
  C^\infty(\bbar{S^*M})$,  
$R_\pm(0)f$ extends as a smooth function in the subset of $\bbar{S^*M}$
  where $\tau_\pm$ is smooth, i.e. on $\bbar{S^*M}\setminus  
\bbar{\Gamma_\mp}$. In particular, $R_\pm(0)f$ is smooth near $\pl_\pm S^*M$ 
and since $\tau_\pm|_{\pl_\pm S^*M}=0$, we get $R_\pm(0)f|_{\pl_\pm S^*M}=0$. 
If in addition $\supp(f)\cap \Gamma_\pm=\emptyset$, then $R_\pm(0)f$ is easily seen to 
be in $C^\infty(S^*M)$ (\cite[Lemma 4.1]{Gu}).
Let us now show that $R_\pm(0)f$ makes sense as a function in $L_{\rm
  loc}^p(S^*M)$ for all $p<\infty$. We consider $R_+(0)f$, as the argument
is the same for $R_-(0)f$.  
In view of \eqref{Rpm0trap} and the discussion above, 
it suffices to consider $R_+(0)(\rho \til{f})$ with 
$\til{f}\in C^\infty(\bbar{S^*M})$ supported in an arbitrarily small open set
$U$ contained in a small neighborhood of $\bbar{\Gamma_+}\cap \{z\in
\bbar{S^*M}; \rho(z) \leq \eps\}$.  
If $U$ is a small enough open set  then $\bbar{\varphi}_{-\tau_0}(U)\subset
\{\rho>\delta\}$ for some $\tau_0>0$ and $\delta>0$. The following formula
holds for $\til{f}\in C^\infty(U)\cap C^\infty(\bbar{S^*M})$ 
\[ \bbar{\varphi}_{-\tau_0}^* R_+(0)(\rho \bbar{\varphi}_{\tau_0}^*\til{f})= R_+(0)(\rho \til{f}) 
\]
as functions on $S^*M\setminus \Gamma_-$.
Now we can use $\supp(\rho \bbar{\varphi}_{\tau_0}^*\til{f})\subset
\{\rho>\delta\}$, and \eqref{Rpm0trap} shows that
$\bbar{\varphi}_{-\tau_0}^* R_+(0)(\rho \bbar{\varphi}_{\tau_0}^*\til{f})$
makes sense as a function in $L^p_{\rm loc}(S^*M)$ for all $p<\infty$,
giving  
a sense to $R_+(0)(\rho \til{f})$ as an element in $L^p_{\rm loc}(S^*M)$ for
all $p<\infty$. The wave-front set of $R_+(0)(\rho \til{f})$ is the flowout of
the wave-front set  
$R_+(0)(\rho \bbar{\varphi}_{\tau_0}^*\til{f})$ by the map 
$\bbar{\Phi}_{\tau_0}:=(\bbar{\varphi}_{\tau_0},(d\bbar{\varphi}_{\tau_0}^{-1})^T)$
on $T^*(S^*M)$, and is thus contained in $E_-^*$ (using that $E_-^*$ is
invariant by the flow $\Phi_{t}$ of \eqref{symplift}, it is easy to check
that $\bbar{\Phi}_{\tau_0}(E_-^*)\cap T^*(S^*M)\subset E_-^*$). 
\end{proof}

The same argument as Lemma \ref{santalo} shows that for each $f\in
C_c^\infty(S^*M\setminus (\Gamma_+\cup\Gamma_-))$, we have the identity 
\begin{equation}\label{santalotrap}
\int_{S^*M} f|\mu| =
\int_{\pl_-S^*M}\int_0^{\tau_+(z)}f(\bbar{\varphi}_\tau(z))\frac{1}{\rho(\bbar{\varphi}_\tau(z))}
d\tau |\mu_\pl(z)|
\end{equation} 
and using that $S^*M\setminus (\Gamma_+\cup \Gamma_-)$ is an open set of 
full measure, we can use a density argument to deduce that Santalo's
formula \eqref{santalotrap} holds for all $f\in L^1(S^*M,|\mu|)$ and the
X-ray transform operator  
\[ I: L^1(S^*M,|\mu|)\to L^1(\pl_-S^*M,|\mu_\pl|), \quad
If(z):=\int_0^{\tau_+(z)}f(\bbar{\varphi}_\tau(z))\frac{1}{\rho(\bbar{\varphi}_\tau(z))}d\tau \] 
is bounded and can also be considered as a map 
$I: \rho C^\infty(\bbar{S^*M})\to C^{\infty}(\pl_-S^*M\setminus \bbar{\Gamma_-})$ by 
setting $If=(R_+(0)f)|_{\pl_-S^*M\setminus \bbar{\Gamma_-}}$.

We obtain a Livsic type theorem similar to Lemma \ref{kerI}
\begin{lemm}\label{kerIbis}
Assume that the trapped set $K$ is a hyperbolic set. Let $f\in \rho
C^\infty(\bbar{S^*M})$, then $If=0$ on  
$\pl_-S^*M\setminus \bbar{\Gamma_-}$ if and only if there exists $u\in \rho
C^\infty(\bbar{S^*M})$ such that $Xu=f$. 
\end{lemm}
\begin{proof} 
By Proposition \ref{WFanalysis}, if we set $u:=-R_+(0)f$, we have $Xu=f$ in
the distribution sense with $u\in C^\infty(\bbar{S^*M}\setminus 
\bbar{\Gamma_-})\cap L^p_{\rm loc}(S^*M)$ 
and $u|_{\pl_+S^*M}=0$. By definition of $If$, $If=0$ implies that 
\[R_+(0)f(z)=R_-(0)f(z) \quad \forall z\in S^*M\setminus (\Gamma_-\cup \Gamma_+) \]
but this also implies that $u=-R_-(0)f=-R_+(0)f$ as functions in $L^1_{\rm
  loc}(S^*M)$. Proposition \ref{WFanalysis} shows that $u\in
C^\infty(\bbar{S^*M}\setminus K)$, and that  
\[
{\rm WF}(u)\subset E_-^*\cap E_+^*=\{0\}
\]
thus $u\in C^\infty(\bbar{S^*M})$. We also have $u|_{\pl_-S^*M}=0$ since $u=-R_-(0)f$,
thus $u|_{\pl \bbar{S^*M}}=0$. That $If=0$ if $f=Xu$ with $u\in \rho C^\infty(\bbar{S^*M})$ 
is straightforward and goes like in Lemma \ref{Xrayscat}.
\end{proof}

\subsection{Injectivity of X-ray transform on tensors}

The gradient of a function $f\in C^\infty(S^*M)$ with respect to the Sasaki metric $G$ splits into
\[ 
\nabla f= (Xf)X+ \overset{v}{\nabla}f+\overset{h}{\nabla}f, \quad
\textrm{ with  }\,\overset{v}{\nabla}f\in C^\infty(S^*M,\mc{V}), \,  \,  
\overset{h}{\nabla}f\in C^\infty(S^*M,\mc{H})
\]
where $\mc{H},\mc{V}$ are the horizontal and vertical bundles defined in Section \ref{conjugate}.
We can then view $\overset{v}{\nabla}f,\overset{h}{\nabla}f$ as elements in
$C^\infty(S^*M; \mc{Z})$ using $d\pi$ and $\mc{K}$. 

We need to describe the behavior of $u=R_+(0)f$ near $\pl\bbar{S^*M}$  
when $f$ is a function vanishing to high order at the boundary, and
similarly for its derivative.  
\begin{lemm}\label{propofu} 
Assume that $g$ is either non-trapping or the trapped set $K$ is hyperbolic. 
Let $f\in C^\infty(S^*M)$ be a function which can be written 
as  $f=\rho^{k}\til{f}$ for some $\til{f}\in C^\infty(S^*M)$, and some
$k\in (1,\infty)$ and such that  
\[  ||\til{f}||_{L^\infty(S^*M)}<\infty ,\quad ||\nabla \til{f}||_{L^\infty(S^*M)}<\infty\] 
with respect to the Sasaki metric $G$.
Let $\eps>0$, the function $u_\pm:=R_\pm(0)f$ is such that there is
$C_{k,\eps}>0$ such that  
for all $z\in W_\pm^\eps:=\{ \rho\leq \eps, \pm \xi_0\leq 0\}$
\[ |u_\pm(z)|\leq C_{k,\eps}\rho(z)^k||\til{f}||_{L^\infty}, \quad 
||\nabla u_\pm(z)||_G\leq C_{k,\eps}\rho(z)^{k}(||\nabla \til{f}||_{L^\infty}+||\til{f}||_{L^\infty}).\]
More generally, if $z:=(\rho,y, \bbar{\xi}_0,\bar{\eta})$ are local
coordinates near the boundary of ${^0S}^*\bbar{M}$ with 
$\bbar{\xi}_0^2+\sum_{i,j}h_\rho^{ij}\bar{\eta}_i\bar{\eta}_j=1$ and if
$f\in \rho^{\infty}C^{\infty}({^0S}^*\bbar{M})$,  
for all multiindices $\alpha$, all $N\in\nn$ and all $z\in W_\pm^\eps$ 
\[ 
|\pl_z^\alpha u_\pm(z)| \leq
C_{N,\alpha}\rho(z)^{N} ||\rho^{-N}\pl_z^\alpha
f||_{L^\infty({^0S}^*\bbar{M})}.
\] 
\end{lemm}
\begin{proof} We deal with the case $u_+$ since the $u_-$ case is
  similar. From  \eqref{uniformrho},  
if $\eps>0$ is small enough, there is $C>0$ such that for all $z=(x,\xi)\in
W_+^\eps$ and all $t\geq 0$ 
\begin{equation} \label{unif1}
\rho(\varphi_t(z))\leq C\rho(z)e^{-t}
\end{equation} 
which implies that for such $z$
\[ 
|u_+(z)|\leq C_k\rho(z)^{k}||\til{f}||_{L^\infty}\int_0^\infty e^{-kt}dt\leq
C_k\rho(z)^{k}||\til{f}||_{L^\infty}.
\]
To estimate $\nabla u$, it suffices to estimate
$||\overset{v}{\nabla}u||_G$ and $||\overset{h}{\nabla}u||_G$.  
Using the decomposition of $d\varphi_t$ in the splitting \eqref{splitting}
in terms of Jacobi fields \cite[Lemma 1.40]{Pa}, we have for each $V\in
\mc{V}$ of Sasaki norm $1$ 
\[|G(\overset{v}{\nabla}u_+(z),V)|\leq C\int_{0}^\infty
\rho^{k}(\varphi_t(z))\Big(||\til{f}||_{L^\infty}\Big|\frac{d\rho}{\rho}(\varphi_t(z))\Big|_g 
+||\nabla \til{f}||_{L^\infty}\Big)(|Y_t(z)|_g+|Y'_t(z)|_g) dt
\]
where $Y_t(z)$ is the Jacobi field solving 
\[ 
Y''_t(z)+ \mc{R}(Y_t(z),\dot{x}(t))\dot{x}(t)=0, \quad Y_0(z)=0,
\quad Y'_0(z)=V
\] 
if $\mc{R}$ denotes the Riemann curvature tensor of $g$ and
$x(t)=\pi(\varphi_t(z))$. Since the  sectional curvatures at $x$ are
uniformly pinched in $(-1-c\rho(x),-1+c\rho(x))$  for some $c$ uniform,  
and since $\rho(\varphi_t(z))=\mc{O}(e^{-t})$ uniformly in $z$, we get 
$\mc{R}(Y_t(z),\dot{x}(t))\dot{x}(t)=-Y_t(z)+\mc{O}(e^{-t}|Y_t(z)|)$, and 
by Gronwall's inequality we deduce that there is $C>0$ so that for each $t$ 
and each $z\in W_+^\eps$  
\[ 
|Y_t(z)|_g+|Y'_t(z)|_g \leq Ce^{t}. 
\] 
One has $|d\rho/\rho|_g=1$ in the region $\rho\leq \eps$ and using the
uniform estimates \eqref{unif1} and \eqref{uniformrho}, we deduce that
there is $C_k>0$ such that 
\[
||\overset{v}{\nabla}u_+(z)||_G\leq
C_k\rho(z)^{k}(||\til{f}||_{L^\infty}+||\nabla \til{f}||_{L^\infty}).
\]
The same argument works with $|\overset{h}{\nabla}u_+(z)|$.

To prove the last statement, we first notice that $|\pl_{z}^\alpha\varphi_t(z)|\leq 
C_\alpha e^{c_0|\alpha|.|t|}$ for some $c_0$ by using Gronwall's inequality 
and the fact that the vector field $X$ has Lipschitz constants uniformly
bounded on the compact manifold ${^0}S^*\bbar{M}$. Then take
$k>N+c_0|\alpha|$, we have  
\[|\pl_z^\alpha u_\pm(z)| \leq C_{\alpha}\int_0^\infty \rho^N(\varphi_t(z))
||\rho^{-N}(\pl_z^{\alpha}f)||_{L^\infty} e^{c_0|\alpha|t}dt\]
and the conclusion follows from \eqref{unif1}. 
\end{proof}

We can view a symmetric tensor $f\in C^\infty(M; \otimes^m_ST^*M)$ of rank 
$m\in \nn_0$ as a function on  $S^*M$ by the map
\[ \pi_m^*: C^\infty(M; \otimes^m_ST^*M)\to C^\infty(S^*M), \quad
\pi_m^*f(x,\xi):=f(x)(\otimes^m\xi^\sharp). 
\]
where $\xi^\sharp$ is the dual to $\xi$ through the metric,
i.e. $g(\xi^\sharp,\cdot)=\xi$. 
If $m=0$, $\pi_0^*=\pi^*$ is simply the pullback by the base projection
$\pi: S^*M\to M$.  
Notice also that $\pi_m^*$ maps $C_c^\infty(M; \otimes^m_ST^*M)$ to $C_c^\infty(S^*M)$.
We denote by $\mc{D}'(S^*M)$ (resp. $\mc{D}'(M;\otimes^m_ST^*M)$) 
the space of distributions on $S^*M$, i.e. the dual space to the space
$C_c^\infty(S^*M)$ of compactly supported smooth functions on $S^*M$
(resp. dual space to $C_c^\infty(M;\otimes_S^mT^*M)$). 

It is straightforward to see that  $\pi_m^*$ maps continuously 
\[\pi_m^*: \rho^{-m}C^\infty(\bbar{M}; \otimes^m_ST^*\bbar{M})\to C^{\infty}(\bbar{S^*M}).\] 
The dual operator defined by
\[ {\pi_m}_*: \mc{D}'(S^*M)\to \mc{D}'(M,\otimes_S^mT^*M), \quad \cjg
\pi_m^*f,\til{f}\cjd=\cjg f,{\pi_m}_*\til{f}\cjd 
\] 
(for each $f\in C_c^\infty(M; \otimes^m_ST^*M),\til{f}\in
C_c^\infty(S^*M)$) is also continuous as a map 
\begin{equation}\label{pushm}  
\begin{gathered}
{\pi_m}_*: C^\infty(S^*M)\to C^\infty(M,\otimes_S^mT^*M),\\
{\pi_m}_*: \rho^NL^\infty(\bbar{S^*M})\to \rho^{N-m}L^\infty(\bbar{M},\otimes_S^mT^*\bbar{M}),\\
{\pi_m}_*: \rho^N C^\infty({^0S}^*\bbar{M})\to \rho^{N-m}C^\infty(\bbar{M},\otimes_S^mT^*\bbar{M}).
\end{gathered}
\end{equation} 

If $\mc{S}$ denotes the symmetrization operator on tensors, we define the 
symmetrized derivative as \[ D:= \mc{S}\nabla\]
where $\nabla: C^\infty(M; \otimes_S^mT^*M)\to C^\infty(M;
T^*M\otimes(\otimes_S^mT^*M))$ is the Levi-Civita 
connection for $g$. It is easy to check that for $f\in C^\infty(M; \otimes_S^mT^*M)$
\begin{equation}\label{DvsX_+} 
\pi^*_{m+1}Df=X\pi_m^*f.
\end{equation}
Recall also that for $m=1$ and $f$ a smooth $1$-form, 
$2Df=\mc{L}_{f^\sharp}g$ where $\mc{L}$ is the Lie derivative and 
$f^\sharp$ is the dual vector field to $f$ through $g$. 
For a tensor $f=\rho^{-m}\til{f}$ with $\til{f}\in C^\infty(\bbar{M};\otimes^m_S
T^*\bbar{M})$, one has for $|\xi|_g=1$ 
\begin{equation}\label{estimnablav}
||\nabla \pi_m^*f(x,\xi)||_G\leq C_m(|\nabla^{\bar{g}}
\til{f}(x)|_{\bar{g}}+|\til{f}(x)|_{\bar{g}}).
\end{equation}
for some constant $C_m>0$ depending on $m$.
We call X-ray transform on symmetric tensors of rank $m$ the map   
\[
I_m: \rho^{-m+1} C^\infty(\bbar{M};\otimes^m_S T^*\bbar{M})\to
C^\infty(\pl_-S^*M\setminus \bbar{\Gamma_-} ), \quad I_m:= I \pi_m^*. 
\]
We want to study the kernel of $I_m$ and we follow the presentation of
\cite[Section 2]{PSU2} in that aim. 
The generator $X$ of the flow acts also on smooth sections of $\mc{Z}$ by  
using parallel transport along geodesics: if $v\in \Gamma(\mc{Z})$,
$v(\varphi_t(x,\xi))$ is a vector field along the geodesic 
$\pi(\varphi_t(x,\xi))$, and we set 
\[ 
Xv(x,\xi)=\nabla_{\pl_t}v(\varphi_t(x,\xi))|_{t=0}
\]
(here $\nabla$ is the Levi-Civita
connection pulled back to the bundle $\mc{Z}$ over $S^*M$).  
The adjoints of $\overset{v}{\nabla}$ and $\overset{h}{\nabla}$ acting on 
compactly supported functions are denoted  
$\overset{v}{\rm{div}}$ and $\overset{h}{\rm{div}}$. Finally let
$\mc{R}_{x,\xi}:\mc{Z}_{x,\xi}\to \mc{Z}_{x,\xi}$ be the operator defined
by  
$\mc{R}_{x,\xi}v=\mc{R}(v,\xi^\sharp)\xi^\sharp$ where $\mc{R}$ is the Riemann curvature tensor of $g$.
By Lemma 2.1 in \cite{PSU2}, we have the commutator formulas
\begin{equation}\label{commut}
[X,\overset{v}{\nabla}]=-\overset{h}{\nabla}, \quad
[X,\overset{h}{\nabla}]=\mc{R}\overset{v}{\nabla}, \quad \overset{h}{\rm
  div}\overset{v}{\nabla}-\overset{v}{\rm div}\overset{h}{\nabla}=-nX 
\end{equation}
where we view $\overset{v}{\nabla}u$ and $\overset{h}{\nabla}u$ as sections
of the bundle $\mc{Z}$ for $u\in C^\infty(S^*M)$. 
Let us start with a simple 
\begin{lemm}\label{reductionexact}
Let $f\in \rho^{-m+1} C^\infty(\bbar{M};\otimes^m_S T^*\bbar{M})$ for
$m\in\nn$. Then there exists a tensor $q\in
\rho^{-m+2}C^\infty(\bbar{M};\otimes^{m-1}_S T^*\bbar{M})$ such that
$\iota_{\pl_\rho}(f-Dq)=0$ near $\pl \bbar{M}$. 
\end{lemm}
\begin{proof} We write $f$ as 
\[ f=\sum_{j=0}^m \mc{S}(f_j\otimes d\rho^{m-j})
\]
where $f_j$ are tangential symmetric tensors of rank $j$, i.e 
$\iota_{\pl_\rho}f_j=0$ near $\pl\bbar{M}$, and $f_j\in
\rho^{-m+1}C^{\infty}(\bbar{M};\otimes^j_ST^*\pl\bbar{M})$.   
First we recall (see \cite[Th. 1.159]{Be}) that for $g=\bar{g}/\rho^2$, 
\begin{equation}\label{Besse} 
\nabla^{g}_XY=\nabla^{\bar{g}}_XY
-\frac{d\rho}{\rho}(X)Y-\frac{d\rho}{\rho}(Y)X+\rho^{-1}\bar{g}(X,Y)\pl_\rho.\end{equation}   
Using this and the Koszul formula, we have 
\begin{equation}\label{nabladrho} 
 \nabla^gd\rho=\demi \pl_\rho
 h_\rho+\frac{d\rho^2}{\rho}-\frac{h_\rho}{\rho}. 
 \end{equation}
If $\alpha$ is a smooth tangential $1$-form in a collar $(0,\eps)_\rho\x
\pl\bbar{M}$ near $\pl\bbar{M}$ (that is  $\iota_{\pl \rho}\alpha=0$), we have
from the Koszul formula that  
$(\nabla^{g}\alpha)(\pl_\rho,\pl_\rho)=0$ near $\pl\bbar{M}$. 
If in addition $\alpha$ is smooth up to $\pl \bbar{M}$, we also get from \eqref{Besse} 
\begin{equation}\label{nablaalpha} 
\nabla^g\alpha =\frac{d\rho}{\rho}\otimes \alpha+\alpha\otimes \frac{d\rho}{\rho}+
\alpha', \quad \alpha' \in C^{\infty}(\bbar{M}; \otimes^2T^*\bbar{M}).
\end{equation}
We also have for $q_0$ a smooth function near $\pl \bbar{M}$
\[ D(q_0d\rho^{m-1})(\pl_\rho,\dots,\pl_\rho)=\pl_\rho q_0+(m-1)\rho^{-1}q_0.\]
To eliminate the $d\rho^m$ term in $f$, we have to solve 
$\pl_\rho(\rho^{m-1}q_0)=\rho^{m-1}f_0$ and, assuming that 
$f_0=\mc{O}(\rho^{-m+1})$, we thus set (for some $\chi\in
C_c^\infty([0,\eps))$ equal to $1$ near $0$) 
 \[ q_0(\rho,y)=\chi(\rho)\rho^{-m+1}\int_0^\rho s^{m-1}f_0(s,y)ds \in \rho^{-m+2}C^\infty(\bbar{M})\]
 so that $(f-D(q_0d\rho^{m-1}))(\pl_\rho,\dots,\pl_\rho)=0$ near
 $\pl\bbar{M}$. This means that
 $f-D(q_0d\rho^{m-1})=\sum_{j=0}^{m-1}\mc{S}(\til{f}_j\otimes d\rho^{m-1-j})$
 near $\pl\bbar{M}$ for some tensors $\til{f}_j\in 
 \rho^{-m+1}C^\infty(\bbar{M};\otimes^{j+1}_ST^*\pl\bbar{M})$. 
We can then proceed by induction. For a tensor $q_j\in C^\infty((0,\eps)\x
\pl\bbar{M};\otimes^{j}_ST^*\pl\bbar{M})$,
we have from \eqref{nabladrho}, \eqref{nablaalpha} 
\[ 
D(\mc{S}(q_j \otimes d\rho^{m-1-j}))=\mc{S}((\pl_\rho
q_j+(m-1+j)\rho^{-1}q_j+Aq_j)\otimes d\rho^{m-j}) + T,
\] 
where $A$ is a smooth section of ${\rm End}(\otimes_S^{j}T^*\pl \bbar{M})$
up to $\rho=0$, and $T$ is a section of 
$\mc{S}(d\rho^{m-j-1} \otimes (\otimes_S^{j+1}T^*\pl \bbar{M}))\oplus 
\mc{S}(d\rho^{m-j-2} \otimes(\otimes_S^{j+2}T^*\pl\bbar{M}))$.  
Consequently, for $r_j\in C^\infty((0,\eps)\x \pl\bbar{M};
\otimes_S^{j}T^*\pl \bbar{M}))$, the equation  
$D(\mc{S}(q_j \otimes d\rho^{m-1-j}))=\mc{S}(r_j\otimes d\rho^{m-j})$ modulo terms in 
$\mc{S}(d\rho^{m-1-j} \otimes(\otimes_S^{j+1}T^*\pl \bbar{M}))
\oplus \mc{S}(d\rho^{m-j-2} \otimes(\otimes_S^{j+2}T^*\pl\bbar{M}))$ 
becomes an ODE of the form
\[ 
(\pl_\rho+A)(\rho^{m+j-1}q_j)=\rho^{m+j-1}r_j.
\]
If $r_j\in \rho^{-m+1} C^\infty([0,\eps)\x \pl\bbar{M}; 
  \otimes_S^{j}T^*\pl \bbar{M}))$, there is $q_j\in
  \rho^{-m+2}C^\infty([0,\eps)\x\pl\bbar{M};\otimes_S^{j}T^*\pl \bbar{M}))$
    solving this ODE. Therefore, we can inductively construct $q\in 
    \rho^{-m+2}C^\infty(\bbar{M};\otimes_S^{m-1}T^*\bbar{M})$ which 
    satisfies $\iota_{\pl\rho}(f-Dq)=0$ near $\pl \bbar{M}$.  
\end{proof}

Next we will show the following 
\begin{prop}\label{vanishing}
Let $(M,g)$ be an asymptotically hyperbolic manifold and $m\geq 0$. Let
$f\in \rho^{1-m} C^\infty(\bbar{M};\otimes^m_S T^*\bbar{M})$ be a tensor
satisfying $I_mf=0$.  Then there exists a tensor $q\in \rho^{2-m}  
C^\infty(\bbar{M};\otimes_S^{m-1}T^*\bbar{M})$ such that, for all
$N\in\nn$, 
$f-Dq\in \rho^{N}C^\infty(\bbar{M}; \otimes_S^mT^*\bbar{M})$.
(In particular, in the case $m=0$, this states that $f\in
\rho^\infty C^\infty(\bbar{M})$.)  
\end{prop}
\begin{proof} Begin with the case $m=0$.  We will show that if 
$f=\rho^{k} \til{f}$ with $\til{f}\in   
C^\infty(\bbar{M})$ and $k\geq 1$, for $y_0\in \pl M$ fixed we will have
$\til{f}(0,y_0)=0$ if $I_0f=0$.  Since this holds for each $y_0$, we deduce 
that $f\in \rho^{k+1}C^\infty(\bbar{M})$ and by induction it vanishes to
all orders at $\pl \bbar{M}$. 

Let $R>R_0$ be large and set $\de=1/R$ as in the proof of Lemma
\ref{smallgeo}, so that the geodesics in $S^*M$
with endpoint in the past given by $(y_0, R\omega_0)$ for  $\omega_0\in  
T_{y_0}^*\pl M$ fixed with length  
$|\omega_0|_{h_0}=1$ are contained in a region $\rho\leq C\de$ 
where we can use the coordinates $(\theta,y,\eta)$. 
The proof of Lemma~\ref{smallgeo} shows that $\rho = \de \sin\theta
+\mc{O}(\de^2)$, 
$y=y_0+\de u=y_0+\mc{O}(\de)$, and $d\theta/dt =
(\sin\theta)(1+\mc{O}(\de)$ when $\theta$ is viewed as a function of $t$.      
Now $\theta:\rr \to (0,\pi)$ is a diffeomorphism for $\de$ small, so we
can change variable in the integral defining $I_0f(y_0,R\omega_0)$:  
\[ 
\begin{split}
I_0f(y_0,R\omega_0)&=\int_{-\infty}^\infty
(\rho^k\til{f})(\pi(\varphi_t(z)))dt=\int_{0}^\pi (\de\sin
  \theta)^k\til{f}(\gamma(\theta))\frac{d\theta}{\sin\theta}+\mc{O}(\de^{k+1})\\ 
&=\de^k\til{f}(0,y_0)\int_{0}^\pi (\sin
  \theta)^{k-1}d\theta+\mc{O}(\de^{k+1})
\end{split}
\]
where $\gamma(\theta)=(\rho(\theta),y(\theta))$ denotes the $\theta$
parametrization of the geodesic starting at $(y_0,R\om_0)$.
Thus if $I_0f=0$, we get $\til{f}(0,y_0)=0$.   

Now we show the case $m\geq 1$ similarly. We use 
Lemma~\ref{reductionexact} and since $I_m(Dq)=0$, we are reduced to
analyze the case where $\iota_{\pl\rho}f=0$ near $\pl\bbar{M}$.   
We can assume that the tensor $f$ can be written in the decomposition 
$[0,\eps)_\rho\x \pl \bbar{M}$ near a point $y_0\in \pl \bbar{M}$ as 
\[
f(\rho,y)=\rho^{k-m}\til{f}(\rho,y)=\rho^{k-m}\sum_{J}
\til{f}_J(\rho,y)dy^{J}  
\]
for some $\til{f}_J\in C^\infty(\bbar{M})$ and some $k\geq1$,
where $dy^J:=dy^{j_1}\cdots dy^{j_m}$ if $J=(j_1,\dots,j_m)$ with
$1\leq j_i\leq n$.  Since $X=(\sin\theta)Y$, \eqref{forminXU2} shows that   
\[
\frac{dy^j}{dt}=(\sin\theta)^2
\sum_i\frac{h^{ij}_{\rho}\eta_i}{|\eta|_{h_\rho}^2} 
=\de(\sin\theta)^2(\om_0^\sharp)^j+O(\de^2)
\]
where $\om_0^\sharp$ denotes the dual vector using $h_0(y_0)$.  
As before, for each 
$\omega_0\in T_{y_0}^*\pl \bbar{M}$, we have 
\[
\begin{split} 
I_mf(y_0,R\omega_0)=& 
\int_{-\infty}^\infty
(\rho^{k-m}\til{f})(\gamma(t))(\otimes^m\dot{y}(t))dt\\     
=&\de^k\til{f}(0,y_0)(\otimes^m\om_0^\sharp)\int_{0}^\pi (\sin
\theta)^{k+m-1}d\theta +\mc{O}(\de^{k+1}).  
\end{split}
\]
Thus 
$\til{f}(0,y_0)=0$ if $I_m(f)=0$, which shows by induction on 
$k$ that $f$ vanishes to all orders at $\pl \bbar{M}$.  
\end{proof}
Now we prove Theorem~\ref{injXray} using Proposition~\ref{vanishing} and 
energy identities.

\noindent
{\it Proof of Theorem~\ref{injXray}.}
First, we use Proposition \ref{vanishing}. In the case $m=0$,
  $f\in \rho^NC^\infty(\bbar{M})$ for all $N\in\nn$; in the case $m\geq 1$,
we have  
$f=Dq+f_1$ with $q\in
\rho^{2-m}C^\infty(\bbar{M};\otimes_S^{m-1}T^*\bbar{M})$ and $f_1\in 
  \rho^NC^\infty(\bbar{M};\otimes_S^mT^*\bbar{M})$ for all $N\in\nn$.  
Then $\pi_m^*f_1\in \rho^\infty C^\infty({^0S}^*\bbar{M})$.
 Note that $I_m(Dq)=IX\pi_{m-1}^*q=0$ by Lemma
 \ref{Xrayscat}, thus $I_m(f_1)=0$.  We can thus replace $f$ by $f_1$ 
 and to avoid too many notations, we will assume $f=f_1$ for the rest of
 the proof.  

The Pestov identity in the strictly convex region   
$W_\eps:=\{z\in S^*M; \rho(z)\geq \eps\}$ is proved for instance in 
\cite[Proposition 2]{PSU2} but we also need to take into account boundary terms. 
The manifold $W_\eps$ has boundary denoted by $\pl W_\eps$ and the natural 
volume form on it $\mu_\eps:=\iota_{\eps}^*\iota_X\mu$ if $\iota_\eps:
W_\eps\to S^*M$ is the inclusion map.  
We write $L^2$ for the space $L^2(W_\eps,|\mu|)$, then  using
\eqref{commut}, we get for each $u\in C^\infty(S^*M)$ 
\[ \begin{split}
||\overset{v}{\nabla}Xu||_{L^2}^2-||X\overset{v}{\nabla}u||_{L^2}^2=& 
\cjg \overset{v}{\rm div}\overset{v}{\nabla}Xu,Xu\cjd_{L^2}+\cjg
X^2\overset{v}{\nabla}u,\overset{v}{\nabla}u\cjd_{L^2}- 
\int_{\pl W_\eps}\cjg\overset{v}{\nabla}u,X\overset{v}{\nabla}u\cjd \mu_\eps\\
 =& \cjg (\overset{v}{\rm div}X^2\overset{v}{\nabla}-X\overset{v}{\rm
   div}\overset{v}{\nabla}X)u,u\cjd_{L^2} 
 -\int_{\pl W_\eps}\cjg\overset{v}{\nabla}u,(X\overset{v}{\nabla}-\overset{v}{\nabla} X)u\cjd\mu_\eps
\\
 =& \cjg (\overset{h}{\rm div}\overset{v}{\nabla}X-\overset{v}{\rm
   div}X\overset{h}{\nabla})u,u\cjd_{L^2} 
 +\int_{\pl W_\eps}\cjg\overset{v}{\nabla}u,\overset{h}{\nabla}u\cjd\mu_\eps\\
 =& \cjg (\overset{h}{\rm div}\overset{v}{\nabla}X-\overset{v}{\rm
   div}\overset{h}{\nabla}X-\overset{v}{\rm
   div}\mc{R}\overset{v}{\nabla})u,u\cjd_{L^2}+\int_{\pl
   W_\eps}\cjg\overset{v}{\nabla}u,\overset{h}{\nabla}u\cjd\mu_\eps\\ 
 =& -\cjg (nX^2+\overset{v}{\rm
   div}\mc{R}\overset{v}{\nabla})u,u\cjd_{L^2}+\int_{\pl
   W_\eps}\cjg\overset{v}{\nabla}u,\overset{h}{\nabla}u\cjd\mu_\eps\\ 
||\overset{v}{\nabla}Xu||_{L^2}^2-||X\overset{v}{\nabla}u||_{L^2}^2=&
n||Xu||^2_{L^2}-\cjg
\mc{R}\overset{v}{\nabla}u,\overset{v}{\nabla}u\cjd_{L^2}+\int_{\pl
  W_\eps}(\cjg\overset{v}{\nabla}u,\overset{h}{\nabla}u\cjd-nuXu)\mu_\eps 
  \end{split}\]
By Lemma \ref{kerI} and Lemma \ref{kerIbis} we have
$R_+(0)\pi_m^*f=R_-(0)\pi_m^*f$, which we denote by
$-u\in \rho C^\infty(\bbar{S^*M})$, and which satisfies $Xu=\pi_m^*f$. 
By Lemma \ref{propofu} we have estimates in $\{\rho\leq \eps, \pm \xi_0\leq
0\}$ for $R_\pm(0)\pi_m^*f$ and thus some estimates on $u$ in $\{\rho\leq
\eps\}$: this implies in particular that $u$ can be extended as a smooth
function on ${^0S}^*\bbar{M}$ which vanishes to all orders at the boundary 
$\{\rho=0\}$, i.e. $u\in \rho^\infty C^\infty({^0S}^*\bbar{M})$.  
Using also \eqref{estimnablav} we deduce that 
$u,||\overset{v}{\nabla}u||_G,||\overset{h}{\nabla}u||_G$ are in $\rho^{N}L^\infty \subset L^2(S^*M)$ 
for all $N>0$, and the following functions are also in these spaces
\[
||\overset{v}{\nabla}Xu||_{L^2}=  ||\overset{v}{\nabla}f||_G,\quad
||X\overset{v}{\nabla}u||_G\leq ||\overset{v}{\nabla}f||_G
+||\overset{h}{\nabla}u||_G. 
\]
A consequence of this is that we can let $\eps\to 0$ to 
obtain the Pestov identity
\[
||\overset{v}{\nabla}Xu||_{L^2(S^*M)}^2-||X\overset{v}{\nabla}u||_{L^2(S^*M)}^2
= n||Xu||^2_{L^2(S^*M)}-\cjg
\mc{R}\overset{v}{\nabla}u,\overset{v}{\nabla}u\cjd_{L^2(S^*M)}.
\]
If $m=0$ or $m=1$, we have $||\overset{v}{\nabla}Xu||^2=m(n-1+m)||Xu||^2$ and thus 
\begin{equation}\label{conseqpestov}
0=||X\overset{v}{\nabla}u||_{L^2(S^*M)}^2 + n(1-m)||Xu||^2_{L^2(S^*M)}-\cjg
\mc{R}\overset{v}{\nabla}u,\overset{v}{\nabla}u\cjd_{L^2(S^*M)}.
\end{equation} 

Let $Z\in C^\infty(S^*M;\mc{Z})\cap \rho^NL^\infty$ so that 
$|XZ|_G\in \rho^NL^\infty$ and $|X^2Z|_G\in \rho^NL^\infty$, and define the quadratic form
\[
A(Z)=||XZ||^2_{L^2(S^*M)}-\cjg \mc{R}Z,Z\cjd_{L^2(S^*M)}.
\]
We first claim that $A(Z)\geq 0$ for all such $Z$. Indeed, let 
$\chi_\eps(x)=\chi(\rho(x)/\eps)\in C_c^\infty(S^*M)$ with $\chi\in C^\infty([0,\infty))$
equal to $0$ in $\rho \in [0,1/2]$ and $1$ near $\rho=1$; then 
$\supp X(\chi_\eps)\subset \{\rho\in [\eps/2,\eps]\}$ and
$|X(\chi_\eps)|\leq C\eps^{-1}$. 
Write $Z_\eps=Z\chi_\eps$, which is compactly supported.  Since 
\[ 
A(Z_\eps)=\int_{S^*M}\chi_\eps (|XZ|_g^2-\cjg \mc{R}Z,Z\cjd_g) |\mu|+
\int_{S^*M} (|X(\chi_\eps)Z|^2_g+X(\chi_\eps^2)\cjg Z,XZ\cjd_g) |\mu| 
\]
we deduce that 
\[ 
\lim_{\eps\to 0}A(Z_\eps)=A(Z). 
\]
Now, we can use the fact that $g$ has no conjugate points, 
thus for each geodesic $\gamma_{x,\xi}$ with initial condition $(x,\xi)\in S^*M$, 
the index form satisfies for each smooth vector field $Z$ in $\mc{Z}$ along 
$\gamma_{x,\xi}$ with $Z(0)=Z(T)=0$
\[
\int_{0}^T\big(|\nabla_{\pl_t}Z(t)|^2_g-\cjg
\mc{R}_{x(t)}(Z(t),v(t))v(t),Z(t)\cjd_g\big) dt\geq 0,
\] 
where $x(t)=\pi(\varphi_t(x,\xi))$ and $v(t)=\dot{x}(t)$.
Decomposing $A(Z_\eps)$
along geodesics using Lemma \ref{santalo} (or \eqref{santalotrap} in case with trapping)
\[ 
A(Z_\eps)=\int_{\pl_-S^*M\setminus \bbar{\Gamma_-}}\int_0^\infty
\big(|\nabla_{\pl_t}Z_\eps(t)|_g^2-\cjg 
\mc{R}_{x(t)}(Z_\eps(t),v(t))v(t),Z_\eps(t)\cjd_g\big)dt|\mu_\pl| \geq 0.  
\]
We conclude that $A(Z)\geq 0$ as announced.

We next claim that if $Z$ is as in the previous paragraph and $A(Z)=0$,
then $Z=0$.  
When restricted to a non-trapped geodesic $\{\varphi_t(x,\xi);t\in\rr\}$,
$Z$ is viewed as a vector field $Z(t)$ in $\mc{Z}_{\varphi_t(x,\xi)}$ along  
$\gamma_{x,\xi}:=\pi(\{\varphi_t(x,\xi);t\in\rr\})$, and it satisfies   
$|Z|=\mc{O}(e^{-N |t|})$ and 
$|\nabla_{\pl_t}Z|=|XZ|=\mc{O}(e^{-N |t|})$ 
as $t\to\pm \infty$ by \eqref{uniformrho}.
We claim that if $A(Z)=0$, then for each geodesic 
$\gamma_{x,\xi}:=\pi(\{\varphi_t(x,\xi);t\in\rr\})$, 
we have 
$Z''(t)+\mc{R}_{x(t)}(Z(t),v(t))v(t)=0$, i.e. $Z$ is a Jacobi field which
in turn vanishes faster than any exponential as $t\to \pm \infty$. Indeed,
if $A(Z)=0$, since $A(Z+sY)\geq 0$ for all $s\in \rr$ and all  
$Y\in C_c^\infty(S^*M;\mc{Z})$, we have by differentiating at 
$s=0$ that  
\[ 
\int_{S^*M} (\cjg XZ,XY\cjd_g-\cjg\mc{R}Z,Y\cjd_g) |\mu|=0
\] 
and thus by integrating by parts, $X^2Z+\mc{R}Z=0$. Restricting this
identity to the geodesic $\gamma_{x,\xi}$ gives that $Z(t)$ is a Jacobi
field vanishing faster than any exponential at $\pm \infty$.  Any Jacobi
field vanishing faster than $e^{-|t|}$ as 
$t\to \infty$ or $t\to -\infty$ must vanish identically (see 
Lemma~\ref{lower}), which shows that $Z=0$ on the set of non-trapped
geodesics and thus everywhere by density (recall ${\rm
  Vol}(\Gamma_\pm)=0$). 

Now, using this with $Z=\overset{v}{\nabla}u$ in \eqref{conseqpestov}, we
obtain that $\pi_0^*f=Xu=0$ when $m=0$, 
showing (i). When $m=1$ we get $\overset{v}{\nabla}u=0$,
which means that $u=\pi_0^*q$  
for $q=c_n{\pi_0}_*u$ with $c_n>0$ depending only on $n$.  By \eqref{pushm}
we deduce that $q\in \rho^\infty C^\infty(\bbar{M})$, which shows (ii). 

To conclude, we consider the case (iii) of a symmetric tensor $f$ in $\ker
I_m$. We assume that the curvature is non-positive, so that the flow is
$1$-controlled in the sense of \cite[Section 4]{PSU2}. We use the proof of
\cite[Theorem 10.1]{PSU2}, which applies almost verbatim in our case. If
$u\in \rho^\infty C^\infty({^0S}^*\bbar{M})$ satisfies  
$Xu=\pi_m^*f$ (just as above for $m=0,1$), we decompose $u$ into eigenmodes
of the vertical Laplacian $\overset{v}{\Delta}:=\overset{v}{{\rm
    div}}\overset{v}{\nabla}$,   
$u=\sum_{k=0}^\infty u_k$. We recall from \cite{PSU2} that these eigenmodes
generate some bundles $\Omega_k$ over $S^*M$, and the operator $X$ maps  
$C^\infty(S^*M;\Omega_k)\to C^\infty(S^*M;\Omega_{k-1}\oplus
\Omega_{k+1})$. Let $\til{u}=u-\sum_{k\leq m-1}u$, then the same arguments
as in the proof of Theorem 10.1 of \cite{PSU2} show that
$X(\overset{v}{\Delta}\til{u}-(m(m+1+n)+n)\til{u})=0$. But  
$\overset{v}{\Delta}\til{u}$ and $\til{u}$ decay to all orders at $\pl 
\bbar{S^*M}$, thus $\overset{v}{\Delta}\til{u}-(m(m+1+n)+n)\til{u}=0$ and
thus $u_k=0$ for all $k\not=m+1$. Then $X\til{u}=Xu_{m+1}=X_-u_{m+1}$ with   
$X_\pm: C^\infty(S^*M;\Omega_{k}) \to C^\infty(S^*M;\Omega_{k\pm 1})$ being
the differential operators so that $Xw=X_+w+X_-w$ for $w\in
C^\infty(S^*M;\Omega_{k})$. In particular we obtain $X_+u_{m+1}=0$ and that
is equivalent to $u_{m+1}$ being the lift by $\pi_{m+1}^*$ of a trace-free
conformal Killing tensor. But since $u_{m+1}$ decays to all orders at  
$\pl\bbar{S^*M}$, the conformal  
Killing tensor vanishes at the boundary $\pl \bbar{M}$ and a standard
Weitzenbock formula shows that $u_{m+1}=0$ (see for example the proof in
\cite[Proposition 6.6]{HMS}).  
By using \eqref{pushm}, this implies that $u=\sum_{k\leq m-1}u_k$ is of the
form $\pi_{m-1}^*q$ for some $q\in \rho^\infty C^\infty(\bbar{M};
\otimes_S^{m-1}T^*\bbar{M})$, which satisfies $Xu=\pi_m^*f$. By
\eqref{DvsX_+}, we have $Dq=f$.  
\stopthm
\begin{rem} The proof above actually shows injectivity of $I_m$ on spaces 
with weaker regularity and decay. For example for $I_0$, it is only
 required that the quantities in the Pestov identity are finite, the
 boundary term at $\rho=\eps$ tends to $0$ and $A(Z)\geq 0$. For
 Cartan-Hadamard manifolds (complete simply connected manifolds with
 non-positive curvature), sharp conditions on the decay are given in 
 \cite{Le} in dimension $2$ and in \cite{LRS} in dimensions greater than 
 $2$.    
\end{rem}

\section{Renormalized length of geodesics} 

\subsection{Renormalized length}
First, we define the renormalized length.
\begin{lemm}\label{renormdist}
Let $(M,g)$ be an asymptotically hyperbolic manifold and $\rho$ a
geodesic defining function. For each $z\in 
\pl_-S^*M\setminus \bbar{\Gamma_-}$, the function  
$\la\mapsto I_0(\rho^{\la})(z)$ has a meromorphic extension from ${\rm
  Re}(\la)>0$ to 
${\rm Re}(\la)>-1$, with only a simple pole at $\la=0$ and residue 
\[{\rm Res}_{\la=0}I_0(\rho^\la)=2.\]
The regular value $L_g(z):=(I_0(\rho^\la)(z)-2/\la)|_{\la=0}$ for each
$z\in \pl_- S^*M\setminus \bbar{\Gamma_-}$ is also given by  
\begin{equation}\label{Lg}
L_g(z)=\lim_{\eps\to 0}\Big(\ell_g(\gamma_z\cap
\{\rho>\eps\})+2\log\eps\Big)
\end{equation}
where $\gamma_z$ is the geodesic $\pi(\bbar{\varphi}_\tau(z))_{\tau\in(0,\tau_+(z))}$. \end{lemm}
\begin{proof} Equation \eqref{R_+viabarphi} shows that for $z\in
  \pl_-S^*M\setminus \bbar{\Gamma_-}$ and ${\rm Re}(\la)>0$ we can write  
\[
I_0(\rho^\la)(z)=
\int_{0}^{\tau_+(z)}\rho(\bbar{\varphi}_\tau(z))^{\la-1}d\tau
=\int_0^\delta + \int_\delta^{\tau_+(z)-\delta} +\int_{\tau_+(z)-\delta}^{\tau_+(z)}.
\]
The analogue of \eqref{rhovphitau} near $\pl_-S^*M$ shows that if
$\delta>0$ is small enough, then there is $f(\tau,z)$ smooth for 
$\tau \in[0,\delta]$ with $|\tau f(\tau,z)|<1/2$ so that   
$\rho(\bbar{\varphi}_\tau(z))=\tau(1+\tau f(\tau,z))$.  Integration by
parts shows that
\[
\int_{0}^\delta \rho(\bbar{\varphi}_\tau(z))^{\la-1}d\tau
=\int_{0}^{\delta}\tau^{\la-1}(1+\tau f(\tau,z))^{\la-1}d\tau
= \frac{1}{\la} +p(\la)
\]
where $p(\la)$ is holomorphic in ${\rm Re}(\la)>-1$. 
Likewise 
$\int_{\tau_+(z)-\delta}^{\tau_+(z)}\rho(\bbar{\varphi}_\tau(z))^{\la-1}d\tau$
has the same form.  Since
$\int_\delta^{\tau_+(z)-\delta}\rho(\bbar{\varphi}_\tau(z))^{\la-1}d\tau$
is an entire function of $\lambda$, the first statement in the Lemma is
proved.  

Next, we see that if we set 
$\rho(t):=\rho(\varphi_t(z_0))$ for $z_0$ a point on the orbit 
$\{\bbar{\varphi}_\tau(z);\tau\in(0,\tau_+(z))\}$, 
then for any $\eps>0$
\[ 
\lim_{\la\to 0}\Big(I_0(\rho^\la)(z)-\frac{2}{\la}\Big)= \int_{\rho(t)>\eps}dt+\lim_{\la\to 0}
\Big(\int_{\rho(t)<\eps}\rho(t)^{\la}dt-\frac{2}{\la}\Big).
\]
By \eqref{formofX}, we have for $\pm t>0$ large enough 
$\pl_t\rho(t)=X(\rho)(\varphi_t(z_0))=\mp \rho(t)+\mc{O}(\rho(t)^2)$, so we 
can change variable $u:=\rho(t)$ and write 
\[\int_{\rho(t)<\eps}\rho(t)^{\la}dt= 2\int_0^{\eps}u^\la \frac{du}{u(1+\mc{O}(u))}=
\frac{2\eps^{\la}}{\la}+ \int_0^\eps u^{\la}a(u)du
\]
where $a(u)$ is continuous in $[0,1]$. Thus 
\[\lim_{\la\to 0}
\Big(\int_{\rho(t)<\eps}\rho(t)^{\la}dt-\frac{2}{\la}\Big)=2\log\eps+\mc{O}(\eps)\] 
which completes the proof by letting $\eps\to 0$.
\end{proof}

\begin{defi}\label{renormlength}
The function $L_g:\pl_-S^*M\setminus \bbar{\Gamma_-}$ defined in
Lemma~\ref{renormdist} is called the \emph{renormalized length
  function} associated to $\rho$.  
\end{defi}

If $\widehat{\rho}$ is an arbitrary boundary defining function (not necessarily
geodesic), it can be written as $\widehat{\rho}=\rho e^{\omega}$ for some
geodesic boundary defining function $\rho$, and   
the same argument as in Lemma \ref{renormdist} shows that
$I_0(\widehat{\rho}^{\la})$ has a meromorphic extension to ${\rm
  Re}(\la)>-1$ with a pole at $\la=0$ and 
residue ${\rm Res}_{\la=0}I_0(\widehat{\rho}^\la)=2$. Moreover, one has  
\[
\begin{split}
\big[I_0(\widehat{\rho}^\la)(z)-I_0(\rho^\la)(z)\big]\big|_{\la=0}
&=\lim_{\la\to   0}I_0(\rho^\la(e^{\la\omega}-1))(z)=\lim_{\la\to 0} \la 
I_0(\rho^\la \omega)(z)\\
&= {\rm Res}_{\la=0}I_0(\rho^\la \omega)(z).
\end{split}
\]
By the same argument as in Lemma \ref{renormdist}, it is direct to evaluate
this residue to obtain
\[
\big[I_0(\widehat{\rho}^\la)(z)-I_0(\rho^\la)(z)\big]\big|_{\la =0}
=\omega(\pi(z))+\omega(\pi(S_g(z))).
\]
So the renormalized length function associated to any defining function can
be defined as in Lemma~\ref{renormdist}, and it satisfies 
\begin{equation}\label{dependLg}
\widehat{L}_g(z)-L_g(z)=\omega(\pi(z))+\omega(\pi(S_g(z))). 
\end{equation}
Note that two defining functions determine the same $L_g$ if they induce
the  same representative for the conformal infinity.  In particular,  
a general defining function determines the same $L_g$ as the geodesic
defining function inducing the same representative for the conformal
infinity. 

It is evident that if $\psi:\bbar{M}\to \bbar{M}$ is a diffeomorphism which
restricts to the identity on $\pl\bbar{M}$, then $L_g=L_{\psi^*g}$, where
both renormalized lengths are calculated with respect to the same
representative for the conformal infinity.  

\subsection{Boundary determination}

In this section we prove Theorems~\ref{determination} and
\ref{realanalytic}.  In both proofs we will use the observation that if
$g,g'$ are two asymptotically hyperbolic metrics on $M$ and $h,h'$ are 
representative metrics for their respective conformal infinities, there
exists a diffeormorphism $\psi:\bbar{M}\to \bbar{M}$ equal to the identity
on $\pl\bbar{M}$ so that in the product decomposition $[0,\eps]_\rho\x 
\pl\bbar{M}$ for $g$ induced by $h$, one has   
\begin{equation}\label{choicediffeo}
g=\frac{d\rho^2+h_\rho}{\rho^2},\qquad \psi^*g'=\frac{d\rho^2+h'_\rho}{\rho^2}
\end{equation}
where $h_\rho$ and $h'_\rho$ are smooth 1-parameter families of metrics on 
$\pl \bbar{M}$ with $h_0=h$, $h_0'=h'$.  In fact, if 
$\chi,\chi':[0,\eps]_\rho\x \pl\bbar{M}\to \bbar{M}$ are the boundary
identification maps for $g$, $g'$ corresponding to $h$, $h'$, meaning that 
\[
\chi^*g=\frac{d\rho^2+h_\rho}{\rho^{2}},\qquad
\chi'^*g'=\frac{d\rho^2+h'_\rho}{\rho^{2}},
\]
then one can take $\psi$ to be an extension of $\chi'\circ\chi^{-1}$ to 
$\bbar{M}$.  Note that if $g$ and $g'$ are real-analytic, then $\psi$ can
be taken to be real-analytic near $\pl\bbar{M}$.  

\noindent
{\it Proof of Theorem~\ref{determination}.}
Choose $\psi$ as in \eqref{choicediffeo}; we will show that 
$L_g=L_{g'}$ implies that $h_\rho =h'_{\rho}$ to infinite order.   

We work with one metric $g$ in normal form and use the short geodesics
derived in Lemma~\ref{smallgeo} to show that $L_g$ determines the Taylor 
expansion of $h_\rho$ at $\rho=0$.  .  
Fix $y_0\in \pl\bbar{M}$ and consider the renormalized  
length $L_g(y_0,\eta_0)$ (using $\rho$) where 
we write $\eta_0=\delta^{-1}\omega_0$ with $\delta$ small and $0\neq 
\omega_0\in T_{y_0}^*\pl\bbar{M}$ fixed, but, at least to start, not
necessarily satisfying $|\omega_0|_{h_0}=1$.      
Equation \eqref{forminXU2} implies that $X\theta = \sin\theta(1+Q)$, 
where $Q=\frac{\sin \theta}{2|\eta|^3_{h_\rho}}\pl_{\rho}|\eta|^2_{h_{\rho}}$. 
So we can rewrite the integral (with $z$ any point on the 
orbit $\bbar{\varphi}_\tau(y_0,\eta_0)$) for ${\rm Re}(\la)>0$    
\[ 
\int_\rr \rho(\pi(\varphi_t(z))^\la dt= \int_0^{\pi} 
(\sin \theta)^{\la-1} |\eta|_{h_\rho}^{-\la}\frac{d\theta}{(1+Q)}. 
\]
Since 
$\lim_{\la\to 0}\Big(\int_{0}^\pi(\sin \theta)^{\la-1}\,d\theta
-2/\la\Big)=2\log 2$, we have
\[ 
\begin{split}
L_g(y_0,\delta^{-1}\om_0)&=\lim_{\la\to 0}\Big(\int_{0}^\pi (\sin
\theta)^{\la-1}  
\Big(\frac{|\eta|_{h_\rho}^{-\la}}{1+Q}\Big)d\theta-2/\la\Big)\\
&=\lim_{\la\to 0}\int_{0}^\pi (\sin \theta)^{\la-1}
\Big(\frac{|\eta|_{h_\rho}^{-\la}}{1+Q}-1\Big)
d\theta+2\log 2\\ 
&=\lim_{\la\to 0}\Big[\int_{0}^\pi \frac{(\sin \theta)^{\la-1}}{1+Q}
\Big(|\eta|_{h_\rho}^{-\la}-1\Big)d\theta
+\int_{0}^\pi(\sin\theta)^{\la-1}
\Big(\frac{1}{1+Q}-1\Big)d\theta\Big]+2\log 2\\  
&=\lim_{\la\to 0}\int_{0}^\pi \frac{(\sin \theta)^{\la-1}}{1+Q}
\Big(|\eta|_{h_\rho}^{-\la}-1\Big)d\theta
-\int_{0}^\pi \frac{Q}{(\sin\theta)(1+Q)}d\theta+2\log 2  
\end{split}
\]
If $f(\lambda,\theta)$ is a smooth function satisfying $f(0,\theta)=0$,
then the same argument as in the proof of Lemma~\ref{renormdist}
shows that
\[
\lim_{\la\to 0}\int_0^\pi (\sin\theta)^{\la -1} f(\la,\theta)\,d\theta
=\pl_\la f(0,0)+\pl_\la f(0,\pi).
\]
This can be used to evaluate the limit in the last line above, giving  
\begin{equation}\label{L0}
\begin{split}
L_g(y_0,\delta^{-1}\om_0)&= -\log(|\eta|_{h_\rho})|_{\theta=0}  
-\log(|\eta|_{h_\rho})|_{\theta=\pi}+2\log 2
-\delta\int_{0}^\pi \frac{\til{Q}}{(\sin\theta)(1+\delta\til{Q})}d\theta\\
&= 2\log 2\delta -\log(|\omega|_{h_\rho})|_{\theta=0}  
-\log(|\omega|_{h_\rho})|_{\theta=\pi}
-\delta\int_{0}^\pi \frac{\til{Q}}{(\sin\theta)(1+\delta\til{Q})}d\theta  
\end{split}
\end{equation}
where $\til{Q}=Q/\de=\frac{\sin\theta}{2|\om|^3_{h_\rho}}\pl_\rho 
h_\rho^{ij}\om_i\om_j$ as in the proof of Lemma~\ref{smallgeo}.  
Thus  
\[
L_g(y_0,\delta^{-1}\omega_0)= 2\log 2\delta-2\log|\omega_0|_{h_0} 
+O(\delta).
\]
Taking $\delta\to 0$, this shows that $L_g(y_0,\eta_0)$ determines
$|\omega_0|_{h_0}$, thus the metric $h_0$.   

Henceforth assume that $|\om_0|_{h_0}=1$.  Equation \eqref{L0} shows that
$L_g(y_0,\delta^{-1}\om_0)$ determines  
\begin{equation}\label{Fdef}
F(\de):=-\log(|\omega|_{h_\rho})|_{\theta=\pi} 
-\frac{\delta}{2}
\int_{0}^\pi \frac{\pl_\rho h^{ij}\om_i\om_j}{|\om|^3_{h_\rho}}
(1+\delta\til{Q})^{-1}d\theta.    
\end{equation}
Now $F(\de)$ is a smooth function of $\delta$ down to $\delta=0$ which
satisfies $F(\de)=\mc{O}(\de)$.  We show that the Taylor expansion of $F$
at $\de=0$ determines the Taylor expansion of $h_\rho$ at $\rho =0$.  
Denote $'=\pl_\de|_{\de=0}$, write $h(\rho,y)=h_\rho(y)$, 
and recall the solution 
$u(\theta)=(1-\cos\theta)\om_0^\sharp$, $\om(\theta)=\om_0$ for $\de=0$
derived in the proof of Lemma~\ref{smallgeo}.  Differentiation of
\eqref{Fdef} gives
\[
\begin{split}
F'(0)&=-\frac12 \Big(|\om|^2_{h_\rho}|_{\theta=\pi}\Big)'
-\frac{\pi}{2}\pl_\rho h(0,y_0)(\om_0,\om_0)\\
&=-\frac12 \Big(h^{ij}(0,y_0+\de u(\pi))\om_i(\pi)\om_j(\pi)\Big)' 
-\frac{\pi}{2}\pl_\rho h(0,y_0)(\om_0,\om_0)\\
&=-(\om_0)^k(\pl_{y^k}h)(0,y_0)(\om_0,\om_0)
-h(0,y_0)(\om(\pi)',\om_0)
-\frac{\pi}{2}\pl_\rho h(0,y_0)(\om_0,\om_0).
\end{split}
\]
It is clear that with the possible exception of $\om(\pi)'$,  
the first two terms are determined by $h_0$.  This is
the case for $\om(\pi)'$ as well:  
taking $\theta$ as independent variable, \eqref{intY} becomes   
\begin{equation}\label{intY2}
\frac{du^i}{d\theta}
=\sin\theta\,\frac{h^{ij}_\rho\om_j}{|\om|^2_{h_\rho}}
(1+\de\til{Q})^{-1}\qquad
\frac{d\om_i}{d\theta}=-\delta \sin\theta\,  
\frac{\pl_{y^i}h^{jk}_\rho\om_j\om_k}{2|\om|^2_{h_\rho}}
(1+\de\til{Q})^{-1}.
\end{equation}
The linearization of the second equation about $\de=0$ is
\[
\frac{d\om_i'}{d\theta}=- \tfrac12 (\sin\theta)
\pl_{y^i}h(0,y_0)(\om_0,\om_0),
\]
from which it is clear that $\om(\pi)'$ also is determined by $h_0$.  
Since we have already determined $h_0$, it follows that
$F'(0)$ determines $\pl_\rho h(0,y_0)(\om_0,\om_0)$ for all $y_0$, $\om_0$.
Thus $L_g$ determines $\pl_\rho h_\rho|_{\rho=0}$.    

We now claim by induction on $k$ that $L_g$ determines $\pl_\rho^k
h_\rho|_{\rho=0}$.  Apply $\pl_\de^k|_{\de=0}$ to \eqref{Fdef} and expand
using the chain rule (recall that $h$ and its derivatives are evaluated at 
$(\rho,y_0+\de u)$, and $\rho$ is determined implicitly as a function of
$(\theta,u,\om,y_0,\de)$ by $\rho |\om|_{h_\rho}=\de\sin\theta$).  The  
derivatives of $\om$ and $u$ which appear are $\pl_\de^l\om$ for  
$0\leq l\leq k$ and $\pl_\de^lu$ for $0\leq l\leq k-1$.
Differentiation of  
\eqref{intY2} shows that the pair $(\pl_\de^k u|_{\de=0}, \pl_\de^k
\om|_{\de=0})$ satisfies an inhomogeneous system of linear differential 
equations involving $\pl_\de^l
u|_{\de=0}$ and $\pl_\de^l \om|_{\de=0}$ for $0\leq l\leq k-1$;  
the equation for $\pl_\de^k u|_{\de=0}$ involves
$\pl_\rho^lh_\rho|_{\rho =0}$ for $0\leq l\leq k$ but the equation for  
$\pl_\de^k \om|_{\de=0}$ involves only  
$\pl_\rho^lh_\rho|_{\rho=0}$ for $0\leq l\leq k-1$ because of the leading
factor of $\delta$ in the equation for $\om$.  Thus the derivatives 
$\pl_\de^l\om$ for $0\leq l\leq k$ and $\pl_\de^lu$ for $0\leq l\leq k-1$
which appear in $\pl_\de^kF|_{\de=0}$ are all determined by 
$\pl_\rho^lh_\rho|_{\rho=0}$ for $0\leq l\leq k-1$.  The expansion of 
$\pl_\de^kF|_{\de=0}$ via the chain rule also has explicit dependence on
derivatives of $h_\rho$.  Since $\rho =0$ when $\theta=\pi$, only
tangential derivatives of $h_0$ appear when 
$\pl_\de^k \log(|\omega|_{h_\rho})|_{\theta=\pi}$ is expanded.  Hence  
$\pl_\de^k \log(|\omega|_{h_\rho})|_{\theta=\pi}$ is 
determined by $\pl_\rho^lh_\rho|_{\rho =0}$ for $0\leq l\leq k-1$.  On the
other hand, since $\pl_\de \rho|_{\de=0}=\sin\theta$, we have  
\[
\pl_\de^k\Big(\de \int_{0}^\pi \frac{\pl_\rho h^{ij}\om_i\om_j}{|\om|^3_{h_\rho}}
(1+\delta\til{Q})^{-1}d\theta\Big)
=k\pl_\rho^kh(0,y_0)(\om_0,\om_0)\int_0^\pi (\sin\theta)^{k-1}d\theta +
R_k, 
\]
where $R_k$ depends only on $\pl_\rho^lh_\rho|_{\rho =0}$ for $0\leq l\leq
k-1$.  Thus it follows by induction that $L_g$ determines
$\pl_\rho^kh(0,y_0)(\om_0,\om_0)$ for each 
$y_0$, $\om_0$, so also $\pl_\rho^kh_\rho|_{\rho=0}$.  
\stopthm

We remark that the determination of $h_\rho$ to infinite order in the proof
of Theorem~\ref{determination} above is constructive in the sense that 
it provides an algorithm for calculating the Taylor expansion in $\rho$.  
We also remark that it follows from the proof that $h=h'$ under the
hypotheses of Theorem~\ref{determination}.

\noindent
{\it Proof of Theorem~\ref{realanalytic}.}  
Let $\psi$ be a real-analytic diffeomorphism defined in a neighborhood of
$\pl\bbar{M}$, equal to the identity on $\pl\bbar{M}$, and for which 
\eqref{choicediffeo} holds.  Theorem~\ref{determination} shows that 
$h'_\rho -h_\rho$ vanishes to infinite order at $\rho=0$.  Real-analyticity
implies that $\psi^*g'=g$ near $\pl\bbar{M}$.  

We show that $\psi$ extends to an isometry 
$\psi:\bbar{M}\to\bbar{M}$.  If $\bbar{M}$ is simply connected, this is 
the standard result (Corollary 6.4, p. 256 of \cite{KN}) 
that a local isometry between simply connected, complete, real-analytic
Riemannian manifolds extends to a global isometry.  We claim that since our
$\psi$ is defined in a full neighborhood of $\pl\bbar{M}$, the same
argument applies under the weaker hypothesis
$\pi_1(\bbar{M},\pl\bbar{M})=0$.  The argument goes as follows.  Choose 
a point $p\in \pl\bbar{M}$.  If $q\in M$, completeness implies that $\psi$
can be extended by analytic continuation along any curve from $p$ to $q$.
It must be shown that the continuation in a neighborhood of $q$ is
independent of the curve.  If $\bbar{M}$ is 
simply connected, any closed curve based at $p$ can be deformed to the 
constant curve, and the result follows.  Under the hypothesis 
$\pi_1(\bbar{M},\pl\bbar{M})=0$, the closed curve can only 
be deformed into a closed curve in $\pl\bbar{M}$.  But since $\psi$ is
already defined in a full neighborhood of $\pl\bbar{M}$, this is sufficient
to conclude that the analytic continuation is path-independent.  $\psi$ is
a diffeomorphism since $\psi^{-1}$ extends by the same argument, and  
the relation $\psi^*g'=g$ follows by analytic continuation.    
\stopthm

\subsection{Deformation rigidity}

In this section we prove Theorem~\ref{deformation}.

\noindent
{\it Proof of Theorem~\ref{deformation}.}
First suppose (i) holds.  
There exists a smooth family $\phi(s):\bbar{M}\to 
  \bbar{M}$ of diffeomorphisms satisfying $\phi(s)|_{\pl\bbar{M}}={\rm Id}$
  such that   
$\phi(s)^*g(s)=(d\rho^2+h_\rho(s))/\rho^2$ in a product decomposition near
  $\pl \bbar{M}$, with $h_\rho(s)$ 
a smooth family of metrics on $\pl\bbar{M}$ satisfying $h_0(s)=h(s)$. Thus
we can assume that $g(s)$ is already reduced to that form.  
By Theorem~\ref{determination}, we also have  
$h_\rho(s)=h_\rho(0)+\mc{O}(\rho^\infty)$ uniformly in $s$, and in
particular $h(s)=h(0)$.  Thus the identifications    
of $\pl_\pm S^*M$ with $T^*\pl\bbar{M}$ agree for all $s$, so in the rest
of the proof we view the boundary as $\pl_\pm S^*M$ rather than
$T^*\pl\bbar{M}$.   

Fix $z\in \pl_-S^*M$; the geodesics for $g(s)$ with initial value 
$z\in \pl_-S^*M$ form a smooth in $s$ family of curves given by 
$\gamma_s(\tau,z):=\pi(\bbar{\varphi}_{\tau}(s,z))$ if
$\bbar{\varphi}_\tau(s,\cdot)$ is the flow  
of $\bbar{X}(s)=\rho^{-1}X(s)$ associated to $g(s)$.
Since $S_{g(s)}=S_{g(0)}$ by assumption, we have 
$z':=S_{g(0)}(z)=S_{g(s)}(z)\in \pl_+S^*M$.  We denote by 
$\tau_+(s,z)$ the time so that 
$\bbar{\varphi}_{\tau_+(s,z)}(s,z)=z'$, and 
we define $\gamma_s(\tau,z'):=\pi(\bbar{\varphi}_{-\tau}(s,z'))$ for  
$0\leq \tau\leq \tau_+(s,z)$.  
Since $\bbar{X}(s)=\bbar{X}(0)+\mc{O}(s\rho^\infty)$ when viewed as smooth vector fields 
on ${^bT}^*\bbar{M}$, we have for all $N\in\nn$ and $\tau$ small
\[\begin{gathered}
\bbar{\varphi}_{\tau}(s,z)=\bbar{\varphi}_{\tau}(0;z)+\mc{O}(s\max_{\sigma\leq
  \tau}\rho(\gamma_{s}(\sigma,z))^N), \\ 
\bbar{\varphi}_{-\tau}(s,z')=\bbar{\varphi}_{-\tau}(0;z')+\mc{O}(s\max_{\sigma\leq
  \tau}\rho(\gamma_s(\sigma,z'))^N),  
\end{gathered}\]
where  here the
remainder is uniform in $\tau$ (this follows for instance from the formula
in \cite[Lemma 2.2]{SUV} relating two flows in terms of the difference of
their vector fields).   
We write dot for the $\tau$ derivative and prime for the $s$ derivative at 
$s=0$, and we remove the $0$ index when  
$s=0$ (e.g. $g(0)$ is denoted $g$, $g'(0)$ is denoted $g'$, etc). We have
from the discussion above  
that the vector fields $\gamma'(\tau)$ and $\dot{\gamma}'(\tau)$ satisfy for all $N\in\nn$
\begin{equation}\label{fastvanish} 
\gamma'(\tau,z)=\mc{O}(\tau^N), \quad \dot{\gamma}'(\tau,z)=\mc{O}(\tau^N)
\end{equation}
for $\tau$ small, and the same holds with $z'$ replacing $z$.  If
$\bbar{g}(s):=\rho^2 g(s)$, remark that
$\bbar{g}(s)_{\gamma_s(\tau,z)}(\dot{\gamma}_s(\tau,z),\dot{\gamma}_s(\tau,z))=1$,
thus for $\eps\in (0,\tau_+(s,z))$ small, we get for ${\rm Re}(\la)>0$  
\[ 
\begin{split}
\pl_s\Big[\int_0^{\eps}\rho^{\la-1}&(\gamma_s(\tau,z))
\bbar{g}(s)_{\gamma_s(\tau,z)}(\dot{\gamma}_s(\tau,z),\dot{\gamma}_s(\tau,z))d\tau
\Big]\Big|_{s=0}\\
=&\int_0^\eps
\rho^{\la}(\gamma(\tau,z))\pl_{s}\Big[\rho^{-1}(\gamma_s(\tau,z)) 
\bbar{g}(\dot{\gamma}_s(\tau,z),\dot{\gamma}_s(\tau,z))\Big]\Big|_{s=0}d\tau\\
&+\la\int_{0}^{\eps}
\rho^{\la-2}(\gamma(\tau,z))d\rho(\gamma(\tau,z)).\gamma'(\tau,z)d\tau\\
&+\int_0^\eps  \rho^{\la-1}(\gamma(\tau,z))
\bbar{g}'(\dot{\gamma}(\tau,z),\dot{\gamma}(\tau,z))d\tau.  
\end{split}
\]
Due to \eqref{fastvanish} and $g'=\mc{O}(\rho^\infty)$, the three integrals
extend holomorphically near $\la=0$ and are uniformly $\mc{O}(\eps^{N})$
for all $N$ for $\la$ near $0$.  
Now the same arguments give the same identity with $z$ replaced by
$z'$. Finally we have 
\[
\begin{split}
\pl_s\Big[\int_\eps^{\tau_+(s,z)-\eps}\rho^{\la-1}(\gamma_s&(\tau,z))
\bbar{g}(s)_{\gamma_s(\tau,z)}(\dot{\gamma}_s(\tau,z),\dot{\gamma}_s(\tau,z))d\tau
\Big]\Big|_{s=0}\\
=&\pl_s\Big[\int_\eps^{\tau_+(s,z)-\eps}  \rho^{\la}(\gamma(\tau,z))
  \rho^{-1}(\gamma_s(\tau,z))
  \bbar{g}(\dot{\gamma}_s(\tau,z),\dot{\gamma}_s(\tau,z))d\tau\Big]\Big|_{s=0}\\
&+\int_\eps^{\tau_+(z)-\eps} \rho^{\la-1}(\gamma(\tau,z))
\bbar{g}'(\dot{\gamma}(\tau,z),\dot{\gamma}(\tau,z))d\tau\\
&+\la\int_\eps^{\tau_+(z)-\eps}\rho^{\la-2}(\gamma(\tau,z))d\rho(\gamma(\tau,z)).\gamma'(\tau,z)d\tau.   
\end{split}
\]
Summing up, and evaluating at $\la=0$, we obtain 
\begin{equation}\label{variationLgfinal}
\begin{split}
\pl_s [L_{g(s)}(z)]|_{s=0}=&
\int_\eps^{\tau_+(z)-\eps} \rho^{-1}(\gamma(\tau,z))
\bbar{g}'(\dot{\gamma}(\tau,z),\dot{\gamma}(\tau,z))d\tau\\ 
&+\pl_s\Big[\int_\eps^{\tau_+(s,z)-\eps}  \rho^{-1}(\gamma_s(\tau,z))
  \bbar{g}(\dot{\gamma}_s(\tau,z),\dot{\gamma}_s(\tau,z))d\tau\Big]\Big|_{s=0}+\mc{O}(\eps^N). 
\end{split}
\end{equation}
As $\eps\to 0$, the first term tends to $I_2(g')(z)$ where $I_2$ is the
X-ray transform on symmetric $2$-tensors associated to $g=g(0)$. 
Let $t(\tau,z)=\int_\eps^\tau \rho^{-1}(\gamma_s(\sigma,z))d\sigma$ and make the change of variable 
$\tau\mapsto t(\tau,z)$ in the second integral of \eqref{variationLgfinal},
which
becomes  \[\pl_s\Big[\int_{0}^{t_s(\eps)}g(\dot{\alpha}_s(t),\dot{\alpha}_s(t))dt\Big]\Big|_{s=0}\] 
where $\alpha_s(t):=\gamma_s(\tau,z)$ and
$t_s(\eps):=t(\tau_+(s,z)-\eps,z)$. We recognize the energy functional of
the curve $\alpha_s(t)$ for $g(0)$, and since $\alpha_0(t)$ is a geodesic
of $g(0)$,  
we get by \cite[Theorem 3.31]{GHL} and \eqref{fastvanish}  that 
\[\pl_s\Big[\int_{0}^{t_s(\eps)}g(\dot{\alpha}_s(t),\dot{\alpha}_s(t))dt\Big]\Big|_{s=0}=
g(\gamma'(\eps,z),\dot{\alpha}_0(t_0(\eps)))-g(\gamma'(\eps,z'),\dot{\alpha}_0(0))=\mc{O}(\eps^N)\]
for all $N\in\nn$.
Therefore, by letting $\eps\to 0$ in \eqref{variationLgfinal}, we conclude that 
\[\pl_s [L_{g(s)}(z)]|_{s=0}=I_2(g')(z).\]
 
By Theorem~\ref{injXray}, 
we deduce that $g'=Dq$ for some $q\in \rho^N L^\infty \cap 
C^\infty(M,T^*M)$ for all $N\geq 1$.  (The proof of Theorem~\ref{injXray}
shows that $q$ can be chosen to vanish to infinite order if $f=g'$ does.)   
The same holds by linearizing at any $s$ and by duality we obtain a vector
field $q^\sharp(s)\in C^\infty(M,TM)\cap \rho^NL^\infty$ so that  
$\pl_sg(s)=\mc{L}_{\demi q^\sharp(s)}g(s)$, whose dependence is smooth in  
$s$ since $q^\sharp(s)$ is the dual to
$q(s)={\pi_1}_*(R_{s,+}(0)\pl_sg(s))$ where $R_{s,+}(0)$ is the resolvent
of the vector field $X(s)$, which is smooth in $s$ when acting on
$\pi_2^*(\rho^\infty C^\infty(\bbar{M},\otimes_S^2T^*\bbar{M}))$ by the
expression \eqref{Rpm}. Integrating the vector field  
$\demi q^\sharp(s)$, we obtain a smooth $1$-parameter family of diffeomorphisms 
$\psi(s):\bbar{M}\to \bbar{M}$ 
equal to the identity at $\pl\bbar{M}$ so that $\psi(s)^*g(s)=g(0)$.  This
concludes the proof under assumption (i).  

Proposition~\ref{equivalencedglens} shows that (ii) implies (i), so the
result is true under assumption (ii) as well.  
\stopthm

\section{Simplicity and renormalized distance} 

In this section we parametrize geodesics on a non-trapping asymptotically
hyperbolic manifold by their two endpoints instead of their starting point
and direction.  To do this, we clearly need assumptions which imply at
least that there is a unique geodesic connecting any two points of the
boundary.  The assumption we make is 
that there are no conjugate points at infinity, i.e. there are no nonzero
Jacobi fields along any geodesic which decay as $t\to \infty$ and as $t\to
-\infty$.  We will call such asymptotically hyperbolic manifolds simple.
As we discuss below, this implies that $(M,g)$ has no conjugate points.    

Stable and unstable bundles for the geodesic flow on a complete manifold 
$(M,g)$ with no conjugate points are defined in \S 2 of \cite{Eb2}.   
An alternative statement of the condition that there are no nonzero
Jacobi fields decaying as $t\to \infty$ and as $t\to -\infty$ is that the  
stable and unstable bundles are everywhere transverse.  In \S 3 of
\cite{Eb2} it is proved that transversality of the stable and  
unstable bundles implies that the geodesic flow is Anosov 
in the case that the universal cover of $M$ is compactly homogeneous,
i.e. it can be 
covered by translates by isometries of a fixed compact set.  We first need 
to establish the analogous result for non-trapping asymptotically
hyperbolic manifolds.  An easy alternate construction of the stable and  
unstable bundles can be given in this setting using elementary ODE theory
which also proves that they are smooth.  In terms of this construction, it
is not difficult to establish directly that transversality implies that the
geodesic flow is hyperbolic.  
We begin this section by presenting these arguments.              
\begin{rem}
We typically use the terminology ``Anosov geodesic flow'' in classical
settings such as a compact manifold or the universal cover of  
a compact manifold, and ``hyperbolic geodesic flow'' otherwise, for
example for asymptotically hyperbolic manifolds.  But these 
mean the same thing: the existence of a splitting
\eqref{hyperbolicsplitting} with uniform contraction/expansion estimates
(for some $\nu>0$) as in Proposition~\ref{uniformbounds}.      
\end{rem}

We assume throughout this section that $(M,g)$ is a non-trapping
asymptotically hyperbolic manifold.  Let $z_0\in S^*M$.  Locally near
$z_0$, choose a smooth hypersurface  
$\mathcal{S}\subset S^*M$ transverse to $X$ with $z_0\in \mathcal{S}$. 
For $z\in \mathcal{S}$, choose an orthonormal basis
$\{w_1^z,\ldots,w_n^z\}$ for $\mc{Z}_z$ varying smoothly with $z$.  
For each $z$ and each $j$, extend $w_j^z$ to the geodesic $\gamma_z$ by
parallel translation.  This gives an orthonormal frame field $w_j^z(t)$ for 
$\dot{\gamma_z}(t)^\perp$, varying smoothly with
$(z,t)\in \mathcal{S}\times \rr$.  A normal vector field along   
$\gamma_z$ can be written $Y(t)=\sum_{j=1}^ny^j(t)w_j^z(t)$.  The Jacobi 
equation takes the form
\begin{equation}\label{Jacobi2}
\ddot{y}+R^z(t)y=0\qquad(\,\,\dot{}=\partial_t),
\end{equation}
where $y(t)=(y^1(t),\ldots,y^n(t))^T$ and $R^z(t)\in \rr^{n\times n}$
is the matrix of the linear transformation 
$Y\rightarrow \mc{R}(Y,\dot{\gamma}_z(t))\dot{\gamma}_z(t)$ in the frame     
$(w_1^z(t),\ldots,w_n^z(t))$.  Certainly $R^z(t)$ is smooth in $t$ and
$z$.  Using \eqref{curvature} and the arguments of the proof of
Lemma~\ref{smoothpb}, one sees that $R^z(t)$ has the form  
\[
R^z(t)=-I-S^z(t)  
\]
where $S^z(t)\in \rr^{n\times n}$ satisfies 
$|\partial_z^{\alpha}S^z(t)|\leq C_\alpha e^{-|t|}$, $t\in \rr$, 
with $C_\alpha$ independent of $z\in \mathcal{S}$ near $z_0$.  Here $\pl_z$ 
denotes partial differentiation with respect to some choice of local 
coordinates on $\mc{S}$.  

Reduce \eqref{Jacobi2} to a first order system in the usual way:
introduce 
\[
x=
\begin{pmatrix}
y\\
\dot{y}
\end{pmatrix},\quad
A=\begin{pmatrix}
0&I\\
I&0
\end{pmatrix},\quad
\widetilde{S}^z(t)= 
\begin{pmatrix}
0&0\\
S^z(t)&0 
\end{pmatrix}
\]
so that \eqref{Jacobi2} becomes
\begin{equation}\label{firstorder}
\dot{x}=\big(A+\widetilde{S}^z(t)\big)x.  
\end{equation}
Let $e_j$, $1\leq j\leq n$ denote the standard basis for $\rr^n$ and set 
\[
\eb_j^\pm=\begin{pmatrix}
e_j\\
\pm e_j
\end{pmatrix}.
\]
Then $\{\eb_j^-,\eb_j^+;1\leq j\leq n\}$ is an orthogonal basis for
$\rr^{2n}$ satisfying $A\eb_j^\pm = \pm \eb_j^\pm$.  
\begin{prop}\label{odetheorem}
For each $j$, $1\leq j\leq n$ and each choice of $\pm$, there is a solution
$x^z_{\pm,j}$ of \eqref{firstorder} satisfying 
\[
\lim_{t\rightarrow\infty}e^{\mp t}x^z_{\pm,j}(t)=\eb^\pm_j. 
\]
Moreover, $x^z_{-,j}$ is uniquely determined by this condition, and 
$x^z_{-,j}(t)$ is $C^\infty$ in $(z,t)$.  
\end{prop} 
The existence and uniqueness part of Proposition~\ref{odetheorem} 
is a special case of Problem 29, p. 104 of 
\cite{CL}, which applies to an ode system of the form \eqref{firstorder}
with $A$ diagonalizable and $E(t)$ integrable.  A argument similar to that
outlined in the solution to this problem in \cite{CL} proves smooth
dependence on the parameter $z$.  

Note that $\{x^z_{\pm,j}\}$ is a basis for the solutions of
\eqref{firstorder}.  Their first components thus form a basis for the
solutions of \eqref{Jacobi2}, and the corresponding 
$Y^z_{\pm,j}$ form a basis for the space of normal Jacobi fields with the
property that $Y^z_{\pm,j}$ is asymptotic to $e^{\pm t}w^z_j(t)$
as $t\rightarrow\infty$.  Under the isomorphism $\mc{L}$ defined in \eqref{L},
the stable bundle corresponds to initial data of the Jacobi fields which
decay like $e^{-t}$ as $t\rightarrow \infty$:  
\begin{defi}\label{stable}
The stable bundle $E_s$ is the subbundle of $\ker\alpha$ defined by    
\[
\mc{L}\big(E_s(\varphi_t(z))\big)=
\operatorname{span}\{\big(Y^z_{-,j}(t),(Y^z_{-,j})'(t)\big): 1\leq j\leq
n\}, \quad t\in \mathbb{R}.
\]
\end{defi}

\noindent 
Here $' = D_t$.  
When viewed as a function of the point on a geodesic rather than on the
time parameter, the decaying Jacobi fields are independent of which 
point is labeled as the initial point.  Consequently $E_s$ is well-defined
independent of the choice of initial point $z$ on the geodesic.
Since the $Y^z_{-,j}(t)$ are smooth functions of $(z,t)$, $E_s$ is a smooth
subbundle of $TS^*M$.  The unstable bundle $E_u$ 
is defined analogously in terms of the initial data of the Jacobi fields
which decay as $t\rightarrow -\infty$.  The bundles $E_s$ and $E_u$ are
invariant under the geodesic flow since Jacobi fields along a geodesic in
$M$ correspond to flow-invariant vector fields along the corresponding
integral curve in $S^*M$.  

\begin{defi}\label{simpleuptobdry}
A non-trapping asymptotically hyperbolic manifold $(M,g)$ is said to be
\emph{simple} if $E_s(z)\cap E_u(z)=\{0\}$ for each $z\in S^*M$.      
\end{defi}

\noindent
A non-trapping asymptotically hyperbolic manifold with non-positive
curvature is simple.  That there are
no solutions to \eqref{Jacobi2} which decay exponentially as $t\to \pm
\infty$ follows by taking the inner product with $y(t)$ and integrating by
parts.    

When $E_s(z)\cap E_u(z)=\{0\}$ for all $z\in S^*M$, we have a hyperbolic
splitting for the flow: 
\begin{equation}\label{hyperbolicsplitting}
T(S^*M)=\rr X\oplus E_s\oplus E_u. 
\end{equation}
The next proposition asserts that the decay estimates are uniform in this
case.     
\begin{prop}\label{uniformbounds}
Let $(M,g)$ be a non-trapping asymptotically hyperbolic manifold such that 
$E_s(z)\cap E_u(z)=\{0\}$ for each $z\in S^*M$.  Then its 
geodesic flow is hyperbolic in the following sense:  for 
any $0<\nu<1$, there exists a constant $C>0$ so that:   
\begin{enumerate}
\item
If $z\in S^*M$ and $\zeta\in E_s(z)$, then 
\[
\| d\varphi_t(z).\zeta\|_G\leq Ce^{-\nu t}\|\zeta\|_G,\enspace t\geq 0\quad
\text{and}\quad \| d\varphi_t(z).\zeta\|_G\geq C^{-1}e^{-\nu t}\|\zeta\|_G,\enspace
t\leq 0  
\]
\item
If $z\in S^*M$ and $\zeta\in E_u(z)$, then 
\[
\| d\varphi_t(z).\zeta\|_G\leq Ce^{\nu t}\|\zeta\|_G,\enspace t\leq 0\quad
\text{and} \quad\| d\varphi_t(z).\zeta\|_G\geq C^{-1}e^{\nu t}\|\zeta\|_G,\enspace 
t\geq 0  
\]
\end{enumerate}
\end{prop}

\noindent
The proof of Proposition~\ref{uniformbounds} will be given after
Proposition~\ref{nearboundary} below.   

It is a consequence of Proposition~\ref{uniformbounds} and the following
result of Gerhard Knieper that simple asymptotically hyperbolic manifolds
have no conjugate points.   
\begin{prop}\label{knieper}\cite{Kn} 
Let $(M,g)$ be a complete connected non-compact Riemannian manifold with
sectional curvature bounded below by $-\beta^2$ and with hyperbolic
geodesic flow with constants $C$, $\nu$ (as in the statement of  
Proposition~\ref{uniformbounds}).  There is a constant
$\sigma(\beta,\nu,C)>0$ so that if $(M,g)$ satisfies the following three 
conditions:
\begin{itemize}
\item[(i)] For any $z\in S^*M$, there is an open neighborhood
$\mathcal{U}\subset S^*M$ of $z$ such that 
$\lim_{t\to \infty} d_g(\gamma(0),\gamma(t))=\infty$ uniformly for all
  geodesics $\gamma$ with $\dot{\gamma}(0)\in \mathcal{U}$.  
\item[(ii)] There exists a compact set $K\subset M$ such that for
all $p\in M\setminus K$ and for all geodesics $\gamma$ with $\gamma(0)=p$,   
$\gamma(t)|_{t\in [-1,\sigma]}$ has no conjugate points   
\item[(iii)] $(M,g)$ has at least one geodesic without conjugate points, 
\end{itemize}

\noindent
then $(M,g)$ has no conjugate points.
\end{prop}

\noindent
It is not hard to verify that (i) and (ii) hold for any nontrapping
asymptotically hyperbolic manifold, using Lemma~\ref{convtobdry} and the
fact that there are no conjugate points on any geodesic segment
sufficiently near infinity, where the curvature is negative.  Condition
(iii) also holds since the short geodesics described in
Lemma~\ref{smallgeo} have no conjugate points.      

Proposition~\ref{knieper} is not necessary for the purposes of this paper:  
if preferred, one can just add the assumption that there are no conjugate
points in the definition of a simple asymptotically hyperbolic manifold.   

We will prove Proposition~\ref{uniformbounds} by dividing the set of all  
orbits of the geodesic flow (i.e. lifted unparametrized geodesics) into two
subsets, one a compact  
set of orbits, and the other consisting of orbits, each of whose
projection to $M$ stays in a fixed small neighborhood of $\pl M$ (short
geodesics).  A different argument is used to establish the bounds when
$z$ lies in either of the two sets of orbits.  

We begin by establishing uniform bounds for any compact set of orbits.  
This is done by deriving uniform bounds locally in the set of
orbits, and for this case we obtain the estimates (1), (2) with $\nu=1$.   
Unless explicitly stated otherwise, $(M,g)$ is assumed only to be 
non-trapping.   

Let $z_0\in S^*M$, choose a transverse hypersurface 
$\mc{S}\subset S^*M$, and rewrite the equation for normal Jacobi fields as a   
$\mathbb{R}^{2n}$-valued first order system via a choice of parallel orthonormal
frame as above.  We have the following two lemmas, the first asserting
uniform upper bounds and the second uniform lower bounds on solutions.  

\begin{lemm}\label{upper}
There is a constant $K>0$ independent of $z\in \mc{S}$ near $z_0$, so that   
for all $t\in \mathbb{R}$ and $1\leq j\leq n$: 
\[
|x_{-,j}^z(t)|\leq K e^{-t},\qquad |x_{+,j}^z(t)|\leq K e^{|t|}.
\]
\end{lemm}

\begin{lemm}\label{lower}
There is a constant $k$ independent of $z\in \mc{S}$ near $z_0$ such that
if $\la=(\la_1,\ldots, \la_n)^T\in \mathbb{R}^n$, $\mu=(\mu_1,\ldots,
\mu_n)^T\in \mathbb{R}^n$ and $t\in \mathbb{R}$, then   
\[
\Big|\sum_{j=1}^n \left(\la_j x_{+,j}^z(t)+\mu_jx_{-,j}^z(t)\right)\Big|
\geq k|(\la,\mu)|e^{-|t|}
\]
and
\[
\Big|\sum_{j=1}^n \la_j x_{+,j}^z(t)\Big|\geq k|\la|e^{t}.
\]
If $E_s(z)\cap E_u(z)=\{0\}$ for $z\in \mc{S}$ near $z_0$, then also  
\begin{equation}\label{anosovlower}
\Big|\sum_{j=1}^n\mu_j x_{-,j}^z(t)\Big|\geq k|\mu|e^{-t}.
\end{equation}
\end{lemm}

The construction of the solutions $x^z_{\pm,j}$ outlined in \cite{CL} gives
explicit estimates for $t$ large.  These estimates can be extended to all
$t\in \rr$ using Gronwall's inequality.  The arguments are standard, so the  
proofs of Lemmas~\ref{upper} and \ref{lower} are omitted.   

\begin{lemm}\label{upperbound}
Suppose that $E_s(z)\cap E_u(z)=\{0\}$ for $z\in \mc{S}$ near $z_0$.  There
is a constant $C$ independent of $z\in \mc{S}$ near $z_0$ so that if  
$x(t)=\sum_{j=1}^n \mu_j x_{-,j}^z(t)$ with $\mu=(\mu_1,\ldots, \mu_n)^T\in
\rr^n$,  
then 
\[
|x(t+s)|\leq Ce^{-t}|x(s)|,\qquad s,t\in \rr.
\] 
\end{lemm}
\begin{proof}
The first estimate of Lemma~\ref{upper} implies 
\[
|x(t)|\leq K\sqrt{n}|\mu|e^{-t}.
\]
Applying this with $t$ replaced by $t+s$ and then \eqref{anosovlower}
with $t$ replaced by $s$ gives 
\[
|x(t+s)|\leq K\sqrt{n}|\mu|e^{-t}e^{-s}\leq
\frac{K\sqrt{n}}{k}e^{-t}|x(s)|.  
\]
\end{proof}

\noindent
Lemma~\ref{upperbound} can be reformulated and extended as follows.   

\begin{prop}\label{compactset}
Suppose that $E_s(z)\cap E_u(z)=\{0\}$ for $z\in S^*M$.   
Let $\mathcal{C}$ be a compact set of orbits of the geodesic flow on $S^*M$.
There exists a constant $C>0$ so that if $z\in o$ for some 
$o\in \mathcal{C}$ and $\zeta\in E_s(z)$, then 
\[
\| d\varphi_t(z).\zeta\|_G\leq Ce^{-t}\|\zeta\|_G,\qquad t\geq 0.
\]
\end{prop}

\noindent
Recall that each orbit intersects $\pl_-S^*M$ exactly once.  So we 
identify the space of orbits with $\pl_-S^*M$, with the induced  
topology.  
\begin{proof}
First note that Lemma~\ref{upperbound} gives a uniform bound for all
points in the orbit $o$ through $z$.  Namely, $x(0)$ corresponds to an  
element of $E_s(z)$, $x(s)$ corresponds to an element 
$\zeta\in E_s(\varphi_s(z))$, and $x(t+s)$ corresponds to 
$d\varphi_t(\varphi_s(z)).\zeta$.  As $s$ varies over $\rr$, $\varphi_s(z)$
varies over all points on $o$, and as $\mu$ varies over $\rr^n$, $\zeta$ varies
over all elements of $E_s(\varphi_s(z))$.  So, restricting to $t\geq 0$,
Lemma~\ref{upperbound} asserts that there is a uniform upper bound
locally in a neighborhood of any 
orbit in the space of orbits.  Hence there is a uniform bound on any
compact subset of the space of orbits.   
\end{proof}

Geodesics which stay in a set where the curvature is negative are
hyperbolic.  The 
following proposition is a special case of \cite[Theorem 3.2.17]{Kl2}.    
\begin{prop}\label{nearboundary}
Suppose $(M,g)$ is asymptotically hyperbolic.  Let $0<\epsilon<1$.  
If $z\in S^*M$ and the geodesic 
$\{\pi\big(\varphi_t(z)\big):t\in \rr\}$ is contained in the set where all  
sectional curvatures $K$ satisfy 
\begin{equation}\label{epcurv}
-(1+\eps)^2\leq K\leq -(1-\eps)^2,
\end{equation}
then    
\[
\| d\varphi_t(z).\zeta\|_G\leq 
Ce^{-(1-\eps)t}\|\zeta\|_G,\qquad 
t\geq 0,\qquad  \zeta\in E_s(z)
\]
with $C=\frac{1+\eps}{1-\eps}$.
\end{prop}

\noindent
{\it Proof of Proposition~\ref{uniformbounds}.}
Given $0<\nu<1$, set $\eps=1-\nu$.  The set of orbits which lie in 
the region where 
\eqref{epcurv} holds is contained in the complement of a compact set 
$\mathcal{C}$ of orbits.  So the inequality for $t\geq 0$ in (1) follows  
immediately from Propositions~\ref{compactset} and \ref{nearboundary}.  

The inequality for $t\leq 0$ in (1) follows from the 
inequality for $t\geq 0$ upon replacing $\zeta$ by
$d\varphi_{-t}(\varphi_t(z)).\zeta$ in the inequality for $t\geq 0$, and then
replacing $t$ by $-t$.     

The inequalities in (2) follow from those in (1) upon noting that the 
stable and unstable spaces are related by $E_u(z)=dS\big(E_s(Sz)\big)$,
where $S:S^*M\rightarrow S^*M$ is the involution $S(x,\xi)=(x,-\xi)$.  
\stopthm

\begin{prop}\label{uniquegeodesics}
Let $(M,g)$ be a simple asymptotically hyperbolic manifold.  If $p,q\in
\bbar{M}$, $p\neq q$, there is a unique geodesic 
(viewed as an unparametrized curve) connecting $p$ and $q$.  In the case
that $p$ and/or $q$ is in $\pl\bbar{M}$, this is interpreted to mean that
the geodesic approaches the point as $t\to \pm \infty$.    
\end{prop}
\begin{proof}
If both $p$ and $q\in M$, this follows from the fact that the exponential
map is a diffeomorphism at each point.  If one or both of $p$, 
$q$ is in $\pl \bbar{M}$, 
it follows from results in \cite{Kl1}, \cite{Eb1} as follows.  First, the
conformal compactification $\bbar{M}$ agrees with the compactification
used in \cite{Kl1}, \cite{Eb1} where the boundary at infinity consists of 
equivalence classes of geodesics that are asymptotic; two oriented 
geodesics $\gamma_1$ and $\gamma_2$ are said to be asymptotic 
if for any given unit speed parametrization $\gamma_i(t)$ of $\gamma_i$
there is $C>0$ such that $d_g(\gamma_1(t),\gamma_2(t))\leq C$ for all
$t\geq 0$.  In fact, it is easily seen that this notion of asymptotic is
equivalent to the condition that the two geodesics have the same endpoint
in $\bbar{M}$ as $t\to \infty$.  The appendix to \cite{Kl1} is formulated
for the 
universal cover of a compact manifold with Anosov geodesic flow, but the
arguments apply to any complete, simply connected manifold with no
conjugate points for which the geodesic flow is Anosov and for which there
is a uniform lower bound on sectional curvatures.  This appendix  
proves, first, that any such manifold satisfies the uniform visibility
axiom of \cite{Eb1}.  Now Proposition 1.5 of \cite{Eb1} asserts that any
complete, simply connected manifold with no conjugate points satisfying the
uniform visibility axiom has the property that there exists a unique 
geodesic connecting a point of $M$ and a point of $\pl \bbar{M}$, and  
Proposition 1.7 of 
\cite{Eb1} asserts that there exists a geodesic connecting any two distinct   
points of $\pl\bbar{M}$.  Moreover, the appendix to \cite{Kl1} additionally
proves that there exists a unique geodesic connecting any two distinct
points of $\pl\bbar{M}$.     
\end{proof}

We next study boundary mapping properties of the extended 
exponential map for simple asymptotically hyperbolic 
manifolds.  The behavior of $E_s$ and $E_u$ near $\partial \bbar{S^*M}$ 
plays an important role.  
\begin{prop}\label{bundlesbdry}
On a non-trapping asymptotically hyperbolic manifold, each of the
subbundles $E_s$ and $E_u$ of $TS^*M$ extends smoothly to 
a subbundle of $T\bbar{S^*M}$.  
Moreover, the extensions satisfy $\mc{V}_z=E_s(z)$ for $z\in
\partial_+ S^*M$ and $\mc{V}_z=E_u(z)$ for $z\in \partial_- S^*M$.      
\end{prop}
\begin{proof}
Let $z_0\in S^*M$ and choose a basis  
$\{\zeta_j(z); 1\leq j\leq n\}$ for $E_s(z)$ for $z$ near $z_0$,   
depending smoothly on $z$.  Since $E_s$ is invariant under the geodesic
flow, it follows that $\{d\varphi_t(z).\zeta_j(z); 1\leq j\leq n\}$ is a   
basis for $E_s(\varphi_t(z))$ for all $t\in \mathbb{R}$.  Now write 
$\bbar{\varphi}_\tau(z)=\varphi_{t(\tau,z)}(z)$ as in \eqref{reparam}.  We
will show that $d\varphi_{t(\tau,z)}(z).\zeta_j(z)$, $1\leq j\leq n$,
extend smoothly up to $\tau=\tau_+(z)$ as a function of $(\tau,z)$ for $z$
near $z_0$, and that they remain linearly independent.  {From} this it
follows that the bundle $E_s$ extends smoothly to $\pl_+S^*M$.    

The chain rule gives
\begin{equation}\label{chainrule}
\begin{split}
d\varphi_{t(\tau,z)}(z).\zeta 
&=d \bbar{\varphi}_\tau(z).\zeta -
X(\bbar{\varphi}_\tau(z))d_zt(\tau,z).\zeta\\
&=d \bbar{\varphi}_\tau(z).\zeta -
\rho(\bbar{\varphi}_\tau(z))\bbar{X}(\bbar{\varphi}_\tau(z))d_zt(\tau,z).\zeta 
\end{split}
\end{equation}
Certainly $d \bbar{\varphi}_\tau(z).\zeta_j(z)$ and 
$\bbar{X}(\bbar{\varphi}_\tau(z))$ extend smoothly up to $\tau=\tau_+(z)$.   
Differentiation of \eqref{taupm} shows that  
\[
d_zt(\tau,z)=\eta(\tau,z)\big(\tau-\tau_+(z)\big)^{-1} 
\]
where $\eta(\tau,z)$ is smooth up to $\tau=\tau_+(z)$.  But 
$\rho(\bbar{\varphi}_\tau(z))$ is smooth and vanishes when
$\tau=\tau_+(z)$.  Thus 
$\rho(\bbar{\varphi}_\tau(z))d_zt(\tau,z).\zeta_j(z)$ is smooth up to 
$\tau=\tau_+(z)$, and so also is $d\varphi_{t(\tau,z)}(z).\zeta_j(z)$. 

Now $d\bbar{\varphi}_\tau(z)$ is an isomorphism up to $\tau=\tau_+(z)$
since $\bbar{\varphi}_\tau (z)$ is a smooth flow.  Since
$\{\zeta_1(z),\ldots,\zeta_n(z),\bbar{X}(z)\}$ is a linearly independent
set of vectors at $z$, it follows that 
$\{d\bbar{\varphi}_{\tau_+(z)}(z).\zeta_1(z),\ldots,
d\bbar{\varphi}_{\tau_+(z)}(z).\zeta_n(z),
d\bbar{\varphi}_{\tau_+(z)}(z).\bbar{X}(z)\}$   
is a linearly independent set of vectors at
$\bbar{\varphi}_{\tau_+(z)}(z)$.  But 
$d\bbar{\varphi}_{\tau_+(z)}(z).\bbar{X}(z)=
\bbar{X}(\bbar{\varphi}_{\tau_+(z)}(z))$ 
since $\bbar{\varphi}_\tau$ is the flow of $\bbar{X}$.  Therefore 
$\{d\bbar{\varphi}_{\tau_+(z)}(z).\zeta_j(z)+\alpha_j\bbar{X}(\bbar{\varphi}_{\tau_+(z)}(z)); 
1\leq j\leq n\}$ 
is linearly independent for any $\alpha_1,\ldots,\alpha_n\in \mathbb{R}$. 
It thus follows from \eqref{chainrule} that  
$\{d\varphi_{t(\tau,z)}(z).\zeta_j(z);1\leq j\leq n\}$ remains linearly
independent up to $\tau=\tau_+(z)$ as claimed.

The same argument letting $\tau\rightarrow -\tau_-(z)$ shows that $E_s$ 
extends smoothly to $\pl_-S^*M$.  Then the same argument, taking the
$\zeta_j(z)$ to be a basis for $E_u(z)$, shows that $E_u$ extends smoothly
to $\bbar{S^*M}$.  

Near the boundary, we use the smooth coordinates \eqref{newcoords} for 
${}^bT^*\bbar{M}$. So $\zeta\in T{}(^bT^*\bbar{M})$ can be written   
\[
\zeta = \bbar{a}^0\pl_{\bbar{\rho}} +\sum_j\bbar{a}^j\pl_{\bbar{y}^j}
+\bbar{b}_0\pl_{\bbar{\xi}_0}+\sum_j\bbar{b}_j\pl_{\bbar{\eta}_j}. 
\]
The
$(\bbar{\rho},\bbar{y}^j,\bbar{\xi}_0,\bbar{\eta}_j,\bbar{a}^0,\bbar{a}^j,\bbar{b}_0,\bbar{b}_j)$ 
are smooth coordinates for $T({}^bT^*\bbar{M})$ near $\rho=0$.    
The fiber of $T\bbar{S^*M}\subset T({}^bT^*\bbar{M})$ is given by 
$\bbar{b}_0=0$ over $\partial \bbar{S^*M}$.  
The condition $\alpha(\zeta)=0$ reads $\bbar{\xi}_0\bbar{a}^0 = -\rho
\bbar{\eta}_j\bbar{a}^j$, so the fiber of the smooth  
extension of $\ker \alpha$ is given by $\bbar{a}^0=0$ over 
$\partial \bbar{S^*M}$.  
The fiber of the vertical bundle of $T({}^bT^*\bbar{M})$ is given by
$\bbar{a}^0=0$, $\bbar{a}^j=0$, and to obtain the fiber of $\mc{V}\subset 
T\bbar{S^*M}$ one adds the condition $\bbar{\xi}_0\bbar{b}_0=-\rho^2 
\sum_{ij}h^{ij}\bbar{\eta}_i\bbar{b}_j$ of tangency to $\bbar{S^*M}$.  In
particular, the fiber of $\mc{V}$ over $\partial \bbar{S^*M}$  
is given by $\bbar{a}^0=0$, $\bbar{a}^j=0$, $\bbar{b}_0=0$.   

To identify the fiber of $E_s$ over $\pl_+ S^*M$, we consider the 
asymptotics of the Sasaki metric.
Let $\sigma(\tau)$ be a smooth curve in $\bbar{S^*M}$ with
$\sigma(\tau_0)\in\pl_+ S^*M$ such that $\sigma$ is transverse to $\pl_+
S^*M$ at $\sigma(\tau_0)$, and let $\zeta(\tau)$ be a smooth
section of $\ker\alpha\cap T\bbar{S^*M}$ along $\sigma(\tau)$ with
$\zeta(\tau)\to \zeta(\tau_0)\in T_{\sigma(\tau_0)}\bbar{S^*M}$.  We claim
that if $\zeta(\tau_0)\notin \mc{V}_{\sigma(\tau_0)}$, then there is $C>0$
so that $\|\zeta(\tau)\|_G \geq C\rho(\sigma(\tau))^{-1}$ for $\tau$
sufficiently close to $\tau_0$.  In fact, the coordinates of
$\zeta(\tau_0)$ must satisfy $\bbar{a}^0=\bbar{b}_0=0$ since 
$\zeta(\tau)\in \ker\alpha\cap T\bbar{S^*M}$.  Since 
$\zeta(\tau_0)\notin \mc{V}_{\sigma(\tau_0)}$, it must be that
$\bbar{a}^j(\tau_0)\neq 0$ for some $j$.  But 
\[
G(\zeta,\zeta)\geq
g(d\pi(\zeta),d\pi(\zeta) )
=\rho^{-2}\big[(\bbar{a}^0)^2+h_{ij}\bbar{a}^i\bbar{a}^j\big]
\geq \rho^{-2}\sum_{ij}h_{ij}\bbar{a}^i\bbar{a}^j,
\]
so the claimed inequality follows upon taking $\zeta=\zeta(\tau)$ with
$\tau$ close to $\tau_0$.  

To apply this observation to $E_s$, choose $z\in S^*M$ and take 
$\sigma(\tau)=\bbar{\varphi}_\tau(z)=\varphi_{t(\tau,z)}(z)$, so 
$\tau_0=\tau_+(z)$.  Choose a basis $\zeta_j(z)$, $1\leq j\leq n$ for
$E_s(z)$ and take $\zeta(\tau)=d\varphi_{t(\tau,z)}(z).\zeta_j(z)$ for some
$j$.  Equation \eqref{uniformrho} together with Lemma~\ref{upper} show
that $\|\zeta(\tau)\|_G\leq Ce^{-t}\leq C\rho$.  The above observation 
implies that it must be that $\zeta(\tau_0)\in \mc{V}_{B_+(z)}$.     
But the argument earlier in this proof establishing the smoothness of $E_s$
up to the boundary shows that the $\zeta(\tau_+(z))$ obtained by varying
$j$ form a basis for $E_s(B_+(z))$.  Thus the fibers of $E_s$ and $\mc{V}$
coincide on $\pl_+ S^*M$.  The same argument applies to $E_u$ at $\pl_-
S^*M$.  
\end{proof}
\begin{rem}
By analyzing the behavior of the connection map $\mc{K}$ near 
$\partial \bbar{S^*M}$, it can be shown that also the horizontal bundle
$\mc{H}$ extends smoothly to a subbundle of $T\bbar{S^*M}$, and 
$\mc{H}_z=\mc{V}_z$ for $z\in \partial \bbar{S^*M}$.  The details are
omitted since we will not use this fact.    
\end{rem}

Set
\[
\mc{T}^+\bbar{S^*M}=\{(\tau,z)\in \rr\times \bbar{S^*M}: 0<\tau\leq
\tau_+(z)\}.
\]
\begin{prop}\label{Ediff}
If $(M,g)$ is simple, then the following map is a diffeomorphism
\[
\Psi:\mc{T}^+\bbar{S^*M}\to \bbar{M}\times\bbar{M}\setminus
{\rm diag} ,\qquad
\Psi(\tau,z)=\big(\pi(z),\pi(\bbar{\varphi}_\tau(z))\big). 
\]
\end{prop}
\begin{proof}
The bijectivity of $\Psi$ is an immediate consequence of
Proposition~\ref{uniquegeodesics}.  Namely, given 
$(p,q)\in \bbar{M}\times\bbar{M}\setminus \Delta$, there is a unique
geodesic connecting $p$ and $q$; $z$ corresponds to the initial point and 
direction, and $\tau$ the time when the geodesic is parametrized as
$\tau\to\pi(\bbar{\varphi}_\tau(z))$.  

The map $\Psi$ is clearly smooth, so it suffices to show that $d\Psi$ 
is injective at each point.  Write 
\[
\Psi(\tau,z)=\big(\Psi_1(\tau,z),\Psi_2(\tau,z)\big), 
\qquad \Psi_1(\tau,z)=\pi(z),\quad 
\Psi_2(\tau,z)=\pi(\bbar{\varphi}_\tau(z))
\]
so that $\ker d\Psi=\ker d\Psi_1\cap \ker d\Psi_2\subset \rr\pl_\tau \oplus 
T_z\bbar{S^*M}$.  Clearly $\ker d\Psi_1 = \rr\pl_\tau \oplus \mc{V}_z$.   
For $a\in \rr$ and $\zeta \in T_z\bbar{S^*M}$, we have
\[
d\Psi_2(a\pl_\tau +\zeta)=d\pi\big(d\bbar{\varphi}_\tau(z).\zeta
+a\bbar{X}(\bbar{\varphi}_\tau(z))\big).
\]
So injectivity of $d\Psi(\tau,z)$ is equivalent to the statement that if
$\zeta\in \mc{V}_z$ and 
$d\bbar{\varphi}_\tau(z).\zeta\in
\mc{V}_{\bbar{\varphi}_\tau(z)}\oplus\rr\bbar{X}(\bbar{\varphi}_\tau(z))$, 
then $\zeta=0$.  

If both $\pi(z)$ and $\pi(\bbar{\varphi}_\tau(z))$ are
in $M$, this follows from the fact that there are no conjugate points:   
by \eqref{chainrule}, we deduce that 
$d\varphi_{t(\tau,z)}(z).\zeta \in
\mc{V}_{\bbar{\varphi}_\tau(z)}\oplus\rr\bbar{X}(\bbar{\varphi}_\tau(z))$.
But $d\varphi_{t(\tau,z)}(z).\zeta$ is already in $\ker\alpha$, so 
$d\varphi_{t(\tau,z)}(z).\zeta \in \mc{V}_{\bbar{\varphi}_\tau(z)}$ and 
\eqref{noconjugatepoints} implies that $\zeta=0$.  

Suppose next that $\pi(z)\in \pl \bbar{M}$ (so $z\in \pl_-S^*M$) 
and $\pi(\bbar{\varphi}_\tau(z))\in M$.  If $\zeta\in \mc{V}_z$, then 
$\zeta\in E_u(z)$ by Proposition~\ref{bundlesbdry}.  Since 
$E_u\oplus \rr\bbar{X}$ is invariant under the flow $\bbar{\varphi}_\tau$, 
it follows that $d\bbar{\varphi}_\tau(z).\zeta   
\in E_u(\bbar{\varphi}_\tau(z))\oplus
\rr\bbar{X}(\bbar{\varphi}_\tau(z))$.   Proposition 2.11 of \cite{Eb2}
implies that if $(M,g)$ is complete with no conjugate points and sectional 
curvatures bounded from below, then $E_u$
and $E_s$ each intersect $\mc{V}$ only in $\{0\}$.  So if also 
$d\bbar{\varphi}_\tau(z).\zeta \in \mc{V}_{\bbar{\varphi}_\tau(z)}\oplus 
\rr\bbar{X}(\bbar{\varphi}_\tau(z))$, then 
$d\bbar{\varphi}_\tau(z).\zeta \in \rr\bbar{X}(\bbar{\varphi}_\tau(z))$.
Since $\bbar{X}$ is invariant under the flow $\bbar{\varphi}_{\tau}$, it
follows that $\zeta\in \rr\bbar{X}(z)\cap \mc{V}_z=\{0\}$, so $\zeta =0$ as
desired.  The argument if $\pi(z)\in M$ and
$\pi(\bbar{\varphi}_\tau(z))\in \pl\bbar{M}$ (so
$\bbar{\varphi}_\tau(z)\in \pl_+S^*M$) is similar.

The argument if $z\in \pl_-S^*M$ and
$\bbar{\varphi}_\tau(z)\in \pl_+S^*M$ follows the same idea.  In 
this case $\zeta\in E_u(z)$, so $d\bbar{\varphi}_\tau(z).\zeta   
\in E_u(\bbar{\varphi}_\tau(z))\oplus
\rr\bbar{X}(\bbar{\varphi}_\tau(z))$.  If also 
$d\bbar{\varphi}_\tau(z).\zeta \in \mc{V}_{\bbar{\varphi}_\tau(z)}\oplus 
\rr\bbar{X}(\bbar{\varphi}_\tau(z))=E_s({\bbar{\varphi}_\tau(z)})\oplus  
\rr\bbar{X}(\bbar{\varphi}_\tau(z))$, then 
$d\bbar{\varphi}_\tau(z).\zeta \in \rr\bbar{X}(\bbar{\varphi}_\tau(z))$
since $E_u\cap E_s=\{0\}$ (this transversality
holds at $\pl\bbar{S^*M}$ too by Proposition~\ref{bundlesbdry}  
and invariance under the geodesic flow).  
Once again, translating back to $z$ shows that $\zeta\in \rr\bbar{X}(z)$,
so $\zeta =0$.  
\end{proof}

\begin{rem}
Simplicity is a necessary condition for $\Psi$ to be a
local diffeomorphism everywhere on $\bbar{M}\times\bbar{M}\setminus {\rm 
  diag}$.  If $(M,g)$ is non-trapping, $z\in\pl_-S^*M$, and 
$E_s\cap E_u\neq \{0\}$ along the integral curve $\bbar{\varphi}_\tau(z)$, 
then $\Psi$ is not a local diffeomorphism near
$(\tau_+(z),z)$.  In fact, as the proof above shows, for any $\zeta\in 
E_s(z)\cap E_u(z)$, there is $a_\zeta\in \rr$ so that
$d\Psi(a_\zeta\pl_\tau + \zeta)=0$.  
\end{rem}
\begin{rem}
The analogue of Proposition~\ref{Ediff} obtained by replacing 
$\mc{T}^+\bbar{S^*M}$ by the set
$\mc{T}^-\bbar{S^*M}=\{(\tau,z)\in \rr\times \bbar{S^*M}: 
-\tau_-(z)\leq \tau<0\}$ is also true, with the same proof.
\end{rem}

\begin{corr}\label{mapB}
If $(M,g)$ is simple, the following map is a diffeomorphism 
\[ 
B: \pl_-S^*M \to  \pl \bbar{M}\x \pl \bbar{M}\setminus {\rm diag} ,
\quad B(z):= (\pi(z),\pi(S_g(z))).
\] 
\end{corr}
\begin{proof}
$B$ is the restriction of $\Psi$ to $\{(\tau_+(z),z):z\in \pl_-S^*M\}$.   
\end{proof}

Let $d_g:M\times M\setminus {\rm diag} \to (0,\infty)$ denote the distance 
function in the metric $g$.  If $(M,g)$ has no conjugate points, then $d_g$
is smooth.  Let $\rho$ be a smooth defining function for $\pl \bbar{M}$,
and define 
\[\widetilde{d}_g(p,q):=d_g(p,q)+\log\rho(p) +\log\rho(q).\]    
\begin{prop}\label{dtsmooth}
If $(M,g)$ is simple, then $\widetilde{d}_g$ extends
smoothly to $\bbar{M}\times \bbar{M}\setminus {\rm diag}$.   
\end{prop}
\begin{proof}
It suffices to prove the result for geodesic defining functions.  
First use Proposition~\ref{Ediff}.  Given $p$, $q\in M$, we can uniquely
write $p=\pi(z)$, $q=\pi(\bbar{\varphi}_\tau(z))$ for $z\in S^*M$
and $0<\tau<\tau_+(z)$.  We need to show that $\til{d}_g(p,q)$ extends
smoothly to $\mc{T}^+\bbar{S^*M}$ as a function of $(\tau,z)$.  

Now $d_g(p,q)$ is the elapsed time $t(\tau,z)$ for the  
geodesic segment joining $p$ to $q$, given by \eqref{reparam}.   As in
\eqref{rhovphitau},   
\eqref{taupm}, for $z\in \bbar{S^*M}$ and $s$ near $\tau_+(z)$ we have 
$\rho(\bbar{\varphi}_s(z))=(\tau_+(z)-s)A_+(s,z)$ for $A_+$ smooth
satisfying $A_+(\tau_+(z),z)=1$, and for $z\in \bbar{S^*M}$ and $s$ near
$-\tau_-(z)$ we have $\rho(\bbar{\varphi}_s(z))=(s+\tau_-(z))A_-(s,z)$ for
$A_-$ smooth satisfying $A_-(-\tau_-(z),z)=1$.  Consequently 
\[
\rho(\bbar{\varphi}_s(z))^{-1}=\big(s+\tau_-(z)\big)^{-1}
+\big(\tau_+(z)-s\big)^{-1}+B(s,z)
\]
with $B(s,z)$ smooth for $z\in \bbar{S^*M}$, $s\in [-\tau_-(z),\tau_+(z)]$.
Carrying out the integration in \eqref{reparam}, one deduces that 
$t(\tau,z)+\log(\tau_+(z)-\tau)+\log\tau_-(z)$
extends smoothly to $\mc{T}^+\bbar{S^*M}$.  Since  
$\log\rho(\bbar{\varphi}_\tau(z))-\log(\tau_+(z)-\tau)$ and 
$\log\rho(z)-\log\tau_-(z)$ both extend smoothly to 
$\mc{T}^+\bbar{S^*M}$, the result follows. 
\end{proof}
\begin{rem}
Theorem 1.2 of \cite{SaWa1} asserts that the conclusion of
Proposition~\ref{dtsmooth} holds under the assumption that $(M,g)$ is
geodesically convex, which is equivalent to assuming that it is
non-trapping with no conjugate points.  However, the proof appears to be
incomplete.  Our proof above uses the additional assumption that there are 
no conjugate points at infinity.
\end{rem}

\begin{rem}
A study of $d_g$ has been carried out for certain perturbations of
hyperbolic space in \cite{MSV}.  This has been extended to non-trapping
asymptotically hyperbolic manifolds in a neighborhood of the boundary
diagonal by Chen-Hassell \cite{ChHa} and Sa Barreto-Wang \cite{SaWa1},
\cite{SaWa2}:  the function $\beta^*d_g+\log\rho_L +\log\rho_R$ extends to
a smooth function on $\mathcal{U}\setminus {\rm diag}_0$, where
$\mathcal{U}$ is a neighborhood of the front face in the Mazzeo-Melrose
stretched product space $\bbar{M}\x_0\bbar{M}$, ${\rm diag}_0$ denotes the
closure of the lift of the interior diagonal, $\rho_R$ and $\rho_L$ are
defining functions for the right and left faces in $\bbar{M}\x_0\bbar{M}$,
and $\beta$ denotes the blow-down map.  
Combining this with Proposition~\ref{dtsmooth}, it follows that 
$\beta^*d_g+\log\rho_L +\log\rho_R\in
C^\infty(\bbar{M}\x_0\bbar{M}\setminus {\rm diag}_0)$ for simple
asymptotically hyperbolic metrics.  The analysis in these papers of the 
short geodesics is closely related to Lemma~\ref{smallgeo} above.    
\end{rem}

For a simple asymptotically hyperbolic metric and a choice of defining 
function $\rho$, we define the \emph{renormalized boundary distance}
$d^R_g\in C^\infty(\pl \bbar{M}\times \pl \bbar{M}\setminus {\rm diag})$ by     
\begin{equation}\label{defdR}
d^R_g:= \til{d}_g|_{\pl \bbar{M}\x \pl \bbar{M}\setminus {\rm diag}}.   
\end{equation}
The realization \eqref{Lg} of $L_g$ shows that 
$d^R_g(p,q)=L_g(B^{-1}(p,q))$
with $B$ defined in Corollary~\ref{mapB}.  Either using \eqref{dependLg} or
directly from the definition, it follows that 
if $\widehat{\rho}=\rho e^{\omega}$  
is another choice of boundary defining function (with $\omega\in
C^\infty(\bbar{M})$), and if $\widehat{d}_g^R(p,q)$ denotes the  
renormalized distance associated to $\widehat{\rho}$, then
\[
\widehat{d}^R_g(p,q)-d^R_g(p,q)=\omega(p)+\omega(q),\qquad p,q\in \pl\bbar{M}. 
\]

The renormalized distance can be defined assuming only that 
there is a unique geodesic joining any two points of $\pl\bbar{M}$.  But if
$(M,g)$ is not simple, the map $B$ in 
Corollary~\ref{mapB} is not a local diffeomorphism,  
and we do not how to prove Proposition~\ref{equivalencedglens} below.     

\begin{rem}
Recall that a Busemann function for a point $p\in \pl\bbar{M}$  
is defined (typically for a Hadamard manifold) as follows.  Choose a
geodesic $\gamma(t)$ for which $\lim_{t\to \infty} \gamma(t)=p$.  The
Busemann function associated to $\gamma$ is the function $B_\gamma:M\to\rr$ 
defined by $B_\gamma(q)=\lim_{t\to \infty}\big(d_g(q,\gamma(t))-t\big)$.   
Observe that it follows from Proposition~\ref{dtsmooth} that the function  
$d^1_g(p,q)=d_g(p,q)+\log\rho(p)=\til{d}_g(p,q)-\log\rho(q)$ 
extends smoothly to $\bbar{M}\x M\setminus {\rm diag}$.  If $p\in
\pl\bbar{M}$, the function $q\to d^1_g(p,q)$ on $M$ is a Busemann function
for $p$, depending on the choice of defining function $\rho$. 
In fact, it is clear that if $\gamma$ is a geodesic and $\rho$ is any 
defining function such that $\lim_{t\to\infty} \rho(\gamma(t))/e^{-t}=1$,
then $B_\gamma(q)=d^1_g(p,q)$.  In particular, it is a consequence of  
Proposition~\ref{dtsmooth} that on a simple asymptotically hyperbolic
manifold, any Busemann function is in $C^\infty(M)$.     
\end{rem}

In the next two propositions we fix a representative $h$ for the conformal 
infinity of a simple asymptotically hyperbolic manifold $(M,g)$, thus 
determining a geodesic defining function $\rho$.   
$d^R_g$ and $L_g$ will denote the corresponding renormalized boundary
distance and renormalized length function.  
The product identification associated to $h$ induces the identification 
\eqref{ident} of each of $\pl_\pm S^*M$ with $T^*\pl\bbar{M}$.  We thereby
view $L_g$ as defined on $T^*\pl\bbar{M}$, and $S_g$ as mapping 
$T^*\pl\bbar{M}$ to itself. 

\begin{prop}\label{scatpoints}
Let $(M,g)$ be a simple asymptotically hyperbolic manifold and let $h$ be a
representative metric for the conformal infinity.   If $p$, $q\in
\pl\bbar{M}$, $p\neq q$, then 
\[
S_g\big(p,-d_p(d^R_g(p,q))\big)=\big(q,d_q(d^R_g(p,q))\big).  
\]
Here $d^R_g(p,q)$ is the renormalized distance function determined by $h$, 
$d_p(d^R_g(p,q))\in T^*_p\pl\bbar{M}$ denotes its exterior derivative with 
respect to $p$, and $d_q(d^R_g(p,q))\in T^*_q\pl\bbar{M}$ its exterior 
derivative with respect to $q$. 
\end{prop}
\begin{proof}
Define $d^1_g\in C^\infty(\bbar{M}\times M\setminus {\rm diag})$ by 
\[
d^1_g(p',q')=d_g(p',q')+\log\rho(p')=-\log\rho(q')+\til{d}_g(p',q').  
\]
If $p',q'\in M$, $p'\neq q'$, then $\grad_{q'} d_g^1(p',q')=\grad_{q'}
d_g(p',q')$ is the 
unit tangent vector at $q'$ to the geodesic joining $p'$ and $q'$, oriented
to point away from $p'$.  Here $\grad_{q'}$ denotes the gradient with
respect to $g$ in the second argument.  Since $(M,g)$ is simple and 
$d^1_g(p',q')$ is smooth in $p'$ up to $\pl\bbar{M}$, for $p\in
\pl\bbar{M}$ we can let $p'\to p$ along the geodesic joining $p$ to $q'$ to
deduce that 
$\grad_{q'} d^1_g(p,q')$ is the unit tangent vector at $q'$ to the geodesic
joining $p$ and $q'$, oriented to point away from $p$.  The corresponding
dual element of $S_{q'}^*M$ is $d_{q'}(d^1_g(p,q'))$.    

Now fix $p,q\in \pl\bbar{M}$, $p\neq q$, and consider the asymptotics of
$d_{q'}(d^1_g(p,q'))$ as $q'$ approaches $q$ along the geodesic joining $p$ 
to $q$.  Write $q'=(\rho,y)$ in the boundary identification
induced by $h$.  Recalling \eqref{defdR}, we have 
\[
\begin{split}
d_{q'}(d^1_g(p,q'))&=d_{q'}(-\log\rho+\til{d}_g(p,q'))\\
&=-\rho^{-1}d\rho+d_\rho(\til{d}_g(p,q'))
+d_y(\til{d}_g(p,q'))\\
&=-\rho^{-1}d\rho +\mc{O}(1)d\rho
+d_q(d^R_g(p,q))+\mc{O}(\rho)dy.
\end{split}
\]
According to the identification \eqref{ident}, the limiting point in
$\pl_+S^*M\cong T^*\pl\bbar{M}$ is therefore $(q,d_q(d^R_g(p,q)))$, as 
claimed.    

The same argument interchanging the roles of $p$ and $q$ shows that the
beginning point in $\pl_-S^*M\cong T^*\pl\bbar{M}$ for the geodesic from
$p$ to $q$ is $(p,-d_p(d^R_g(p,q)))$.  
\end{proof}

The relation $d^R_g(p,q)=L_g(B^{-1}(p,q))$ and the definition of $B$ in
terms of $S_g$ in Corollary~\ref{mapB} show that the pair $(L_g,S_g)$ 
determines the renormalized length $d^R_g$.  We    
conclude by showing that the converse is true in the following sense:    
\begin{prop}\label{equivalencedglens}
Let $\bbar{M}$ be a compact connected manifold-with-boundary and let $g_1$
and $g_2$ be simple asymptotically hyperbolic metrics on $M$.  Let $h_1$
and $h_2$ be representative metrics 
in the respective conformal infinities.  If $d^R_{g_1}=d^R_{g_2}$, then 
$L_{g_1}=L_{g_2}$ and $S_{g_1}=S_{g_2}$.  
\end{prop}  
\begin{proof}
Given $(p,\eta)\in T^*\pl\bbar{M}$, let $q=\pi (S_{g_1}(p,\eta))$ be the 
ending point of the geodesic for $g_1$ starting from $p$ with initial 
direction $\eta$.  By Proposition~\ref{uniquegeodesics}, there is a unique 
geodesic for $g_2$ starting at $p$ and ending at $q$.
Proposition~\ref{scatpoints} shows that the starting and 
ending directions for a geodesic are determined by the endpoints and 
the renormalized distance function $d^R_g$.  Since $d_{g_1}^R=d^R_{g_2}$,
one concludes that $S_{g_1}=S_{g_2}$.  Since  
$L_g(p,\eta)=d^R_g(B(p,\eta))$ and $B$ is determined by $S_g$, it 
follows that also $L_{g_1}=L_{g_2}$. 
\end{proof}


\begin{thebibliography}{CLMS}\frenchspacing 
\bibitem[AlMa]{AlMa} S. Alexakis, R. Mazzeo, \emph{Renormalized area and
  properly embedded minimal surfaces in hyperbolic 3-manifolds}, 
  Comm. Math. Phys. \textbf{297} (2010), 621--651. 

\bibitem[AnRo]{AnRo} Yu. Anikonov, V. Romanov, \emph{On uniqueness of
  determination of a form of first degree by its integrals along
  geodesics}, J. Inverse Ill-Posed Probl. \textbf{5} (1997), 467--480. 

\bibitem[BeCa]{BeCa}  C.A. Berenstein, E. Casadio Tarabusi,
  \emph{Inversion formulas for the $k$-dimensional Radon transform in real 
    hyperbolic spaces}, Duke Math. J. \textbf{62} (1991), 613--631.  

\bibitem[Be]{Be} A.L. Besse, {\it Einstein Manifolds}, Springer-Verlag, 
  1987.   

\bibitem[Cr]{Cr} C. Croke, \emph{Rigidity theorems in Riemannian geometry}, 
Chapter in Geometric Methods in Inverse Problems and PDE Control, C. Croke,
I. Lasiecka, G. Uhlmann, and M. Vogelius eds., Springer 2004. 

\bibitem[ChHa]{ChHa} X. Chen, A. Hassell, \emph{Resolvent and spectral
  measure on non-trapping asymptotically hyperbolic manifolds I: resolvent
  construction at high energy}, Comm. PDE \textbf{41} (2016), 515--578.  

\bibitem[CoLe]{CL} E. Coddington, N. Levinson, {\it Theory of Ordinary
  Differential Equations}, McGraw-Hill, 1955.
  
\bibitem[CLMS]{CLMS}  B. Czech, L. Lamprou, S. McCandlish, J. Sully,
  \emph{Integral geometry and holography}, JHEP \textbf{10} (2015), 175.  

\bibitem[DyGu]{DyGu} S. Dyatlov, C. Guillarmou, \emph{Pollicott-Ruelle
  resonances for open systems}, Annales IHP \textbf{17} (2016), 3089--3146.    

\bibitem[Eb1]{Eb1} P. Eberlein, \emph{Geodesic flow in certain manifolds
without conjugate points}, Trans. Amer. Math. Soc. \textbf{167} (1972), 
151--170.

\bibitem[Eb2]{Eb2} P. Eberlein, \emph{When is a geodesic flow of Anosov
  type? I} J. Diff. Geom. \textbf{8} (1973), 437--463.  

\bibitem[GHL]{GHL} S. Gallot, D. Hulin, J. Lafontaine, \emph{Riemannian 
  Geometry}, Universitext 3rd edition, Springer, 1987.

\bibitem[Gra]{Gra} C.R. Graham, 
\emph{Volume and area renormalizations for conformally compact Einstein
metrics},
Rend.\ Circ.\ Mat.\ Palermo Ser.\ II \textbf{63} (2000), Suppl., 31--42.

\bibitem[GrLe]{GL} C.R. Graham, J.M. Lee, \emph{Einstein metrics with
  prescribed conformal infinity on the ball}, Adv. Math. {\bf 87} (1991),
  186--225.  

\bibitem[Gre]{Gre} L. Green, {\it A theorem of E. Hopf},
  Mich. Math. J. \textbf{5} (1958), 31--34.  

\bibitem[GKM]{GKM} D. Gromoll, W. Klingenberg, W. Meyer, {\it Riemannsche
  Geometrie im Gro\ss en}, Zweite Auflage, Springer, 1975.  

\bibitem[Gu]{Gu} C. Guillarmou, \emph{Lens rigidity for manifolds with
  hyperbolic trapped set}, J. Amer. Math. Soc. \textbf{30} (2017), 561--599. 

\bibitem[GMS]{GMS} C. Guillarmou, S. Moroianu, J-M. Schlenker, \emph{The
  renormalized volume and uniformization of conformal structures},
  Journ. Inst. Math. Jussieu, https://doi.org/10.1017/S1474748016000244 

\bibitem[HMS]{HMS} K. Heil, A. Moroianu, U. Semmelmann, \emph{Killing and
  Conformal Killing tensors}, J. Geom. Phys. \textbf{106} (2016),
  383--400. 

\bibitem[He1]{He1} S. Helgason, \emph{The totally-geodesic Radon transform on
  constant curvature spaces.} Integral geometry and tomography (Arcata, CA,
  1989), 141--149, Contemp. Math., 113, Amer. Math. Soc., 1990.  

\bibitem[He2]{He2} S. Helgason. \emph{Geometric analysis on symmetric spaces}, 
volume 39 of Mathematical Surveys and Monographs. Amer. Math. Soc., 1994.  

\bibitem[HoUh]{HoUh} S. Holman, G. Uhlmann, \emph{Microlocal analysis of
  the geodesic X-ray transform with conjugate points.},  arXiv:1502.06545,
  to appear, J. Diff. Geom. 

\bibitem[Iv]{Iv} S. Ivanov, \emph{Volume comparison via boundary distances}
Proceedings of the International Congress of Mathematicians, 
vol. II , 769--784, New Delhi, 2010.

\bibitem[Kl1]{Kl1} W. Klingenberg, \emph{Manifolds with geodesic flow of
  Anosov type}, Ann. Math. \textbf{99} (1974), 1-13. 	  

\bibitem[Kl2]{Kl2} W. Klingenberg, \emph{Riemannian Geometry}, 
De Gruyter, 1995.

\bibitem[Kn]{Kn} G. Knieper, {\it A note on Anosov flows of non-compact
  Riemannian manifolds}, in preparation.  

\bibitem[KN]{KN} S. Kobayashi, K. Nomizu, {\it Foundations of
  Differential Geometry, I}, Interscience, 1963.  

\bibitem[LSU]{LSU} M. Lassas, V. Sharafutdinov, G. Uhlmann,
\emph{Semiglobal boundary rigidity for Riemannian metrics} 
Math. Ann.  \textbf{325} (2003), 767--793.

\bibitem[Le]{Le} J. Lehtonen, \emph{The geodesic ray transform on
  two-dimensional Cartan-Hadamard manifolds},  arXiv:1612.04800. 

\bibitem[LRS]{LRS} J. Lehtnonen, J. Railo, M. Salo, \emph{Tensor tomography  
  on Cartan-Hadamard manifolds}, arXiv:1705.10126.  

\bibitem[Ma]{Ma} R. Mazzeo, {\it Hodge Cohomology of Negatively Curved
  Manifolds}, M.I.T. Ph.D. dissertation, 1986.   

\bibitem[MaMe]{MaMe} R.R.~Mazzeo, R.B.~Melrose,
{\em Meromorphic extension of the resolvent on complete spaces
with asymptotically constant negative curvature},
J. Funct. Anal. \textbf{75} (1987), 260--310.

\bibitem[Me]{Me} R.B. Melrose, \emph{The Atiyah-Patodi-Singer index
theorem}, AK Peters, 1993.   

\bibitem[MSV]{MSV}  R.B. Melrose, A. Sa Barreto, A. Vasy, \emph{Analytic
  continuation and semiclassical resolvent estimates on asymptotically
  hyperbolic spaces}, Comm. PDE \textbf{39} (2014), 452--511. 

\bibitem[Mi]{Mi} R. Michel, \emph{Sur la rigidit\'e impos\'ee par la
  longueur des g\'eod\'esiques} Invent. Math. \textbf{65} (1981/82),
  71--83. 

\bibitem[MSU]{MSU} F. Monard, P. Stefanov, G. Uhlmann, 
\emph{The geodesic ray transform on Riemannian surfaces with conjugate
points} Comm. Math. Phys. \textbf{337} (2015), 1491--1513.

\bibitem[Mu]{Mu} R.G. Mukhometov, \emph{On a problem of reconstructing
  Riemannian metrics}, Siberian Math. J. \textbf{22} (1982), 420--433.  
  
\bibitem[Pa]{Pa} G.P. Paternain, \emph{Geodesic flows}, Progress in
  Mathematics 180, Birkhauser, 1999.
		
\bibitem[PSU1]{PSU1} G.P. Paternain, M. Salo, G. Uhlmann, \emph{Tensor 
  tomography on surfaces}, Invent. Math. \textbf{193} (2013), 229--247.   
		
\bibitem[PSU2]{PSU2} G.P. Paternain, M. Salo, G. Uhlmann, \emph{Invariant
  distributions, Beurling transforms and tensor tomography in higher
  dimensions}, Math. Ann. \textbf{363} (2015), 305--362. 

\bibitem[PeSh]{PeSh} L.N. Pestov, V.A. Sharafutdinov, \emph{Integral
  geometry of tensor fields on a manifold of negative curvature}, Siberian 
  Math. J. \textbf{29} (1988), 427--441. 

\bibitem[PoRa]{PoRa} M. Porrati, R. Rabadan, \emph{Boundary rigidity and 
  holography}, JHEP \textbf{01} (2004), 034. 

\bibitem[SaWa1]{SaWa1} A. Sa Barreto, Y. Wang, \emph{The scattering
  relation on asymptotically hyperbolic manifolds}, arXiv:1410.6842. 

\bibitem[SaWa2]{SaWa2} A. Sa Barreto, Y. Wang, \emph{The scattering operator
  on asymptotically hyperbolic manifolds}, arXiv:1609.02332, to appear,
  J. Spectr. Theory.

\bibitem[Sh1]{Sh1} V. A. Sharafutdinov,
\emph{Integral geometry of tensor fields}, Inverse and Ill-posed Problems
Series, VSP, 1994. 

\bibitem[Sh2]{Sh2} V. Sharafutdinov, \emph{Variations of
  Dirichlet-to-Neumann map and deformation boundary rigidity of simple
  2-manifolds},  J. Geom. Anal. \textbf{17} (2007), 147--187.

\bibitem[StUh1]{StUh1} P. Stefanov, G. Uhlmann, \emph{Boundary
rigidity and stability for generic simple metrics.} 
J. Amer. Math. Soc.  \textbf{18} (2005), 975--1003.

\bibitem[StUh2]{StUh2} 
P. Stefanov, G. Uhlmann. \emph{Boundary and lens rigidity, tensor
tomography and analytic microlocal analysis.} In Algebraic Analysis of
Differential Equations. Springer, 2008. 

\bibitem[StUh3]{StUh3} P. Stefanov, G. Uhlmann, \emph{Local lens rigidity
  with incomplete data for a class of non-simple Riemannian manifolds}, 
J. Diff. Geom.  \textbf{82} (2009), 383--409.

\bibitem[StUh4]{StUh4} P. Stefanov, G. Uhlmann, \emph{The geodesic X-ray
  transform with fold caustics} Anal. PDE \textbf{5} (2012), 219--260.

\bibitem[SUV1]{SUV} P. Stefanov, G. Uhlmann, A. Vasy, \emph{Boundary
  rigidity with partial data}, J. Amer. Math. Soc. \textbf{29} (2016),
  299--332.    

\bibitem[SUV2]{SUV2} P. Stefanov, G. Uhlmann, A. Vasy, \emph{Inverting the
  local geodesic X-ray transform on tensors}, J. d'Analyse Math., to
  appear, arXiv:1410.5145.  

\bibitem[SUV3]{SUV3} P. Stefanov, G. Uhlmann, A. Vasy, \emph{Local and
  global boundary rigidity and the geodesic X-ray transform in the normal
  gauge},  arXiv:1702.03638. 

\bibitem[UhVa]{UhVa} G. Uhlmann, A. Vasy, \emph{The inverse problem for the
  local geodesic ray transform},  Invent. Math. \textbf{205} (2016),
  83--120. 

\bibitem[Va]{Va} J. Vargo, {\it A proof of lens rigidity in the category of
  analytic metrics}, Math. Res. Lett. \textbf{16} (2009), 1057--1069. 

\end{thebibliography}
\end{document}